\documentclass[a4paper,oneside,12pt]{amsart}
\usepackage{placeins} 

\usepackage{amsmath,amssymb,amsthm}
\usepackage{epsf}
\usepackage{epsfig}
\usepackage[english]{babel}

\addto\captionsenglish{}

\usepackage{mathrsfs}   

\usepackage{fancyhdr}           
\newcommand{\RR}{\mathbb R}
\newcommand{\CC}{\mathbb C}

\newcommand{\DD}{\mathbb D}

\renewcommand{\Re}{\mathop{\rm Re}\nolimits}
\renewcommand{\Im}{\mathop{\rm Im}\nolimits}
\newcommand{\trace}{\mathop{\mathrm{trace}}\nolimits}
\newcommand{\arctanh}{\mathop{\mathrm{arctanh}}\nolimits}

\newcommand{\jj}{\mathrm{j}}
\newcommand{\qq}{\mathrm{q}}

\newcommand{\manifold}[1]{\mathcal{#1}}
\newcommand{\M}{\manifold{M}}
\newcommand{\D}{\manifold{D}}

\newcommand{\vect}[1]{\mathrm{#1}} 
\newcommand{\x}{\vect{x}}
\newcommand{\y}{\vect{y}}
\newcommand{\va}{\vect{a}}
\newcommand{\vb}{\vect{b}}
\newcommand{\vc}{\vect{c}}
\newcommand{\vd}{\vect{d}}
\newcommand{\vX}{\vect{X}}
\newcommand{\vY}{\vect{Y}}
\newcommand{\vZ}{\vect{Z}}
\newcommand{\vW}{\vect{W}}
\newcommand{\vH}{\vect{H}}
\newcommand{\n}{\vect{n}}
\newcommand{\m}{\vect{m}}

\newcommand{\e}{\mathrm{e}}

\newtheorem{theorem}{Theorem}[section]
\newtheorem{prop}[theorem]{Proposition}
\newtheorem{lem}[theorem]{Lemma}

\theoremstyle{remark}
\newtheorem{remark}{Remark}[section]
\theoremstyle{definition}
\newtheorem{dfn}{Definition}[section]

\newcommand{\ds}{\displaystyle}

\makeatletter
\@addtoreset{equation}{section}
\makeatother

\renewcommand{\baselinestretch}{1.10}

\begin{document}

\title
{Canonical Coordinates and Natural Equations for Minimal Time\,-like Surfaces in $\RR^4_2$}%

\author{Georgi Ganchev and Krasimir Kanchev}

\address{Bulgarian Academy of Sciences, Institute of Mathematics and Informatics,
Acad. G. Bonchev Str. bl. 8, 1113 Sofia, Bulgaria}
\email{ganchev@math.bas.bg}%

\address {Department of Mathematics and Informatics, Todor Kableshkov University of Transport,
158 Geo Milev Str., 1574 Sofia, Bulgaria}%
\email{kbkanchev@yahoo.com}%



\begin{abstract}
We apply the complex analysis over the double numbers $\DD$ to study the minimal time-like surfaces in $\RR^4_2$.
A minimal time-like surface which is free of degenerate points is said to be of general type. We divide the minimal 
time-like surfaces of general type into three types and prove that these surfaces admit special geometric (canonical)
parameters. Then the geometry of the minimal time-like surfaces of general type is determined by the Gauss curvature 
$K$ and the curvature of the normal connection $\varkappa$, satisfying the system of natural equations for these surfaces. 
We prove the following: If $(K, \varkappa),\, K^2- \varkappa^2 > 0 $ is a solution to the system of natural equations, then 
there exists exactly one minimal time-like surface of the first type and exactly one minimal time-like surface of 
the second type with invariants $(K, \varkappa)$;\, if $(K, \varkappa),\, K^2- \varkappa^2 < 0 $ is a solution to the 
system of natural equations, then there exists exactly one minimal time-like surface of the third type with
invariants $(K, \varkappa)$.
 
\end{abstract}

\maketitle

\thispagestyle{empty}

\tableofcontents



\section{Introduction}\label{sec_introduction}
The most useful parameters in the theory of surfaces in Euclidean space $\RR^3$ are the principal and isothermal 
parameters, whose geometric nature allows a natural interpretation of the results of analytical calculations.
Solving some problems for the surfaces in $\RR^3$ requires further specialization of these parameters. For the 
class of Weingarten surfaces in $\RR^3$ it was shown in \cite{GM} that the principal parameters can be specialized 
so that the coefficients of the first fundamental form are expressed through the invariants of the surface. These 
parameters are determined up to the orientation and renumbering of the parametric lines. That is why these 
coordinates were called \emph{canonical} coordinates. Using canonical parameters, 
it was proved in \cite{GM} that the local geometry of Weingarten surfaces is determined by one function satisfying 
one PDE. This means that the number of determining invariants and the number of integrability conditions are 
minimal. It is worth mentioning that the minimal surfaces in $\RR^3$, the maximal space-like surfaces in
$\RR^3_1$ and the minimal time-like surfaces in $\RR^3_1$ admit locally canonical parameters which are principal 
and isothermal simultaneously.

It is natural to expect that minimal surfaces of co-dimension two in $\RR^4$, $\RR^4_1$ and $\RR^4_2$ admit locally 
geometrically determined special isothermal parameters, having the property that the coefficients of the first 
fundamental form are expressed through the invariants of the surface. Special geometric coordinates on minimal 
surfaces in $\RR^4$ were introduced in \cite{Itoh}. Further, these parameters were used in \cite{T-G-1} to show that
a minimal surface in $\RR^4$ is determined up to a motion by two invariant functions satisfying two PDEs. We note that
the number of the invariant functions determining the surfaces and the number of the integrability conditions are 
minimal. On the base of the canonical coordinates, in \cite{G-K-1}, we solved explicitly the system of background 
PDEs of minimal surfaces in $\RR^4$ in terms of two holomorphic functions in $\CC$.

Special isothermal coordinates on maximal space-like surfaces in Minkowski space-time $\RR^4_1$ were used in \cite{A-P-1} 
to prove that the local geometry of these surfaces is determined by two invariant functions satisfying two PDEs. 
Using these canonical coordinates we solved in \cite{G-K-2} explicitly the system of natural (background) PDEs 
of maximal space-like surfaces in $\RR^4_1$.

Minimal space-like surfaces in the pseudo-Euclidean 4-space with neutral metric $\RR^4_2$ were introduced and used in 
\cite{S-1} to prove similar results. Another approach to the canonical coordinates on maximal space-like surfaces 
in $\RR^4_2$ we used in \cite{G-K-arXiv-3}.  On the base of these canonical coordinates we obtained canonical 
Weierstrass representation for the minimal space-like surfaces in $\RR^4_2$ in terms of two holomorphic functions in 
$\CC$. In \cite{G-K-3} we solved explicitly the system of natural PDEs of minimal space-like surfaces in $\RR^4_2$.

Minimal time-like surfaces in Minkowski space-time were considered in \cite{G-M-1}. For these surfaces it was proved 
that they admit locally canonical parameters and their geometry is determined by two invariant functions, satisfying 
a system of two natural PDEs.

In \cite{S-2} it was found the background system of natural equations of the minimal Lorentzian (time-like) surfaces in $\RR^4_2$
in terms of the Gauss curvature $K$ and the curvature of the normal connection $\varkappa$. Under the condition 
$K^2-\varkappa^2 > 0$ it was proved that there are two families of minimal time-like surfaces, determined by any 
solution $(K, \varkappa)$ to the system of natural equations. Since the considerations in \cite{S-2} are in arbitrary isothermal 
coordinates, then $(K, \varkappa)$ determine the minimal time-like surface up to a transformation preserving both curvatures. 
 
Minimal Lorentzian (time-like) surfaces in $\RR^4_2$, whose Gauss curvature $K$ and curvature of the normal connection 
$\varkappa$ satisfy the inequality $K^2-\varkappa^2 > 0$, were studied in \cite{M-A-1} from  the point of view of canonical coordinates.  

In this paper we study the general class of time-like surfaces in the four-dimensional pseudo-Euclidean space $\RR^4_2$
with neutral metric, i.e. minimal time-like surfaces satisfying the conditions $K^2-\varkappa^2 > 0$ or $K^2-\varkappa^2 < 0$,
where $K$ is the Gauss curvature and $\varkappa$ is the curvature of the normal connection of the surface.

The main feature of our investigations is the application of the analysis over the double numbers in $\DD$ as a convenient tool especially in the geometry of minimal time-like surfaces. To this end, we consider any  time-like surface $\M=(\D ,\x(u,v))$ parametrized by isothermal parameters $(u,v) \in \D \subset \RR^2$. Then we introduce the "complex" variable $t=u+\jj v \in \DD$, where $\jj^2=1$\,. Thus any function on $\M$ can be considered as a function of $t$. Considering the natural extension $\DD^4_2$ 
over $\DD$ of $\RR^4_2$, in Section \ref{sect_Phi_R42-tl} we introduce on any time-like surface $\M$ the $\DD^4_2$-valued 
function $\Phi$ by equality \eqref{Phi_def-tl}. 

In Section \ref{sect_Phi_Psi_R42-tl} we characterize minimal time-like surfaces in $\mathbb R^4_2$ via the function $\Phi$. 
Theorem \ref{Min_x_Phi-thm-tl} states that a time-like surface $\M$ is minimal if and only if the function $\Phi$ is 
holomorphic in $\DD$. Theorem \ref{Min_x_Psi-thm-tl} gives the representation of the minimal surface $\M$ 
through the primitive function  $\Psi $ of $\Phi$.

In Section \ref{sect_K_kappa-Phi_R42-tl} we obtain formulas \eqref{K_kappa_Phi_R42-tl}
for the Gauss curvature $K$ and the curvature of the normal connection $\varkappa$
of the minimal time-like surface expressed by the function $\Phi$. 

In Section \ref{sect_can-def_R42-tl} we introduce degenerate points on a minimal time-like surface by Definition 
\ref{DegP-def-tl}. In Theorem \ref{DegP_kind123-K_kappa-tl} we prove that the degenerate points are characterized by 
the equality $K^2-\varkappa^2=0$. Then the minimal time-like surfaces in $\RR^4_2$, which are free of degenerate points 
are said to be of \textit{general type}. Hence, the minimal time-like surfaces of general type in $\RR^4_2$ are 
characterized by the inequality $K^2-\varkappa ^2 \neq 0$. In Definition \ref{Min_Surf_kind123-def-tl} we divide the 
minimal time-like surfaces of general type into three types: the first, the second and the third type depending on  
the properties of ${\Phi'^\bot}^2$.
Canonical coordinates on any minimal time-like surface of general type are introduced in Definition \ref{Can-def_R42-tl}
by the condition:
\[
{\Phi'^\bot}^2=\varepsilon\,,
\]
where $\varepsilon = 1, \, \varepsilon = -1$ and $\varepsilon =\jj$ if the time-like surface is of the first, the second 
and the third type, respectively. 

In Theorem \ref{DegP_kind123-K_kappa-tl} we show that the inequality $K^2-\varkappa^2 > 0$ characterizes the  minimal 
time-like surfaces of the first or the second type, while the inequality $K^2-\varkappa^2 < 0$ characterizes the minimal 
time-like surfaces of the third type. 
In Proposition \ref{Can-sigma_R42-tl} we characterize the three types of canonical parameters by conditions for the second fundamental form of the surface.

In Theorem \ref{Can_Coord-exist_R42-tl} we prove that any minimal time-like surface of general type in $\RR^4_2$ 
admits locally canonical coordinates. Theorem \ref{Can_Coord-uniq_R42-tl} establishes the uniqueness properties of the 
canonical coordinates.

In  Theorem \ref{Thm-Nat_Eq_K_kappa_R42-tl} it is shown that the curvatures $K$ and $\varkappa$ satisfy the system of 
natural equations \ref{Nat_Eq_K_kappa_R42-tl} of the minimal time-like surfaces of general type in $\RR^4_2$. 

The basic theorems in this paper are Theorem \ref{Bone_Phi_bar_Phi_m1_m2_K_kappa_kind1_R42-tl}, Theorem 
\ref{Bone_Phi_bar_Phi_m1_m2_K_kappa_kind2_R42-tl} and Theorem \ref{Bone_Phi_bar_Phi_m1_m2_K_kappa_kind3_R42-tl}.
Theorem \ref{Bone_Phi_bar_Phi_m1_m2_K_kappa_kind1_R42-tl} states that any solution $(K, \varkappa)$ to the system
\ref{Nat_Eq_K_kappa_R42-tl} with the property $K^2 - \varkappa^2 > 0$ generates exactly one minimal time-like surface 
of the first type up to a motion; Theorem \ref{Bone_Phi_bar_Phi_m1_m2_K_kappa_kind2_R42-tl} states that any solution 
$(K, \varkappa)$ to the system \ref{Nat_Eq_K_kappa_R42-tl} with the property $K^2- \varkappa^2 > 0$ generates exactly 
one minimal time-like surface of the second type up to a motion; Theorem \ref{Bone_Phi_bar_Phi_m1_m2_K_kappa_kind3_R42-tl} 
states that any solution $(K, \varkappa)$ to the system \ref{Nat_Eq_K_kappa_R42-tl} with the property $K^2- \varkappa^2 < 0$ 
generates exactly one minimal time-like surface of the third type up to a motion. Theorem 
\ref{Bone_Phi_bar_Phi_m1_m2_K_kappa_kind3_R42-tl} solves the case, which is not treated in \cite{S-2} and \cite{M-A-1}.

An isometry between two minimal time-like surfaces preserving the curvature $\varkappa$ is called  a \textit{strong} 
isometry. In Theorem \ref{Thm-Str_Isom_Surf_R42-tl} we find the surfaces, strongly isometric to a given minimal 
time-like surface in the general case $K^2-\varkappa^2 \neq 0$. Similar question is considered in \cite{S-2} in the 
case $K^2-\varkappa^2 > 0$.

\section{Preliminaries}\label{sect_preliminaries_R42-tl}

 Let $\RR^4_2$ be the standard four-dimensional pseudo-Euclidean space with neutral metric. The indefinite scalar product in  
$\RR^4_2$ is given by the formula:
\begin{equation}\label{R^4_2-tl}
\va\cdot \vb=-a_1b_1+a_2b_2-a_3b_3+a_4b_4\,.
\end{equation}

 Let $\M_0$ be a two-dimensional manifold and $\x : \M_0 \to \RR^4_2$ be an immersion of $\M_0$ in $\RR^4_2$.
Then it is said briefly that $\M=(\M_0 ,\x)$ (or $\M$) is a regular surface in $\RR^4_2$. $T_p(\M)$ will
denote the tangential space of $\M$ at a point $p$, and through the differential $\x_*$ of $\x$, the space $T_p(\M)$ 
is identified with the corresponding subspace in $\RR^4_2$. $N_p(\M)$ will denote the normal space of $\M$ at the
point $p$, which appears to be the orthogonal complement of $T_p(\M)$ in $\RR^4_2$. The scalar product in $\RR^4_2$ 
induces a scalar product in $T_p(\M)$. If this scalar product is indefinite, then $\M$ is said to be a
\emph{time-like} surface in $\RR^4_2$. Then the induced scalar product onto the normal space $N_p(\M)$ is also 
indefinite. 

 Let $(u,v) \in \D \subset \RR^2$ be a pair of local coordinates (parameters) in a neighborhood of a point $p \in \M$.
The immersion $\x$ generates a vector function $\x (u,v) : \D \to \RR^4_2$. Since our 
considerations are local, we suppose that $\M$ is given by $(\D ,\x)$, where $\D \subset \RR^2$.

We use the classical denotations for the coefficients of the first fundamental form $\mathbf{I}$ of $\M$:
$E=\x_u^2$, $F=\x_u\cdot\x_v$ and $G=\x_v^2$. Then 
\[
\mathbf{I}=E\,du^2+2F\,dudv+G\,dv^2 .
\]

 It is well known that any point $p \in \M$ has a neighborhood in which one can introduce isothermal coordinates 
characterized by the conditions $E=-G$ and $F=0$. Further, we suppose that $(u,v)$ are isothermal local coordinates
on $\M$, such that $E<0$ and $G>0$.

\medskip

Studying time-like surfaces it is convenient to identify the coordinate plane $\RR^2$ with the plane of the 
\emph{double numbers} $\DD$, which are determined in the following way:
\[
\DD=\{t=u+\jj v : \  u,v \in \RR ,\ \jj^2=1 \}\,.
\] 
Along with the real coordinates $(u,v)$, we also consider the coordinate $t=u+\jj v$, $t \in \D\subset\DD$. 
In this way, all functions on $\M$ will also be considered as functions of the variable $t$.

 For any $t=u+\jj v \in \DD$ we denote by $|t|^2$ the square of its modulus (the amplitude), which is given by:
\[
|t|^2 = t\bar t = (u+\jj v)(u-\jj v) = u^2-v^2.
\]
We denote by $\DD_0$ the set of non-invertible elements in $\DD$, characterized in the following way:
\begin{equation}\label{D0}
\DD_0 = \{t\in\DD : \  |t|^2 = 0 \} = \{t\in\DD : \  \Re t = \pm \Im t \}\,.
\end{equation} 
We denote by $\DD_+$ the set of the "positive"\ elements in $\DD$:
\begin{equation}\label{D+}
\DD_+ = \{t\in\DD : \  |t|^2 > 0; \ \Re t > 0 \} = \{t\in\DD : \  \Re t \pm \Im t > 0 \}\,.
\end{equation}

In many cases, operating with the double numbers, it is convenient along with the basis $(1,\jj)$ to use the basis 
$(\qq,\bar\qq)$, where:
\begin{equation}\label{qq-def}
\qq = \frac{1-\jj}{2} \,; \qquad \bar\qq = \frac{1+\jj}{2}\,.
\end{equation}
The basis $(\qq,\bar\qq)$ is called \emph{diagonal basis} or \emph{null basis}. It follows in a straightforward way that:
\begin{equation}\label{qq-prop}
\qq^2 = \qq \,; \qquad \bar\qq^2 = \bar\qq \,; \qquad \qq \bar\qq = 0 \,.
\end{equation} 
Any double number $t\in\DD$ with respect to the basis $(\qq,\bar\qq)$ is represented in the following way:
\begin{equation}\label{t-qq}
t = u+\jj v = (u-v)\qq + (u+v)\bar\qq \,.
\end{equation}
Taking into account \eqref{qq-prop}, it follows that the addition and the multiplication with respect to the basis
$(\qq,\bar\qq)$, are carried out component-wise:
\begin{equation}\label{sum-prod-qq}
\begin{array}{rl}
(a_1\qq + b_1\bar\qq) + (a_2\qq + b_2\bar\qq) \!\!\! &= (a_1 + a_2)\qq + (b_1 + b_2)\bar\qq\,,\\
(a_1\qq + b_1\bar\qq)   (a_2\qq + b_2\bar\qq) \!\!\! &= (a_1   a_2)\qq + (b_1   b_2)\bar\qq\,.
\end{array}
\end{equation}
This means that $\DD$ as an algebra is isomorphic to two copies of $\RR:\DD=\RR\oplus\RR$.
It follows from \eqref{t-qq} that the sets $\DD_0$ and $\DD_+$ can be represented as follows:
\[
\DD_0 = \{t=a\qq+b\bar\qq \in \DD : \  a=0\  \text{or}\  b=0 \}\,.
\]
\[
\DD_+ = \{t=a\qq+b\bar\qq \in \DD : \  a>0\  \text{and}\  b>0 \}\,.
\]

 The trigonometric form of the double numbers is given as follows: 

If $t = u+\jj v \in \DD_+$ and $|t|^2=u^2-v^2=1$\,, then there exists a unique $\theta\in\RR$, such that 
$u=\cosh\theta$\: and\: $v=\sinh\theta$. Then $t=\cosh\theta+\jj\sinh\theta=\e^{\jj\, \theta}$. For any 
$t \in \DD_+$, putting $\rho=\sqrt{|t|^2}$, we obtain $t=\rho\,e^{\jj\,\theta}$. Finally, if $t$ is an
arbitrary double number $t \notin \DD_0$\,, then there exists a unique $\delta=\pm 1; \pm\jj$, such
that $\delta t \in \DD_+$. Therefore $t$ can be written in a unique way in the form:
\begin{equation}\label{t-exp}
t = \delta\rho\,e^{\jj\,\theta}\,;  \qquad\quad  \delta=\pm 1; \pm\jj\,, \quad  \rho>0\,, \quad  \theta\in\RR\,.
\end{equation}

\smallskip

 If $f:\D\to\DD$ is a differentiable function, then the differential of $f$ is represented by:
\[
df = \frac{\partial f}{\partial t} dt + \frac{\partial f}{\partial \bar t} d\bar t \,,
\]  
where $\frac{\partial }{\partial t}$ and $\frac{\partial }{\partial \bar t}$ by definition are given as follows:
\begin{equation}\label{delta-t-delta-uv}
\frac{\partial }{\partial t} = 
\frac{1}{2}\left(\frac{\partial }{\partial u} + \jj \frac{\partial }{\partial v}\right) ; \qquad
\frac{\partial }{\partial \bar t} = 
\frac{1}{2}\left(\frac{\partial }{\partial u} - \jj \frac{\partial }{\partial v}\right) .
\end{equation}
The function $f$ is said to be \emph{holomorphic}, if $\frac{\partial f}{\partial \bar t} = 0$ and
respectively \emph{anti-holomorphic}, if $\frac{\partial f}{\partial t} = 0$\,.
If $f=g+\jj h$, where $g$ is the "real"\ part and $h$ is the "imaginary"\ part,
then $f$ is holomorphic if and only if $g$ and $h$ satisfy the following conditions, analogous to the 
conditions of Cauchy-Riemann:
\begin{equation}\label{CR-DD}
h_u = g_v \,; \qquad h_v = g_u \,.
\end{equation}

 These equalities imply that a map in $\RR^2_1$ is conformal with respect to the indefinite metric,
if and only if it is given by a holomorphic or anti-holomorphic function in $\DD$. Further, it follows 
that the inverse function theorem has its analogue:
If $f$ is a holomorphic function, satisfying the condition $|f'|^2 \neq 0$\,, then there exists at least locally
a unique inverse holomorphic function. Especially, the $n$th root of a holomorphic function gives 
a holomorphic function defined and with values in $\DD_+$.

Foundations of algebra and analysis in $\DD$ can be found in \cite{A_F-1}, \cite{K-1} or \cite{M-R-1}.

\medskip

 Let $\DD^4_2$ denote the space $\DD^4$ endowed with the bilinear product $\va\cdot \vb$, 
which is a natural extension of the scalar product in $\RR^4_2$, given by \eqref{R^4_2-tl}.
Then the scalar square $\va^2\in\DD$ of $\va\in\DD^4_2$ is given by
\[\va^2=\va\cdot \va = -a_1^2+a_2^2-a_1^2+a_4^2\,.\]
The square of the norm $\|\va\|^2\in\RR$ of $\va\in\DD^4_2$ is the number
\[\|\va\|^2=\va\cdot\bar \va = -|a_1|^2+|a_2|^2-|a_1|^2+|a_4|^2 ,\]
which is not necessarily positive.   

 We use the standard imbedding of $\RR^4_2$ into $\DD^4_2$ and consider the "complexified" \,tangent space
$T_{p,\DD}(\M)$ of $\M$ at the point $p$ as a subspace of $\DD^4_2$, which is the linear span of $T_p(\M)$ 
in $\DD^4_2$. Analogously, we identify the "complexified"\, normal space $N_{p,\DD}(\M)$ of $\M$ with the
corresponding subspace of $\DD^4_2$, which is the linear span of $N_p(\M)$ in $\DD^4_2$.

Since $T_{p,\DD}(\M)$ and $N_{p,\DD}(\M)$ are generated by the real subspaces $T_p(\M)$ and $N_p(\M)$ respectively,
then they are closed with respect to the complex conjugation in $\DD^4_2$ and are mutually orthogonal. 
Therefore we have the following orthogonal decomposition:
\[\DD^4_2 = T_{p,\DD}(\M) \oplus N_{p,\DD}(\M)\,.\]

 For a given vector $\va \in \DD^4_2$ let $\va^\top$ denote the orthogonal projection of the vector $\va$
into the complexified tangent space  of $\M$. The orthogonal projection of $\va$ into the complexified normal 
space of ${\M}$ is denoted by $\va^\bot$. Then for any vector $\va$ we have:
\[\va=\va^\top + \va^\bot .\]

\medskip

  Let $\nabla$ be the canonical linear connection in $\RR^4_2$. For any tangent vector fields $\vX$, $\vY$ and normal 
vector field $\n$ the Gauss and Weingarten formulas are as follows: 
\[\nabla_\vX \vY = \nabla^T_\vX \vY + \sigma(\vX,\vY)\,,\]
\[\nabla_\vX \n = -A_\n(\vX) + \nabla^N_\vX \n \,, \] 
where: $\nabla^T$ is the Levi-Civita connection of $\M$,
$\sigma(\vX,\vY)$ is the second fundamental form of $\M$, $A_\n$ is the Weingarten map with respect to $n$ and
$\nabla^N_\vX \n$ is the normal connection of $\M$. The following equality holds:  
\[ A_\n\vX\cdot\vY = \sigma(\vX,\vY)\cdot \n \;. \]

 The curvature tensor $R$ of the Levi-Civita connection on $\M$ and the curvature tensor $R^N$ of the normal connection
on $\M$ are defined as follows:
\[ R(\vX,\vY)\vZ = \nabla^T_\vX \nabla^T_\vY \vZ - \nabla^T_\vY \nabla^T_\vX \vZ - \nabla^T_{[\vX,\vY]} \vZ\,,\]
\[ R^N(\vX,\vY)\n = \nabla^N_\vX \nabla^N_\vY \n - \nabla^N_\vY \nabla^N_\vX \n - \nabla^N_{[\vX,\vY]} \n \,.\]
The covariant derivative of $\sigma$ is given by the formula:
\[ (\overline\nabla_\vX \sigma)(\vY,\vZ) =
\nabla^N_\vX \sigma(\vY,\vZ)-\sigma(\nabla^T_\vX\vY,\vZ)-\sigma(\vY,\nabla^T_\vX\vZ)\,.\]

 The tensors $R$, $R^N$ and $\overline\nabla\sigma$ satisfy the fundamental equations in the theory of
Riemannian submanifolds:
 
\noindent  
\emph{the Gauss equation}: 
\begin{equation}\label{Gauss}
R(\vX,\vY)\vZ\cdot\vW = \sigma(\vX,\vW)\sigma(\vY,\vZ) - \sigma(\vX,\vZ)\sigma(\vY,\vW)\, ;
\end{equation}
 \emph{the Codazzi equation}: 
\begin{equation}\label{Codazzi}
(\overline\nabla_\vX \sigma)(\vY,\vZ) = (\overline\nabla_\vY \sigma)(\vX,\vZ)\; ; 
\end{equation}
\emph{the Ricci equation}:
\begin{equation}\label{Ricci}
R^N(\vX,\vY)\n\cdot\m = [A_\n,A_\m]\vX\cdot\vY\, .
\end{equation}
In the last equality $[A_\n,A_\m]$ denotes the commutator $A_\n A_\m - A_\m A_\n$ of $A_\n$ and $A_\m$.

 Basic invariants of any surface $\M$ in $\RR^4_2$ are the \emph{mean curvature} $\vH$,
the \emph{Gauss curvature} $K$ and the \emph{curvature of the normal connection (normal curvature)} $\varkappa$. 
Let $\vX_1$ and $\vX_2$ be an orthonormal frame field, tangent to $\M$, such that $\vX_1^2=-1$\,. 
Denote by $\n_1$ and $\n_2$ an orthonormal frame field, normal to $\M$, such that $\n_1^2=-1$ and at any point 
$p\in\M$, the quadruple $(\vX_1,\vX_2,\n_1,\n_2)$ is a right oriented orthonormal basis in $\RR^4_2$. 
Then $\vH$ is the vector-valued function with values in $N_p(\M)$ defined by:
\begin{equation}\label{H-def-tl}
\vH = \frac{1}{2}\trace\sigma = \frac{1}{2}(-\sigma(\vX_1,\vX_1)+\sigma(\vX_2,\vX_2)) \,.
\end{equation}
The Gauss curvature $K$ by definition is:
\begin{equation}\label{K-def-tl}
K = -R(\vX_1,\vX_2)\vX_2\cdot\vX_1 \,.
\end{equation}
The normal curvature $\varkappa$ is given by the equality:
\begin{equation}\label{kappa-def-tl}
\varkappa = R^N(\vX_1,\vX_2)\n_2\cdot\n_1 \,.
\end{equation}

 The Gauss equation \eqref{Gauss} implies that $K$ is expressed through $\sigma$ as follows:
\begin{equation}\label{K_sigma-tl}
K = -\sigma(\vX_1,\vX_1)\sigma(\vX_2,\vX_2) + \sigma^2(\vX_1,\vX_2)\,.
\end{equation}
 By virtue of the Ricci equation the normal curvature $\varkappa$ satisfies the following equality:
\begin{equation}\label{kappa_A-tl}
\varkappa = [A_{\n_2},A_{\n_1}] \vX_1\cdot \vX_2=
A_{\n_1}\vX_1\cdot A_{\n_2}\vX_2-A_{\n_2}\vX_1\cdot A_{\n_1}\vX_2 \,.
\end{equation}

\begin{dfn}\label{Min_Surf-def_R42-tl} 
 Any time-like surface in $\RR^4_2$ with $\vH=0$\, is said to be a \emph{minimal} time-like surface.
\end{dfn}


\section{Definition and basic properties of the $\DD^4_2$-valued vector function $\Phi$.}\label{sect_Phi_R42-tl}

Let $\M$ be a time-like surface in $\RR^4_2$, parametrized by isothermal coordinates $t=u+\jj v$.
The vector function $\Phi(t)$ with values in $\DD^4_2$ is defined by the equality:
\begin{equation}\label{Phi_def-tl}
\Phi(t)=2\frac{\partial\x}{\partial t}=\x_u+\jj\x_v \,.
\end{equation}

We use the function $\Phi$ in the present work as a basic analytic tool to study local properties of the
minimal time-like surfaces. First we establish the basic algebraic and analytic properties of the function
$\Phi$.

 From \eqref{Phi_def-tl} we have:
 \[ \Phi^2=(\x_u + \jj\x_v)^2=\x_u^2+\x_v^2+2\jj\,\x_u \x_v \,.\]
This equality implies the following equivalent statements:
\[\Phi^2=0 \  \Leftrightarrow\  \begin{array}{l} \x_u^2+\x_v^2=0\\  \x_u \x_v=0 \end{array} \ 
\Leftrightarrow \  \begin{array}{l} E=\x_u^2=-\x_v^2=-G\\ F=0\,. \end{array}\]
Thus we obtained:
\begin{prop}
If $\M$ is a time-like surface in $\RR^4_2$, then the coordinates $(u,v)$ are isothermal if and only if 
$\Phi^2=0$\,.
\end{prop}

 Further we consider time-like surfaces parametrized by isothermal coordinates, i.e. the function $\Phi$
satisfies the equality
\begin{equation}\label{Phi2-tl} 
\Phi^2=0\,.
\end{equation}

 The  "norm"\, of $\Phi$ satisfies the following equalities:
\[
\|\Phi\|^2=\Phi\bar{\Phi}=\x_u^2-\x_v^2=E-G=2E=-2G\,.
\]
From here the coefficients of the first fundamental form are expressed through $\Phi$ as follows:
\begin{equation}\label{EG-tl}
E=-G=\frac{1}{2}\|\Phi\|^2\,;\quad F=0 \,.
\end{equation}
Therefore the first fundamental form can be written in terms of $\Phi$ in the following way: 
\begin{equation}\label{Idt-tl}
\mathbf{I}=E\,(du^2 - dv^2)=\frac{1}{2}\|\Phi\|^2 (du^2 - dv^2)=\frac{1}{2}\|\Phi\|^2|dt|^2 .
\end{equation}
In what follows we suppose that $E<0$\,. 
Then it follows from \eqref{EG-tl} that $\Phi$ satisfies the condition:
\begin{equation}\label{modPhi2-tl} 
\|\Phi\|^2 < 0\,.
\end{equation}

 Let $\Delta^h$ denote the hyperbolic Laplace operator, given by the equality:
\[
\Delta^h = \frac{\partial^2}{\partial u^2} - \frac{\partial^2}{\partial v^2}\,.
\]  
Differentiating the equality \eqref{Phi_def-tl} and using that
$\frac{\partial}{\partial\bar t} \frac{\partial}{\partial t} = \frac{1}{4}\Delta^h$, we get:
\begin{equation}\label{dPhi_dbt-tl}
\frac{\partial\Phi}{\partial\bar t}=
\frac{\partial}{\partial\bar t}\, \left(2\, \frac{\partial \x}{\partial t} \right)=
\frac{1}{2}\Delta^h \x \,.
\end{equation}

 The last formula implies that $\ds\frac{\partial\Phi}{\partial\bar t}$ is a real valued vector function.
This is equivalent to the following equality:
\begin{equation}\label{dPhi_dbt=dbPhi_dt-tl}
\frac{\partial\Phi}{\partial\bar t}=\frac{\partial\bar\Phi}{\partial t} \,.
\end{equation}

 Thus we obtained that any function $\Phi$, given by \eqref{Phi_def-tl}, has the properties
\eqref{Phi2-tl}, \eqref{modPhi2-tl} and \eqref{dPhi_dbt=dbPhi_dt-tl}. These three properties are also sufficient
for a $\DD^4_2$-valued function to be locally obtained in the way described. We have the following statement:

\begin{theorem}\label{x_Phi-thm-tl}
 Let the time-like surface $\M=(\D,\x)$ in $\RR^4_2$ be parametrized by isothermal coordinates $(u,v)\in \D$ with  $E<0$
 and let $t=u+\jj v$.
Then the function $\Phi$, defined by \eqref{Phi_def-tl}, satisfies the conditions:
\begin{equation}\label{Phi_cond-tl} 
\Phi^2=0\,; \qquad \|\Phi\|^2 < 0\,; \qquad \frac{\partial\Phi}{\partial\bar t}=\frac{\partial\bar\Phi}{\partial t} \,.
\end{equation}

 Conversely, let $\Phi(t):\ \D \to \DD^4_2$ be a vector function, defined in $\D\subset\DD$,
satisfying the conditions \eqref{Phi_cond-tl}. Then for any point $t_0 \in \D$ there exists a neighborhood
$\D_0\subset\D$ of $t_0$ and a function $\x : \D_0 \to \RR^4_2$, such that $(\D_0,\x)$ is a regular time-like surface in $\RR^4_2$.
This surface is parametrized by isothermal coordinates $(u,v)$, given by $t=u+\jj v$ with $E<0$ and satisfies \eqref{Phi_def-tl}.
The surface $(\D_0,\x)$ is determined by the function $\Phi$ through \eqref{Phi_def-tl} uniquely up to a translation in $\RR^4_2$.
\end{theorem}

\begin{proof}
As we noted before the theorem, the equality \eqref{Phi_def-tl} implies the equalities \eqref{Phi_cond-tl}.
Next we prove the inverse assertion. Let the function $\Phi(t)$ satisfy \eqref{Phi_cond-tl}.
From the definition of $\ds\frac{\partial\Phi}{\partial\bar t}$ we have:
\[
2\ds\frac{\partial\Phi}{\partial\bar t}=
(\Re(\Phi)_u-\Im(\Phi)_v)+\jj (-\Re(\Phi)_v+\Im(\Phi)_u) \,.
\]
It follows from the third condition in \eqref{Phi_cond-tl} that $\Im 2\ds\frac{\partial\Phi}{\partial\bar t} = 0$\,,
which gives:
\[ 
\Re(\Phi)_v=\Im(\Phi)_u \,.
\]
From here it follows that for any $t_0\in\D$ there exists a neighborhood $\D_0\subset\D$ of $t_0$ and a function 
$\x : \D_0 \to \RR^4_2$, such that:
\[ 
\Re(\Phi)=\x_u\,; \quad \Im(\Phi)=\x_v \,.
\]
The last equalities are equivalent to \eqref{Phi_def-tl}. The first two conditions in \eqref{Phi_cond-tl} give
$\x_u\cdot\x_v=0$ and $\x_u^2=-\x_v^2 < 0$. 
The last conditions mean that $(\D_0,\x)$ is a regular time-like surface in $\RR^4_2$ parametrized by isothermal 
coordinates $(u,v)$ with $E<0$. 
We note that the derivatives $\x_u$ and $\x_v$, are determined uniquely by \eqref{Phi_def-tl}. Hence the function
$\x(u,v)$ is determined up to an additive constant, which completes the proof.
\end{proof}

 Finally we obtain the transformation formulas for $\Phi$ under a change of the isothermal coordinates 
and under a motion of $\M=(\D,\x)$ in $\RR^4_2$.
Consider a change of the isothermal coordinates, which in complex form is given by: $t=t(s)$.
Since the map under a change of the isothermal coordinates is conformal in $\DD$, then the function $t(s)$ is
either holomorphic or anti-holomorphic. Denote by $\tilde\Phi(s)$ the function, corresponding to the new 
coordinates $s$. First we consider the holomorphic case. From the definition \eqref{Phi_def-tl} of $\Phi$ 
we have:
\begin{equation*}
\tilde\Phi(s)=2\frac{\partial\x}{\partial s}=
2\frac{\partial\x}{\partial t}\frac{\partial t}{\partial s}+
2\frac{\partial\x}{\partial \bar t}\frac{\partial \bar t}{\partial s}=
2\frac{\partial\x}{\partial t}\frac{\partial t}{\partial s}+ 2\frac{\partial\x}{\partial \bar t} 0=
2\frac{\partial\x}{\partial t}\frac{\partial t}{\partial s} \,.
\end{equation*}
Consequently, under a holomorphic change of the coordinates $t=t(s)$ we have:
\begin{equation}\label{Phi_s-hol-tl}
\tilde\Phi(s)=\Phi(t(s)) \frac{\partial t}{\partial s} \,.
\end{equation}
Analogously, in the anti-holomorphic case we have:
\begin{equation}\label{Phi_s-antihol-tl}
\tilde\Phi(s)=\bar\Phi(t(s)) \frac{\partial \bar t}{\partial s} \,.
\end{equation}
Especially, under the change $t=\bar s$, the function $\Phi$ is transformed as follows:
\begin{equation}\label{Phi_s-t_bs-tl}
\tilde\Phi(s)=\bar\Phi(\bar s) \,.
\end{equation}

 Using the above formulas for $\Phi$ we obtain the corresponding formulas for the coefficient $E$ 
of the first fundamental form of $\M$. Thus, under a holomorphic change $t=t(s)$, we get from 
\eqref{EG-tl} and \eqref{Phi_s-hol-tl}:
\begin{equation}\label{E_s-hol-tl}
\tilde E(s)=E(t(s)) |t'(s)|^2 .
\end{equation}
Since we use coordinates, in which $E(t)<0$ and $\tilde E(s)<0$, then the above equality implies that,
admissible changes are those of them satisfying the condition $|t'(s)|^2>0$\,.
 
 Under the change $t=\bar s$, from \eqref{EG-tl} and \eqref{Phi_s-t_bs-tl} it follows that:
\begin{equation}\label{E_s-t_bs-tl}
\tilde E(s)=E(\bar s) \,.
\end{equation}

 Taking into account the Cauchy-Riemann conditions \eqref{CR-DD} it follows that the Jacobian of a holomorphic 
change $t=t(s)$ is equal to $|t'(s)|^2$. Consequently, the conditions $E(t)<0$ and $\tilde E(s)<0$ imply that
the orientation of the surface is preserved under a holomorphic change. Under the change of the coordinates
$t=\bar s$, the orientation of the surface is converted and therefore it is changed under any anti-holomorphic 
change of the coordinates satisfying $\tilde E(s)<0$. 
Thus we have:
\begin{prop}\label{Orient-holo-tl}
Let $\M$ be a time-like surface in $\RR^4_2$ and let $t$ and $s$ give isothermal coordinates on $\M$,
such that the corresponding coefficients of the first fundamental forms satisfy $E(t)<0$ and $\tilde E(s)<0$\,.
Then: $t$ and $s$ generate one and the same orientation of $\M$ if and only if the change $t=t(s)$ is
holomorphic; $t$ and $s$ generate opposite orientations of $\M$, if and only if the change is anti-holomorphic. 
\end{prop}

Now let $\M=(\D,\x)$ and $\hat\M=(\D,\hat\x)$ be two time-like surfaces in $\RR^4_2$, parametrized by 
isothermal coordinates $t=u+\jj v$ in one and the same domain $\D \subset \DD$.
Suppose that $\hat\M$ is obtained from $\M$ through a motion (possibly improper) in $\RR^4_2$ by the formula:
\begin{equation}\label{hat_M-M-mov-tl}
\hat\x(t)=A\x(t)+\vb\,; \qquad A \in \mathbf{O}(2,2,\RR), \ \vb \in \RR^4_2 \,.
\end{equation}
We obtain from here the relation between the corresponding functions $\Phi$ and $\hat\Phi$: 
\begin{equation}\label{hat_Phi-Phi-mov-tl}
\hat\Phi(t)=A\Phi(t)\,; \qquad A \in \mathbf{O}(2,2,\RR) \,.
\end{equation}
Conversely, if $\Phi$ and $\hat\Phi$ are related by \eqref{hat_Phi-Phi-mov-tl}, then we have
$\hat\x_u=A\x_u$, $\hat\x_v=A\x_v$ and consequently the equality \eqref{hat_M-M-mov-tl} 
is satisfied. Therefore the relations \eqref{hat_M-M-mov-tl} and \eqref{hat_Phi-Phi-mov-tl} are equivalent.

 In $\RR^4_2$, apart from the usual motions, which are isometries, there exist also linear anti-isometries.
One such concrete anti-isometry is given by the formula:
\begin{equation}\label{antimov-tl}
A(\, x_1\,, x_2\,, x_3\,, x_4\,) = (\, x_2\,, x_1\,, x_4\,, x_3\,)\,.
\end{equation}
Since the product of two anti-isometries in $\RR^4_2$ is an isometry, then any anti-isometry is represented 
uniquely as a product of this concrete anti-isometry and a motion in $\RR^4_2$. We denote the set of all
linear anti-isometries in $\RR^4_2$ by $\mathbf{AO}(2,2,\RR)$.

 If $\M=(\D,\x(t))$ is a time-like surface in $\RR^4_2$ parametrized by isothermal coordinates $t$, 
and $\hat\M=(\D,\hat\x(t))$ is obtained from $\M$ through an anti-isometry, then the surface $\hat\M$ 
is also time-like and $t$ gives isothermal coordinates for $\hat\M$ as well. In this case we have $\hat\x^2_u=-\x^2_u>0$\,.
Since we use isothermal coordinates satisfying the condition $E=\x^2_u<0$, then we change the numbering
of the coordinates on $\hat\M$, which in terms of the double numbers in $\DD$ means making a change of the type
$t=\jj s$.
This is written  through formulas of the following type:
\begin{equation}\label{hat_M-M-antimov-tl}
\hat\x(s)=A\x(\jj s)+\vb\,; \qquad A \in \mathbf{AO}(2,2,\RR), \ \vb \in \RR^4_2 \,.
\end{equation}
Using \eqref{Phi_s-hol-tl}, for the functions $\Phi$ and $\hat\Phi$ we find:
\begin{equation}\label{hat_Phi-Phi-antimov-tl}
\hat\Phi(s)=\jj A\Phi(\jj s)\,; \qquad A \in \mathbf{AO}(2,2,\RR) \,.
\end{equation}
As in the case of motions, it follows that the relations \eqref{hat_M-M-antimov-tl} and \eqref{hat_Phi-Phi-antimov-tl} 
are equivalent.


\section{Characterization of the minimal time-like surfaces in $\RR^4_2$ through $\Phi$ and its primitive $\Psi$.}\label{sect_Phi_Psi_R42-tl}
 
Let $\M=(\D,\x)$ be a time-like surface in $\RR^4_2$ parametrized by isothermal coordinates $t=u+\jj v$,
and $\Phi$ be the function, defined by \eqref{Phi_def-tl}. First we find the condition for $\M$ to be minimal through 
the function $\Phi$. For this purpose, we introduce the orthonormal tangent frame field $(\vX_1,\vX_2)$ , 
where $\vX_1$ and $\vX_2$ are the unit vector fields oriented as the coordinate vectors $\x_u$ and $\x_v$, respectively:
\begin{equation}\label{vX1_vX2-def-tl}
\vX_1=\frac{\x_u}{\sqrt {-E}}\,; \quad \vX_2=\frac{\x_v}{\sqrt G}=\frac{\x_v}{\sqrt {-E}}\:.
\end{equation}

From \eqref{Phi_def-tl} we express the coordinate vectors $\x_u$ and $\x_v$ through $\Phi$:
\begin{equation}\label{xuxv-tl}
\begin{array}{ll}
\x_u= \Re (\Phi)=\ds\frac{1}{2}(\Phi+\bar\Phi)\:,\\[2ex]
\x_v= \Im (\Phi)=\ds\frac{1}{2\jj}(\Phi-\bar\Phi)=\ds\frac{\jj}{2}(\Phi-\bar\Phi)\,.
\end{array}
\end{equation}

 Differentiating the equality \eqref{Phi2-tl}, we get:
\begin{equation}\label{Phi.dPhi_dbt-tl} 
\Phi\cdot\frac{\partial\Phi}{\partial\bar t}=0 \,.
\end{equation}
According to \eqref{dPhi_dbt-tl} the derivative $\ds\frac{\partial\Phi}{\partial\bar t}$
is a real function and applying complex conjugation to \eqref{Phi.dPhi_dbt-tl} we obtain:
\begin{equation}\label{bPhi.dPhi_dbt-tl} 
\bar\Phi\cdot\frac{\partial\Phi}{\partial\bar t}=0 \,.
\end{equation}
Formulas \eqref{xuxv-tl} show that $\Phi$ and $\bar\Phi$ form a basis of $T_{p,\DD}(M)$ at any point $p\in\M$.
Then it follows from \eqref{Phi.dPhi_dbt-tl} and \eqref{bPhi.dPhi_dbt-tl} that $\ds\frac{\partial\Phi}{\partial\bar t}$
is a vector, orthogonal to $T_p(\M)$ and therefore:
\begin{equation}\label{dPhi_dbt_inN-tl} 
\frac{\partial\Phi}{\partial\bar t} \in N_p(\M) \,.
\end{equation}
Using \eqref{dPhi_dbt_inN-tl} and \eqref{dPhi_dbt-tl} we get:
\begin{equation*}
\begin{array}{rl}\ds
\frac{\partial\Phi}{\partial\bar t}\!\! &=
\ds\left(\frac{\partial\Phi}{\partial\bar t}\right)^\bot=
\frac{1}{2}(\Delta^h \x )^\bot=
\frac{1}{2}(\x_{uu}-\x_{vv} )^\bot=
\frac{1}{2}(\nabla_{\x_u} \x_u - \nabla_{\x_v} \x_v )^\bot\\[2.5ex]
&=\ds\frac{1}{2}(\sigma(\x_u,\x_u) - \sigma(\x_v,\x_v) )=
E\; \frac{1}{2}(-\sigma(\vX_1,\vX_1) + \sigma(\vX_2,\vX_2) ) = E\vH \,.
\end{array}
\end{equation*}
Finally we have:
\begin{equation}\label{dPhi_dbt-Delta_x-EH-tl}
\frac{\partial\Phi}{\partial\bar t}=\frac{1}{2}\Delta^h \x = E\vH \,.
\end{equation}

The last equalities imply the following statement:
\begin{theorem}\label{Min_x_Phi-thm-tl}
Let $\M=(\D,\x)$ be a time-like surface in $\RR^4_2$ parametrized by isothermal coordinates $(u,v)\in \D$,
and $\Phi(t)$ be the vector function in $\D$ given by \eqref{Phi_def-tl}.
Then the following three conditions are equivalent:
\begin{enumerate}
	\item The function $\Phi(t)$ is holomorphic: \ $\ds\frac{\partial\Phi}{\partial\bar t}= 0$\,.
	\item The function $\x (u,v)$ is hyperbolically harmonic: \ $\Delta^h \x = 0$\,.
	\item $\M=(\D,\x)$ is a minimal time-like surface in $\RR^4_2$: \ $\vH=0$\,.  
\end{enumerate}
\end{theorem}

\smallskip

The equality \eqref{Phi_def-tl} implies further:
\begin{equation}\label{dPhi_dt_bot-tl}
\frac{\partial\Phi}{\partial t}=
\frac{\x_{uu} + \x_{vv}}{2} + \jj \x_{uv}\,; \quad 
\left(\frac{\partial\Phi}{\partial t}\right)^\bot\!\!=
\frac{\sigma (\x_u,\x_u)+\sigma (\x_v,\x_v)}{2} + \jj \sigma (\x_u,\x_v)\,.
\end{equation}

 In the case of a minimal time-like surface, $\Phi$ is a holomorphic function, i.e.
$\ds\frac{\partial\Phi}{\partial\bar t}=0$,
and for $\ds\frac{\partial\Phi}{\partial t}$ we use the usual denotation $\Phi'$.

\smallskip
The condition $\vH=\frac{1}{2}(-\sigma(\vX_1,\vX_1)+\sigma(\vX_2,\vX_2))=0$ for $\M$ to be minimal
gives that: 
\begin{equation}\label{sigma_22-tl}
\sigma(\vX_2,\vX_2)=\sigma(\vX_1,\vX_1)\,; \qquad  \sigma(\x_v,\x_v)=\sigma(\x_u,\x_u)\,.
\end{equation}

 Then $\Phi'$ and its orthogonal projection into $N_{p,\DD}(\M)$ are:
\begin{equation}\label{PhiPr-tl}
\Phi^\prime=\frac{\partial\Phi}{\partial u}=\x_{uu}+\jj\x_{uv}\,; \quad
\Phi^{\prime \bot}=\x_{uu}^\bot +\jj \x_{uv}^\bot =\sigma (\x_u,\x_u)+\jj\sigma (\x_u,\x_v)\,.
\end{equation}

 From here we can express $\sigma (\x_u,\x_u)$, $\sigma (\x_v,\x_v)$ and $\sigma (\x_u,\x_v)$ through $\Phi$:
\begin{equation}\label{sigma_uu_uv-tl}
\begin{array}{l}
\sigma(\x_u,\x_u) =  \Re (\Phi^{\prime \bot}) = \,\ds\frac{1}{2}\,(\Phi^{\prime \bot}+\overline{\Phi^{\prime \bot}})=
\ds\frac{1}{2}(\Phi^{\prime \bot}+{\overline{\Phi^\prime}}^\bot)\,,\\[4mm]

\sigma(\x_v,\x_v) =  \Re (\Phi^{\prime \bot}) = \,\ds\frac{1}{2}\,(\Phi^{\prime \bot}+\overline{\Phi^{\prime \bot}})=
\ds\frac{1}{2}(\Phi^{\prime \bot}+{\overline{\Phi^\prime}}^\bot)\,,\\[4mm]

\sigma(\x_u,\x_v) = \Im (\Phi^{\prime \bot})=
\ds\frac{1}{2\jj}(\Phi^{\prime \bot}-\overline{\Phi^{\prime \bot}})=
\ds\frac{\jj}{2}(\Phi^{\prime \bot}-{\overline{\Phi^\prime}}^\bot)\,.
\end{array}
\end{equation}

\medskip

 In view of Theorem \ref{Min_x_Phi-thm-tl}\,, for a minimal time-like surface $\M=(\D,\x)$, 
the position vector function $\x$ is hyperbolically harmonic and therefore we can introduce locally
the harmonically conjugate function $\y$ of $\x$ via the Cauchy-Riemann conditions: 
$\y_u=\x_v$ and $\y_v=\x_u$. 
Let $\Psi$ be the $\DD^4_2$-valued vector function given by:
\begin{equation}\label{Psi-def-tl}
\Psi=\x+\jj\y \,.
\end{equation}
The function $\Psi$ is holomorphic and the functions $\x$ and $\Phi$ are expressed by $\Psi$ as follows:
\begin{equation}\label{x_Phi_Psi-tl}
\x=\Re\Psi \,; \qquad \Phi=\x_u+\jj\x_v=\x_u+\jj\y_u=\frac{\partial\Psi}{\partial u}=\Psi' .
\end{equation}
 Next we express the condition for $\M$ to be minimal through the function $\Psi$:

\begin{theorem}\label{Min_x_Psi-thm-tl}
Let $\M=(\D,\x)$ be a minimal time-like surface in $\RR^4_2$, parametrized by isothermal coordinates $(u,v)\in \D$.
Then\, $\x$ is given locally in the form:
\begin{equation}\label{x_Psi-tl}
\x(u,v)=\Re\Psi(t) \,,
\end{equation}
where $\Psi$ is a holomorphic $\DD^4_2$-valued function of $t=u+\jj v$, and satisfies the conditions:
\begin{equation}\label{Psi_cond-tl}
\Psi'^{\,2}=0\,; \qquad \|\Psi'\|^2 < 0\,.
\end{equation}

 Conversely, if $\Psi$ is a holomorphic $\DD^4_2$-valued function, defined in a domain $\D\subset \DD$, and
satisfying the condition \eqref{Psi_cond-tl}, then the pair $(\D,\x)$, where $\x$ is given by \eqref{x_Psi-tl}, 
is a minimal time-like surface in $\RR^4_2$ in isothermal coordinates $(u,v)$.
\end{theorem}
\begin{proof}
If $\M=(\D,\x)$ is a minimal time-like surface in $\RR^4_2$, then $\Psi$, defined by \eqref{Psi-def-tl},
satisfies \eqref{x_Phi_Psi-tl}. The properties \eqref{Psi_cond-tl} are equivalent to the corresponding
properties \eqref{Phi_cond-tl} of $\Phi$.

 Conversely, if $\Psi$ is a holomorphic $\DD^4_2$-valued function with the properties \eqref{Psi_cond-tl}, 
then the functions $\x=\Re\Psi$ and $\Phi=\Psi'$ satisfy the relation $\Phi=\x_u+\jj\x_v$. Applying 
Theorem \ref{x_Phi-thm-tl}, it follows that $(\D,\x)$ is a regular time-like surface in $\RR^4_2$, parametrized 
by isothermal coordinates $(u,v)$. Since $\x$ is hyperbolically harmonic, then according to Theorem
\ref{Min_x_Phi-thm-tl}, $(\D,\x)$ is a minimal surface. 
\end{proof}

 Next we consider how the function $\Psi$ is transformed under a change of the isothermal coordinates 
and under a motion of the surface $\M$ in $\RR^4_2$.

Under a holomorphic change of the coordinates $t=t(s)$ we have $\x(t(s))=\Re\Psi(t(s))$. 
Since the function $\Psi(t(s))$ is holomorphic, then $\Psi(t)$  is changed into $\Psi(t(s))$. 
Any anti-holomorphic change can be reduced to a holomorphic change and the special change $t=\bar s$. 
Under the latter change we have $\x(\bar s)=\Re\bar\Psi(\bar s)$ and therefore the function $\Psi (t)$ 
is transformed into $\bar\Psi(\bar s)$. 
 
Let $\M=(\D,\x)$ and $\hat\M=(\D,\hat\x)$ be two minimal time-like surfaces, parametrized by isothermal 
coordinates.  These surfaces are related by a motion (possibly improper) in $\RR^4_2$ by the formula 
$\hat\x(t)=A\x(t)+\vb$, where $A \in \mathbf{O}(2,2,\RR)$ and $\vb \in \RR^4_2$, if and only if the
functions $\Psi$ and $\hat\Psi$ are related by the equality $\hat\Psi (t)=A\Psi (t)+\vb$.

 Let us consider also the case, when $\hat\M=(\D,\hat\x)$ is obtained from $\M=(\D,\x)$ through an anti-isometry,
which according to \eqref{hat_M-M-antimov-tl} is of the type $\hat\x(s)=A\x(\jj s)+\vb$.
Then we have $\hat\x(s)=\Re (A\Psi(\jj s) + \vb)$, from where it follows that $\hat\Psi (s)=A\Psi (\jj s)+\vb$. 

\medskip

 Theorem \ref{Min_x_Psi-thm-tl} shows that from a given minimal time-like surface $\M=(\D,\x)$ we can (at least locally)
obtain another minimal time-like surfaces modifying the function $\Psi$. 
For example, if $k= \rm{const}>0$, then the function $k\Psi$ satisfies the conditions \eqref{Psi_cond-tl}
and therefore Theorem \ref{Min_x_Psi-thm-tl} is applicable to $(\D,\hat\x=\Re( k\Psi))$.
In this way we obtained a new minimal time-like surface $\hat\M$ in $\RR^4_2$, for which: $\hat\x=k\x$.   
The last equality means that $\hat\M$ is obtained from $\M$ through a homothety with coefficient $k$.
Thus we have:
\begin{prop}\label{Min_Surf_Hom-tl}
Let $\M$ be a minimal time-like surface in $\RR^4_2$, parametrized by isothermal coordinates $(u,v)$, with $E<0$\,.
If $\hat\M$ is obtained from $\M$ through a homothety with coefficient $k$, 
then $\hat\M$ is also a minimal time-like surface, parametrized by isothermal coordinates $(u,v)$.
The corresponding functions $\hat\Phi$ and $\hat E$ of $\hat\M$ are given by:
\begin{equation}\label{hat_Phi-Phi-hmt-tl}
\hat\Phi(t)=k\Phi(t)\,; \qquad  \hat E(t) = k^2 E(t)\,.
\end{equation}
\end{prop}
\begin{proof}
The notes before the proposition show that the notions of a minimal time-like surface and isothermal 
coordinates are invariant under a homothety in $\RR^4_2$.  
The first equality in \eqref{hat_Phi-Phi-hmt-tl} follows from $\hat\Phi=(k\Psi)'=k\Phi$.
The second equality follows from the first one and from \eqref{EG-tl}. 
\end{proof}

 Any minimal time-like surface generates another minimal time-like surface in the following way:  
Consider the harmonic conjugate function $\y$ to $\x$. The equality $\x=\Re\Psi$ implies that $\y=\Re(\jj\Psi)$. 
The function $\jj\Psi$ satisfies the conditions $(\jj\Psi')^{2}=0$ and $\|\jj\Psi'\|^2 = -\|\Psi'\|^2 > 0$.
In order to apply Theorem \ref{Min_x_Psi-thm-tl} to $\y$ and $\jj\Psi$, it is necessary to make change of 
the isothermal coordinates of the type $t=\jj s$. Then $\hat\Psi(s)=\jj\Psi(\jj s)$ satisfies the conditions 
\eqref{Psi_cond-tl} and Theorem \ref{Min_x_Psi-thm-tl} gives that $\y(\jj s)=\Re(\jj\Psi(\jj s))$ determines
a minimal time-like surface in $\RR^4_2$. Since the considerations are local, it follows that $\y$ is defined
in the whole domain $\jj\D$.
Now we can give the following:
\begin{dfn}\label{Conj_Min_Surf-tl} 
Let $\M=(\D,\x(t))$ be a minimal time-like surface in $\RR^4_2$, parametrized by isothermal coordinates $t\in\D$.
The surface $\bar\M=(\jj\D,\y(\jj s))$, where $\y$ is (hyperbolically) harmonic conjugate to $\x$,
is said to be the \textbf{conjugate minimal surface} to $\M=(\D,\x(t))$. 
\end{dfn}
Taking into account the transformation properties of $\Psi$, we conclude that the function $\y$ is invariant under
a holomorphic change of the isothermal coordinates, while under an anti-holomorphic change the function $\y$ is
changed into $-\y$. Since the harmonic conjugate to a given function is determined uniquely up to an
additive constant, this means geometrically that the conjugate minimal surface of a given minimal surface is 
determined locally up to a motion in $\RR^4_2$.

If $\hat\Phi(s)$ and $\hat E(s)$ are the corresponding functions of $\bar\M$, then from \eqref{EG-tl} we have:
\begin{equation}\label{Phi_conj-tl}
\hat\Phi(s) = (\jj\Psi(\jj s))' = \Phi(\jj s)\,; \qquad \|\hat\Phi(s)\|^2=\|\Phi(\jj s)\|^2\,; \qquad 
\hat E(s) = E(\jj s)\,.
\end{equation}

\medskip

 Further, in the above construction of the conjugate minimal surface we replace $\jj$ in formula $\jj\Psi$ 
by an arbitrary number of the type $\e^{\jj\theta}$, where $\theta \in \RR$. Since the function
$\e^{\jj\theta}\Psi$ satisfies the conditions \eqref{Psi_cond-tl}, then we obtain a one-parameter family 
of minimal time-like surfaces:
\begin{equation}\label{1-param_family-tl}
\x_\theta = \Re \e^{\jj\theta}\Psi = \x\cosh\theta + \y\sinh\theta \,.
\end{equation}
We give the following
\begin{dfn}\label{1-param_family_assoc_surf-tl} 
Let $\M=(\D,\x)$ be a minimal time-like surface in $\RR^4_2$, parametrized by isothermal coordinates $(u,v)\in\D$.
The family $\M_\theta=(\D,\x_\theta)$, where $\x_\theta$ is given by \eqref{1-param_family-tl} and 
$\theta$ is a real parameter, is said to be a one parameter family of minimal surfaces 
\textbf{associated} to the given one.  
\end{dfn}
Here we have to note that the one-parameter family of associated minimal surfaces is determined locally 
up to a motion in $\RR^4_2$.

 In contrast to the case of the space-like surfaces in $\RR^4_2$, the conjugate minimal time-like surface 
does not belong to the one-parameter family of the associated  minimal time-like surfaces, since 
$\jj \neq \e^{\jj\theta}$ for $\theta \in \RR $.

If $\Phi_\theta(t)$ and $E_\theta(t)$ are the corresponding functions for $\M_\theta$, then from \eqref{EG-tl}
we have:
\begin{equation}\label{Phi_1-param_family-tl}
\Phi_\theta(t) = (\e^{\jj\theta}\Psi(t))' = \e^{\jj\theta}\Phi(t)\,; \qquad \|\Phi_\theta(t)\|^2=\|\Phi(t)\|^2\,;
\qquad  E_\theta(t) = E(t)\,.
\end{equation}

\medskip

 A basic property of this family is that any two surfaces of the family are isometric to each other.
\begin{prop}\label{Isom_M_theta-M-tl}
Let $\M=(\D,\x)$ be a minimal time-like surface in $\RR^4_2$, parametrized by isothermal coordinates $(u,v)\in\D$
and $\M_\theta=(\D,\x_\theta)$ be the corresponding one-parameter family of associated minimal time-like surfaces.
The map $\mathcal{F}_\theta: \x(u,v)\rightarrow \x_\theta(u,v)$ gives an \textbf{isometry} between
$\M$ and $\M_\theta$ for any $\theta$.
\end{prop}
\begin{proof}
The map $\mathcal{F}_\theta$, written in local coordinates $(u,v)$, coincides with the identity in $\D$. In
\eqref{Phi_1-param_family-tl} we proved that $E_\theta = E$, which shows that $\mathcal{F}_\theta$ is an 
isometry. 
\end{proof}

 If $\bar\M=(\D,\y(t))$ is the surface conjugate to the given minimal time-like surface $\M=(\D,\x(t))$, 
then there exists also a natural map between $\M$ and $\bar\M$, given by the formula 
$\mathcal{F}: \x(t) \rightarrow \y(t)$. The formulas before Definition \ref{Conj_Min_Surf-tl} show that
the corresponding coefficients satisfy the relation $\hat E(t)=-E(t)$. This means that $\mathcal{F}$ is 
an anti-isometry. Thus we have:
\begin{prop}\label{AntiIsom_conj_M-M-tl}
Let $\M=(\D,\x)$ be a minimal time-like surface in $\RR^4_2$, parametrized by isothermal coordinates $(u,v)\in\D$,
and $\bar\M=(\D,\y)$ be its conjugate minimal time-like surface. Then the map 
$\mathcal{F}: \x(u,v)\rightarrow \y(u,v)$ gives an \textbf{anti-isometry} between $\M$ and $\bar\M$.
\end{prop}


\section{The curvatures $K$ and $\varkappa$ expressed through the Weingarten operators and the function $\Phi$.}\label{sect_K_kappa-Phi_R42-tl}

 Let $\M=(\D,\x)$ be a minimal time-like surface in $\RR^4_2$, parametrized by isothermal coordinates $(u,v)\in\D$ and
let $\vX_1$, $\vX_2$ be the unit tangent vectors, oriented as the coordinate vectors $\x_u$, $\x_v$.
Denote by $\n_1$, $\n_2$ an orthonormal basis of $N_p(\M)$ for any $p \in \M$ with the properties: 
$\n_1^2=-1$\,, $\n_2^2=1$ and the quadruple $(\vX_1,\vX_2,\n_1,\n_2)$ forms a right oriented orthonormal 
frame field in $\RR^4_2$. 
 
Further, $A_{\n}$ denotes the Weingarten operator in $T_p(\M)$, corresponding to the normal vector $\n$. Since
$A_{\n}\vX\cdot \vY=\sigma (\vX,\vY)\cdot \n $ and $\vH=0$, then $\trace A_{\n}=0 $ for any $\n$. Therefore 
the operators $A_{\n_1}$ and $A_{\n_2}$ have the following matrix forms:

\begin{equation}\label{A1A2_nu_lambda_rho_mu_R42-tl}
A_{\n_1}=
\left(
\begin{array}{rr}
\nu      & -\lambda\\
\lambda  & -\nu
\end{array}
\right); \qquad
A_{\n_2}=
\left(
\begin{array}{rr}
\rho & -\mu\\
\mu  & -\rho
\end{array}
\right).
\end{equation}
The second fundamental form $\sigma$ satisfies the equalities:
\begin{equation}\label{sigma_nu_lambda_rho_mu-tl}
\begin{array}{l}
\sigma (\vX_1,\vX_1)=-(\sigma (\vX_1,\vX_1)\cdot \n_1)\n_1+(\sigma (\vX_1,\vX_1)\cdot \n_2)\n_2=
\phantom{-} \nu \n_1 - \rho \n_2 \,;\\
\sigma (\vX_1,\vX_2)=-(\sigma (\vX_1,\vX_2)\cdot \n_1)\n_1+(\sigma (\vX_1,\vX_2)\cdot \n_2)\n_2=
-\lambda \n_1 + \mu \n_2 \,; \\
\sigma (\vX_2,\vX_2)=\sigma (\vX_1,\vX_1)= \nu \n_1 - \rho \n_2 \,.
\end{array}
\end{equation}

 The Gauss curvature in view of \eqref{K_sigma-tl} and \eqref{sigma_22-tl}, satisfies the equality:
\begin{equation}\label{K_min_sigma-tl}
\begin{array}{rl}
K &=-R(\vX_1,\vX_2)\vX_2\cdot \vX_1=-\sigma(\vX_1,\vX_1)\sigma(\vX_2,\vX_2)+\sigma^2(\vX_1,\vX_2)\\
  &=-\sigma^2(\vX_1,\vX_1)+\sigma^2 (\vX_1,\vX_2).
\end{array}
\end{equation}
Taking into account \eqref{sigma_nu_lambda_rho_mu-tl} and \eqref{K_min_sigma-tl}, we express $K$ via $A_{\n_1}$ 
and $A_{\n_2}$:
\begin{equation}\label{K_nu_lambda_rho_mu-tl}
K=-(-\nu^2+\rho^2)+(-\lambda^2+\mu^2)=\nu^2-\lambda^2-\rho^2+\mu^2=-\det(A_{\n_1})+\det(A_{\n_2})\,.
\end{equation}

\medskip

 In order to express $K$ through $\Phi$, we obtain from \eqref{PhiPr-tl}:
\begin{equation*}
\Phi^{\prime \bot}=\sigma (\x_u,\x_u)+\jj\sigma (\x_u,\x_v)=-E(\sigma (\vX_1,\vX_1)+\jj\sigma (\vX_1,\vX_2))\,.
\end{equation*}
The last equality implies that:
\begin{equation*}
{\|\Phi^{\prime \bot}\|}^2
=E^2(\sigma^2(\vX_1,\vX_1)-\sigma^2 (\vX_1,\vX_2))\,.
\end{equation*}
From here and \eqref{EG-tl} we get:
\begin{equation}\label{s2+s2-tl}
\sigma^2(\vX_1,\vX_1)-\sigma^2 (\vX_1,\vX_2)=\frac{{\|\Phi^{\prime \bot}\|}^2}{E^2}=
\frac{4{\|\Phi^{\prime \bot}\|}^2}{\|\Phi\|^4}\:.
\end{equation}
Replacing in \eqref{K_min_sigma-tl}, we obtain the first formula expressing $K$ through $\Phi$:
\begin{equation}\label{K_Phi-tl}
K= -\frac{{\|\Phi^{\prime \bot}\|}^2}{E^2} = -\ds\frac{4{\|\Phi^{\prime \bot}\|}^2}{\|\Phi\|^4}\:.
\end{equation}

 This formula has the disadvantage, that the function $\Phi^{\prime \bot}$ is not holomorphic, since 
the orthogonal projection into the normal space in general does not preserve the holomorphy. That is why
we shall find another representation of $\|\Phi^{\prime \bot}\|^2$ through holomorphic functions.
The formula $\Phi^2=0$ means that $\Phi$ and $\bar\Phi$ are orthogonal with respect to the analogue 
of the Hermitian product ($a \cdot \bar b$) in $\DD^4_2$. Then it follows from \eqref{Phi_def-tl} and 
\eqref{xuxv-tl} that they form an orthogonal basis (with respect to this Hermitian product) of $T_{p,\DD}(\M)$
at any point $p\in\M$. Hence the tangential projection of $\Phi^\prime$ is given as follows:
\[
\Phi^{\prime\top}
=\ds\frac{\Phi^{\prime\top}\cdot\bar \Phi}{\|\Phi\|^2}\Phi+\ds\frac{\Phi^{\prime\top}\cdot \Phi}{\|\bar \Phi\|^2}\bar \Phi 
=\ds\frac{\Phi' \cdot \bar \Phi}{\|\Phi\|^2}\Phi + \ds\frac{\Phi' \cdot \Phi}{\|\bar \Phi\|^2}\bar \Phi\,.
\]

 Differentiating the equality $\Phi^2=0$ we get the following relation:
\begin{equation}\label{Phi.dPhi_dt-tl}
\Phi\cdot\Phi^\prime=0\,.
\end{equation}
Applying the last equality to the formula for $\Phi^{\prime\top}$, we find the projections of $\Phi'$:
\begin{equation}\label{Phipn-tl}
\Phi^{\prime\top}= \ds\frac{\Phi' \cdot \bar \Phi}{\|\Phi\|^2}\Phi\,; \quad\quad \Phi^{\prime\bot}=\Phi'-\Phi^{\prime\top}=
\Phi'-\ds\frac{\Phi' \cdot \bar \Phi}{\|\Phi\|^2}\Phi\,.
\end{equation}
By direct computations we get:
\begin{equation}\label{mPhipn2-tl}
{\|\Phi^{\prime\bot}\|}^2 = \ds\frac{\|\Phi\|^2\|\Phi'\|^2-|\bar \Phi \cdot \Phi'|^2}{\|\Phi\|^2}\ .
\end{equation}
From the above formula and \eqref{K_Phi-tl} it follows that:
\begin{equation}\label{K_Phi.Phip-tl}
K= -\ds\frac{4(\|\Phi\|^2\|\Phi'\|^2-|\bar \Phi \cdot \Phi'|^2)}{\|\Phi\|^6}\;.
\end{equation} 
Denoting the bivector product of $\Phi$ and $\Phi'$ with $\Phi\wedge\Phi'$, then:
\[\|\Phi\wedge\Phi'\|^2=\|\Phi\|^2\|\Phi'\|^2-|\bar \Phi \cdot \Phi'|^2\]  
and replacing into \eqref{K_Phi.Phip-tl}, we obtain:
\begin{equation}\label{K_Phi_bv-tl}
K= -\ds\frac{4\|\Phi\wedge\Phi'\|^2}{\|\Phi\|^6}\;.
\end{equation}
The last formula for $K$ has the advantage over \eqref{K_Phi-tl}, that $\Phi\wedge\Phi'$ is a bivector
holomorphic (with respect to $\DD$) function. 

\medskip

 Now we shall find another representation also for the tangential projection $\Phi^{\prime\top}$ of 
$\Phi^\prime$. In formula \eqref{Phipn-tl} we find the coefficient before $\Phi$ through $E$:
\[
\Phi' \cdot \bar \Phi = \frac{\partial\Phi}{\partial t} \cdot \bar \Phi = 
\frac{\partial(\Phi \cdot \bar \Phi)}{\partial t}- \Phi \cdot \frac{\partial \bar \Phi}{\partial t}=
\frac{\partial(\|\Phi\|^2)}{\partial t}\:,
\]
by using in the last equality that $\bar\Phi$ is anti-holomorphic and therefore
$\frac{\partial \bar \Phi}{\partial t}=0$. Further, applying $\|\Phi\|^2=2E$ in view of \eqref{EG-tl},
we find the coefficient before $\Phi$ in \eqref{Phipn-tl}:
\begin{equation}\label{d_ln_E_dt-tl}
\ds\frac{\Phi' \cdot \bar \Phi}{\|\Phi\|^2}=
\frac{\partial(\|\Phi\|^2)}{\partial t} \frac{1}{\|\Phi\|^2}=
\frac{\partial E}{\partial t} \frac{1}{E} = \frac{\partial\ln |E|}{\partial t} \:.
\end{equation}
Hence the tangential projection of $\Phi'$ has the form:
\begin{equation}\label{Phi_p-Tang-tl}
\Phi^{\prime\top}=
\frac{\partial\ln |E|}{\partial t}\Phi \,. 
\end{equation}

 Finally we shall find another relation between $K$ and $\Phi$ by using that according to the classical
Gauss formula $K$ is expressed through the second derivatives of the coefficients of the first 
fundamental form  of $\M$. This means in isothermal coordinates that we can express $K$ through 
the second derivatives of $E$, or through the second derivatives of $\|\Phi\|^2$ in view of \eqref{EG-tl}. 
In order to obtain this equation, we use the orthogonal decomposition of $\Phi'$ and \eqref{Phi_p-Tang-tl}:
\begin{equation}\label{Phi_p-Ort_1-tl}
\Phi'= \Phi^{\prime\top} + \Phi^{\prime\bot} =
\frac{\partial\ln |E|}{\partial t}\Phi + \Phi^{\prime\bot} .
\end{equation}
Differentiating the last equality with respect to $\bar t$ and taking into account that $\Phi'$ and $\Phi$
are holomorphic, we get:
\begin{equation*}
0=\frac{\partial^2\ln |E|}{\partial\bar t\partial t}\Phi + \frac{\partial (\Phi^{\prime\bot})}{\partial\bar t}\:.
\end{equation*}
Multiplying the last equality by $\bar\Phi$ we find:
\begin{equation}\label{Delta_lnE_Phi_1-tl}
\frac{\partial^2\ln |E|}{\partial\bar t\partial t}\|\Phi\|^2 +
\frac{\partial (\Phi^{\prime\bot})}{\partial\bar t} \bar\Phi = 0\,.
\end{equation}
Applying to the first addend the equalities $\frac{\partial}{\partial\bar t} \frac{\partial}{\partial t} = \frac{1}{4}\Delta^h$ 
and $\|\Phi\|^2=2E$, we obtain:
\[
\frac{\partial^2\ln |E|}{\partial\bar t\partial t}\|\Phi\|^2=
\frac{E\Delta^h\ln |E|}{2}\:.
\]
For the second addend we have:
\[
\begin{array}{rl}
\ds\frac{\partial (\Phi^{\prime\bot})}{\partial\bar t} \bar\Phi 
            &=\ds\frac{\partial (\Phi^{\prime\bot}\cdot\bar\Phi)}{\partial\bar t} -
              \Phi^{\prime\bot} \frac{\partial \bar\Phi}{\partial\bar t}=
              0-\Phi^{\prime\bot}\overline{\Phi'}\\[1ex]
						&=-\Phi^{\prime\bot}\overline{\Phi'}{}^\bot=
						  -\Phi^{\prime\bot}\overline{\Phi^{\prime\bot}}=-\|\Phi^{\prime\bot}\|^2 .
\end{array}
\]
Now \eqref{Delta_lnE_Phi_1-tl} gets the form:
\begin{equation*}
\frac{E\Delta^h\ln |E|}{2} - \|\Phi^{\prime\bot}\|^2 = 0\,.
\end{equation*}
According to equality \eqref{K_Phi-tl} we replace $-\|\Phi^{\prime\bot}\|^2$ with $\frac{1}{4}K\|\Phi\|^4$
and get:
\begin{equation*}
2E\Delta^h\ln |E| + K\|\Phi\|^4 = 0\,.
\end{equation*}
Using again the equality $\|\Phi\|^2=2E$, we find:
\begin{equation}\label{Delta_lnE_2K-tl}
\frac{\Delta^h\ln |E|}{E} + 2K = 0\,.
\end{equation}
This is the classical fundamental Gauss equation for a minimal time-like surface in isothermal coordinates.

By using formula \eqref{Delta_lnE_2K-tl}, we obtain another way to express $K$ through $\Phi$:
\begin{equation}\label{K_Delta_lnE_lnPhi-tl}
K = \frac{\Delta^h\ln |E|}{-2E} = \frac{\Delta^h\ln (-\|\Phi\|^2)}{-\|\Phi\|^2}\:.
\end{equation}

\medskip

	Now we find the corresponding formulas expressing the curvature of the normal connection $\varkappa$ 
on $\M$ through the Weingarten operators and $\Phi$.
Taking into account equalities \eqref{kappa_A-tl} and \eqref{A1A2_nu_lambda_rho_mu_R42-tl} we have:
\begin{equation}\label{kappa_nu_lambda_rho_mu-tl}
\begin{array}{rl}
\varkappa &=\ R^N(\vX_1,\vX_2)\n_2\cdot \n_1 = 
						 [A_{\n_2},A_{\n_1}] \vX_1\cdot \vX_2=A_{\n_1}\vX_1\cdot A_{\n_2}\vX_2-A_{\n_2}\vX_1\cdot A_{\n_1}\vX_2\\
					&=\ (\nu\vX_1 + \lambda\vX_2)\cdot (-\mu\vX_1 - \rho\vX_2)-(\rho\vX_1 + \mu\vX_2)\cdot (-\lambda\vX_1 - \nu\vX_2)\\
					&=\ \nu\mu - \rho\lambda  - (\rho\lambda - \nu\mu) = 2\nu\mu - 2\rho\lambda \,.
\end{array}
\end{equation}

Denote by $\det (\va\,,\vb)$ the determinant of two normal vectors $\va$ and $\vb$, calculated with respect to 
the basis $(\n_1\,,\n_2)$ of $N_p(\M)$. With the aid of \eqref{sigma_nu_lambda_rho_mu-tl} we find:
\[
\begin{array}{rl}
\det \big(\sigma(\x_u,\x_u)\,,\sigma(\x_u,\x_v)\big) &= 
\ \det \big(\sigma(\sqrt {-E} \vX_1,\sqrt {-E} \vX_1)\,,\sigma(\sqrt {-E} \vX_1,\sqrt {-E} \vX_2)\big)\\ 
&=\ E^2 \det \big(\sigma(\vX_1,\vX_1)\,,\sigma(\vX_1,\vX_2)\big)\\
&=\ E^2 \det (\nu \n_1 - \rho \n_2\,,-\lambda \n_1 + \mu \n_2) = E^2 (\nu\mu - \rho\lambda)\,.
\end{array}
\]
 From here and \eqref{kappa_nu_lambda_rho_mu-tl} we express $\varkappa$ through $\sigma$:
\begin{equation}\label{kappa_sigma-tl}
\varkappa = \frac{2\,\det \big(\sigma(\x_u,\x_u)\,,\sigma(\x_u,\x_v)\big)}{E^2}\:.
\end{equation}

 Taking into account \eqref{PhiPr-tl}, we get:
\[
\begin{array}{rl}
\det \big(\Phi^{\prime \bot}\,,\overline{\Phi^{\prime \bot}}\,\big) &=
\ \det \big(\sigma (\x_u,\x_u)+\jj\sigma (\x_u,\x_v)\,,\sigma (\x_u,\x_u)-\jj\sigma (\x_u,\x_v)\big)\\
&=\ -2\jj \,\det \big(\sigma(\x_u,\x_u)\,,\sigma(\x_u,\x_v)\big)\,.
\end{array}
\]
Applying the last equality to \eqref{kappa_sigma-tl}, we find an expression for $\varkappa$ through
$\Phi^{\prime \bot}$:
\begin{equation}\label{kappa_Phi_pn-tl}
\varkappa = - \frac{\jj\,\det \big(\Phi^{\prime \bot}\,,\overline{\Phi^{\prime \bot}}\,\big)}{E^2} =
- \frac{4\jj\,\det \big(\Phi^{\prime \bot}\,,\overline{\Phi^{\prime \bot}}\,\big)}{\|\Phi\|^4}\:.
\end{equation}

 Similarly to the case of $K$, we shall find a formula for $\varkappa$ through $\Phi$ and $\Phi^{\prime}$
without using the orthogonal projection into the normal space. For that purpose, let 
$\det (\va\,,\vb\,,\vc\,,\vd)$ denote the determinant of order four of the vectors $\va$, $\vb$, $\vc$ and $\vd$, 
calculated with respect to the standard basis in $\DD^4_2$. Then we have:
\[
\det \big(\vX_1\,,\vX_2\,,\Phi^{\prime \bot}\,,\overline{\Phi^{\prime \bot}}\,\big)=
\det \big(\Phi^{\prime \bot}\,,\overline{\Phi^{\prime \bot}}\,\big) \det \big(\vX_1\,,\vX_2\,,\n_1\,,\n_2\big)=
\det \big(\Phi^{\prime \bot}\,,\overline{\Phi^{\prime \bot}}\,\big)\,.
\]
From here it follows that:
\[
\det \big(\Phi^{\prime \bot}\,,\overline{\Phi^{\prime \bot}}\,\big)=
\det \big(\vX_1\,,\vX_2\,,\Phi^{\prime \bot}\,,\overline{\Phi^{\prime \bot}}\,\big)=
\det \big(\vX_1\,,\vX_2\,,\Phi^\prime-\Phi^{\prime \top}\,,\overline{\Phi^\prime}-{\overline{\Phi^\prime}}^\top\,\big)\,.
\]
The vectors $\Phi^{\prime \top}$ and ${\overline{\Phi^\prime}}^\top$ tangent to ${\M}$ are linear combinations 
of $\vX_1$ and $\vX_2$. Then we have:
\begin{equation*}
\det \big(\Phi^{\prime \bot}\,,\overline{\Phi^{\prime \bot}}\,\big)=
\det \big(\vX_1\,,\vX_2\,,\Phi^\prime\,,\overline{\Phi^\prime}\,\big)\,.
\end{equation*}
On the other hand we have:
\[
\begin{array}{rl}
\det \big(\Phi\,,\bar\Phi\,,\Phi^\prime\,,\overline{\Phi^\prime}\,\big) &=
\ \det \big(\x_u+\jj\x_v\,,\x_u-\jj\x_v\,,\Phi^\prime\,,\overline{\Phi^\prime}\,\big)\\
&=\ -2\jj\,\det \big(\x_u\,,\x_v\,,\Phi^\prime\,,\overline{\Phi^\prime}\,\big)=
2\jj E \,\det \big(\vX_1\,,\vX_2\,,\Phi^\prime\,,\overline{\Phi^\prime}\,\big)\,.
\end{array}
\]
The last two formulas imply that:
\[
\det \big(\Phi^{\prime \bot}\,,\overline{\Phi^{\prime \bot}}\,\big)=
\frac{\jj\, \det \big(\Phi\,,\bar\Phi\,,\Phi^\prime\,,\overline{\Phi^\prime}\,\big)}{2E}\:.
\]
Replacing in \eqref{kappa_Phi_pn-tl}, we obtain:
\begin{equation}\label{kappa_Phi-tl}
\varkappa = - \frac{\det \big(\Phi\,,\bar\Phi\,,\Phi^\prime\,,\overline{\Phi^\prime}\,\big)}{2E^3} =
- \frac{4\,\det \big(\Phi\,,\bar\Phi\,,\Phi^\prime\,,\overline{\Phi^\prime}\,\big)}{\|\Phi\|^6}\:.
\end{equation}

\smallskip

 For any minimal time-like surface $\M=(\D,\x)$ in $\RR^4_2$, parametrized by isothermal coordinates,
the Gauss curvature $K$ and the normal curvature $\varkappa$ are given by the following formulas:

- through the Weingarten operators:
\begin{equation}\label{K_kappa_nu_lambda_rho_mu-tl}
K= \nu^2-\lambda^2-\rho^2+\mu^2 \,; \quad\quad
\varkappa = 2\nu\mu - 2\rho\lambda \,.
\end{equation}

\medskip

- through the complex function $\Phi$: 
\begin{equation}\label{K_kappa_Phi_R42-tl}
K= \ds\frac{-4{\|\Phi^{\prime \bot}\|}^2}{\|\Phi\|^4}
=\ds\frac{-4\|\Phi\wedge\Phi'\|^2}{\|\Phi\|^6}\,; \quad\quad
\varkappa = \frac{-4\,\det \big(\Phi\,,\bar\Phi\,,\Phi^\prime\,,\overline{\Phi^\prime}\,\big)}{\|\Phi\|^6}\:.
\end{equation}

\medskip

 Along with the orthonormal basis $(\n_1\,,\n_2)$ of $N_p(\M)$, we have used up to now, it is useful to 
use the \emph{isotropic basis} $(\m_1\,,\m_2)$, given by:
\begin{equation}\label{m12_n12_R42-tl}
\m_1=\frac{\n_1 + \n_2}{2}\:; \qquad \m_2=\frac{-\n_1 + \n_2}{2}\:.
\end{equation}
These vectors $\m_1$ and $\m_2$ satisfy the equalities:
\begin{equation}\label{m12_psd_ort_R42-tl}
\m_1^2=\m_2^2=0\,; \qquad \m_1\m_2=\frac{1}{2}\:; \qquad \det(\m_1\,,\m_2)=\frac{1}{2}\:.
\end{equation}
From here it follows that the pair $(\m_1\,,\m_2)$ has the same orientation, as the pair $(\n_1\,,\n_2)$.
 
 Conversely, if $(\m_1\,,\m_2)$ is an isotropic basis $N_p(\M)$, such that $\m_1\m_2=\frac{1}{2}$ and the
quadruple $(\vX_1,\vX_2,\m_1,\m_2)$ is a right orthonormal frame in $\RR^4_2$, then the vectors $\n_1$ and $\n_2$, 
given by:
\begin{equation}\label{n12_m12_R42-tl}
\n_1=\m_1 - \m_2\,; \qquad \n_2=\m_1 + \m_2\,,
\end{equation}
form an orthonormal basis of $N_p(\M)$, such that $\n_1^2=-1$\,.
Moreover, the quadruple $(\vX_1,\vX_2,\n_1,\n_2)$ is a right orthonormal frame in $\RR^4_2$.

 Suppose that $(\hat\m_1\,,\hat\m_2)$ is another isotropic basis, satisfying the conditions \eqref{m12_psd_ort_R42-tl}.
Then $\hat\m_1$ is collinear with $\m_1$ or $\m_2$. If $\hat\m_1 = k \m_2$, then $\hat\m_2 = l \m_1$.
It follows from $\det(\hat\m_1\,,\hat\m_2) = -kl\det(\m_1\,,\m_2)$ that $kl<0$\,.
But from $\hat\m_1 \hat\m_2 = kl\,\m_1\m_2$ it follows that $kl>0$\,, which is a contradiction.
Consequently, $\hat\m_1 = k \m_1$ and $\hat\m_2 = l \m_2$.
The equalities $\hat\m_1 \hat\m_2 = kl\,\m_1\m_2$ and $\hat\m_1 \hat\m_2 = \frac{1}{2} = \m_1\m_2$ imply that: $kl=1$\,.
Hence, if $(\m_1\,,\m_2)$ and $(\hat\m_1\,,\hat\m_2)$ are two isotropic bases, satisfying \eqref{m12_psd_ort_R42-tl},
then they are related by the following equalities: 
\begin{equation}\label{hat_m12_R42-tl}
\hat\m_1 = k\m_1\,, \qquad \hat\m_2 = \frac{1}{k}\m_2\,; \qquad\quad k\in\RR, \  k\neq 0\,.
\end{equation}

\medskip

  Suppose that $\Phi^{\prime \bot}$ has the following representation with respect to the basis $(\m_1\,,\m_2)$:
\begin{equation}\label{Phi_pn_m12_R42-tl}
\Phi^{\prime \bot} = c_1\m_1 + c_2\m_2\,; \qquad c_1,c_2 \in \DD \,.
\end{equation}
Then we get from \eqref{K_Phi-tl} the following expression for $K$:
\[
\begin{array}{rl}
K &=\ -\ds \frac{{\|\Phi^{\prime \bot}\|}^2}{E^2} = -\frac{(c_1\m_1 + c_2\m_2)(\bar c_1\m_1 + \bar c_2\m_2)}{E^2} =
-\frac{(c_1\bar c_2 + \bar c_1 c_2) \m_1\m_2}{E^2}\\[2ex]
  &=\ -\ds\frac{c_1\bar c_2 + \overline{c_1\bar c_2}}{2E^2} = -\frac{\Re(c_1\bar c_2)}{E^2}\:.
\end{array}
\]
Analogously, we get from \eqref{kappa_Phi_pn-tl} the following expression for $\varkappa$:
\[
\begin{array}{rl}
\varkappa &=\ -\ds \frac{\jj\,\det \big(\Phi^{\prime \bot}\,,\overline{\Phi^{\prime \bot}}\,\big)}{E^2} = 
-\frac{\jj\,\det (c_1\m_1 + c_2\m_2 \,, \bar c_1\m_1 + \bar c_2\m_2)}{E^2}\\[2ex]
&=\ -\ds\frac{\jj (c_1\bar c_2 - \bar c_1 c_2) \det(\m_1\,,\m_2)}{E^2} =
    -\frac{\jj (c_1\bar c_2 - \overline{c_1\bar c_2}\,)}{2E^2} = -\frac{\Im(c_1\bar c_2)}{E^2}\:.
\end{array}
\]
Consequently the curvatures $K$ and $\varkappa$ are given by $c_1$ and $c_2$ as follows:
\begin{equation}\label{K_kappa_c12_R42-tl}
K= -\frac{\Re(c_1\bar c_2)}{E^2} \:; \quad\quad
\varkappa = -\frac{\Im(c_1\bar c_2)}{E^2}\:.
\end{equation}

\medskip

 At the end of this section we establish how the curvatures $K$ and $\varkappa$ are transformed under a change 
of the coordinates and under the basic geometric transformations of the minimal time-like surface $\M$ in $\RR^4_2$.
As a basic analytic tool we shall use formulas \eqref{K_kappa_Phi_R42-tl}.
First, let $t=t(s)$ give a change of the coordinates. Since $K$ and $\varkappa$ are scalar invariants, then:
\begin{equation}\label{K_kappa_s_R42-tl}
K(s)=K(t(s))\,; \qquad \varkappa (s) = \varkappa (t(s))\,.
\end{equation}

\smallskip

 Next let $\M=(\D,\x)$ and $\hat\M=(\D,\hat\x)$ be two minimal time-like surfaces in $\RR^4_2$ and suppose that, 
$\hat\M$ is obtained from $\M$ through a motion (possibly improper) in $\RR^4_2$ by the formula:
\begin{equation*}
\hat\x(t)=A\x(t)+\vb\,; \qquad A \in \mathbf{O}(2,2,\RR), \ \vb \in \RR^4_2 \,.
\end{equation*}
Differentiating the above equality, we find that the corresponding functions $\Phi$ and $\hat\Phi$ satisfy the equalities:
\begin{equation}\label{hat_Phi_pn-Phi_pn-mov_R42-tl}
 \hat\Phi(t)=A\Phi(t)\,; \qquad \hat\Phi'(t)=A\Phi'(t)\,; \qquad \hat\Phi^{\prime \bot}(t) = A\Phi^{\prime \bot}(t)\,.
\end{equation}
In order to obtain the last equality, we use that the subspaces $T_p(\M)$ and $N_p(\M)$ are invariant under an 
arbitrary motion and therefore the orthogonal projections of any vector into these subspaces are also invariant.
Since any motion is an isometry between $\M$ and $\hat\M$, then $\hat K = K$. The last equality can also be obtained
by a direct computation using formulas \eqref{K_kappa_Phi_R42-tl} and \eqref{hat_Phi_pn-Phi_pn-mov_R42-tl}. 
The situation with the normal curvature $\varkappa$ is different. 
For the determinant in formula \eqref{K_kappa_Phi_R42-tl} we have:
\[
\det \big(A\Phi,A\bar\Phi,A\Phi^\prime,A\overline{\Phi^\prime}\,\big)=
\det A \cdot \det \big(\Phi,\bar\Phi,\Phi^\prime,\overline{\Phi^\prime}\,\big)\,.
\]
From here it follows that under a proper motion ($\det A = 1$) $\varkappa$ is preserved, 
while under an improper motion ($\det A = -1$) $\varkappa$ changes the sign.
Thus we have:
\begin{equation}\label{hat_K_kappa-K_kappa-mov_R42-tl}
\hat K(t)=K(t)\,; \qquad \hat\varkappa(t) = \varepsilon\varkappa(t) \,;
\qquad\quad  \varepsilon = \det A\,.
\end{equation}

\smallskip

 Let us consider the case, when $\hat\M$ is obtained from $\M$ through an anti-isometry in $\RR^4_2$, which 
according to \eqref{hat_M-M-antimov-tl}, is given by a formula of the type: 
\begin{equation*}
\hat\x(s)=A\x(\jj s)+\vb\,; \qquad A \in \mathbf{AO}(2,2,\RR), \ \vb \in \RR^4_2 \,.
\end{equation*}
Then, analogously to \eqref{hat_Phi_pn-Phi_pn-mov_R42-tl}, we get from \eqref{hat_Phi-Phi-antimov-tl}:
\begin{equation}\label{hat_Phi_pn-Phi_pn-antimov_R42-tl}
 \hat\Phi(s)=\jj A\Phi(\jj s)\,; \qquad \hat\Phi'(s)=A\Phi'(\jj s)\,; \qquad \hat\Phi^{\prime \bot}(s) = A\Phi^{\prime \bot}(\jj s)\,.
\end{equation}
Applying the last formula to \eqref{K_kappa_Phi_R42-tl}, we obtain $\hat K(s) = -K(\jj s)$.
For the determinant in the formula \eqref{K_kappa_Phi_R42-tl} we have:
\[
\det \big(\jj A\Phi(\jj s),\bar\jj A\bar\Phi(\jj s),A\Phi^\prime(\jj s),A\overline{\Phi^\prime}(\jj s)\big)=
-\det A \cdot \det \big(\Phi(\jj s),\bar\Phi(\jj s),\Phi^\prime(\jj s),\overline{\Phi^\prime}(\jj s)\big)\,.
\]
From here it follows that under an anti-isometry in $\RR^4_2$, the curvatures $\hat K$ and $\hat\varkappa$ change as follows:
\begin{equation}\label{hat_K_kappa-K_kappa-antimov_R42-tl}
\hat K(s)=-K(\jj s)\,; \qquad \hat\varkappa(s) = -\varepsilon\varkappa(\jj s) \,;
\qquad\quad  \varepsilon = \det A\,.
\end{equation}

\smallskip

Next we consider the case of a homothety in $\RR^4_2$. If the minimal time-like surfaces $\M=(\D,\x)$ and
$\hat\M=(\D,\hat\x)$ are related by an equality of the type $\hat\x(t)=k\x(t)$, then it follows that  
$\hat\Phi=k\Phi$;\ \ $\hat\Phi'=k\Phi'$\ \ and\ \ $\hat\Phi^{\prime \bot} = k\Phi^{\prime \bot}$.
Applying the last equalities to \eqref{K_kappa_Phi_R42-tl}, we find that the curvatures $K$ and $\varkappa$
are transformed in the following way:
\begin{equation}\label{hat_K_kappa-K_kappa-homotet_R42}
\hat K(t)=\frac{1}{k^2}K(t)\,; \qquad \hat\varkappa(t) = \frac{1}{k^2}\varkappa(t) \,.
\end{equation}

\smallskip

 Now let $\bar\M$ be the conjugate surface of the minimal time-like surface $\M$, according to Definition \ref{Conj_Min_Surf-tl}\,.
Then, for the corresponding functions $\hat\Phi$ and $\Phi$ we have from \eqref{Phi_conj-tl}:
\begin{equation}\label{hat_Phi_pn-Phi_pn-conj_R42-tl}
 \hat\Phi(s)=\Phi(\jj s)\,; \qquad \hat\Phi'(s)=\jj\Phi'(\jj s)\,; \qquad \hat\Phi^{\prime \bot}(s) = \jj\Phi^{\prime \bot}(\jj s)\,.
\end{equation}
Applying the last formula to \eqref{K_kappa_Phi_R42-tl}, we get $\hat K(s) = -K(\jj s)$.
For the determinant in the formula \eqref{K_kappa_Phi_R42-tl} we have:
\[
\det \big(\Phi(\jj s),\bar\Phi(\jj s),\jj\,\Phi^\prime(\jj s),\bar\jj\,\overline{\Phi^\prime}(\jj s)\big)=
- \det \big(\Phi(\jj s),\bar\Phi(\jj s),\Phi^\prime(\jj s),\overline{\Phi^\prime}(\jj s)\big)\,.
\]
Hence, the curvatures $\hat K$ and $\hat\varkappa$ of the conjugate surface satisfy the equalities:
\begin{equation}\label{hat_K_kappa-K_kappa-conj_R42-tl}
\hat K(s)=-K(\jj s)\,; \qquad \hat\varkappa(s) = -\varkappa(\jj s) \,.
\end{equation}

\smallskip

 Finally we consider the one-parameter family $\M_\theta$, given by Definition~\ref{1-param_family_assoc_surf-tl}\,, of  
minimal time-like surfaces, associated to a given minimal time-like surface $\M$.
From  \eqref{Phi_1-param_family-tl} we have:
$\Phi_\theta(t) = \e^{\jj\theta}\Phi(t)$; \ \ $\Phi'_\theta(t) = \e^{\jj\theta}\Phi'(t)$\ \ and\ \ 
$\Phi'^\bot_\theta(t) = \e^{\jj\theta}\Phi'^\bot(t)$.
According to Proposition~\ref{Isom_M_theta-M-tl}\,, the map $\mathcal{F}_\theta: \x(t)\rightarrow \x_\theta(t)$
is an isometry. Therefore the Gauss curvature is invariant: $K_\theta(t)=K(t)$. 
The last equality can be obtained by direct calculations using formulas \eqref{K_kappa_Phi_R42-tl} and
the obtained equalities for $\Phi_\theta$ and $\Phi'^\bot_\theta$. 

Applying the formulas for $\Phi_\theta$ and $\Phi'_\theta$ to \eqref{K_kappa_Phi_R42-tl}, we find that the curvature
$\varkappa$ is also invariant under $\mathcal{F}_\theta$.
Thus, for the one-parameter family $\M_\theta$ of associated minimal time-like surfaces we have:
\begin{equation}\label{K_kappa-1-param_family_R42-tl}
K_\theta(t)=K(t)\,; \qquad \varkappa_\theta(t) = \varkappa(t) \,.
\end{equation}


\section{Canonical coordinates on a minimal time\,-\,like surface in~$\RR^4_2$.}\label{sect_can-def_R42-tl}

 In the previous considerations of the minimal time-like surfaces in $\RR^4_2$ we used arbitrary isothermal coordinates. 
It is known that in the case of $\RR^3_1$, the minimal time-like surfaces admit special geometric isothermal coordinates.
For example in \cite {G} it is shown that any minimal time-like surface in $\RR^3_1$ with negative Gauss curvature $K<0$\, 
admits local coordinates which are at the same time \emph{isothermal} and \emph{principal}. 
Using the standard denotations for the coefficients of the first and the second fundamental form $E, F, G$; $L, M, N$, 
this means: $E=-G$, $F=0$ and $M=0$\,. Furthermore, these coordinates can be normalized in such a way that $L=N=1$\,. 
These properties of the local coordinates determine them uniquely up to orientation of the coordinate lines.
Further, we call such kind of coordinates \emph{canonical coordinates}. In the case of a positive Gauss curvature $K>0$,
it is shown that any minimal time-like surface admits local coordinates, which are at the same time \emph{isothermal} and 
\emph{asymptotic}. These coordinates are characterized by the conditions $L=N=0$ and $M=\pm 1$\,.

 In \cite{G-M-1} it is proved that any minimal time-like surface of a general class in $\RR^4_1$ carries locally
special isothermal coordinates $(u,v)$, which in our denotations are characterized by the conditions:
\begin{equation}\label{Can_GM1-tl}
\begin{array}{ll}
-\x_u^2=\x_v^2>0\,, &
\sigma^2(\x_u,\x_u)+\sigma^2(\x_u,\x_v)=1\,,\\
\x_u \cdot \x_v = 0\,, \qquad &
\sigma(\x_u,\x_u) \cdot \sigma(\x_u,\x_v)=0\,.
\end{array}
\end{equation}
These properties of the local coordinates determine them uniquely up to orientation of the coordinate lines and
that is why further we call them \emph{canonical coordinates} for the minimal time-like surfaces in $\RR^4_1$.

 In the case of $\RR^4_2$, in \cite{S-2} it is shown that under the condition $K^2-\varkappa^2>0$\,,
there exist local coordinates, which also satisfy the properties \eqref{Can_GM1-tl}. In the same work, it is 
noted that under the condition $K^2-\varkappa^2<0$\,, there do not exist isothermal coordinates, satisfying
\eqref{Can_GM1-tl}.

\smallskip

 Further, we show how the conditions for the coordinates to be canonical, can be expressed in terms of the function
$\Phi$, given by \eqref{Phi_def-tl}.
For that purpose, we consider formulas \eqref{PhiPr-tl}. Taking the scalar square of the second equality, we get:
\begin{equation}\label{PhiPr_bot^2-tl}
{\Phi^{\prime \bot}}^2=\sigma^2(\x_u,\x_u)+\sigma^2(\x_u,\x_v)+2\jj\,\sigma(\x_u,\x_u)\cdot\sigma(\x_u,\x_v)\,.
\end{equation}
If $\M$ is a minimal time-like surface in $\RR^3_1$ parametrized by principal canonical coordinates, then the
conditions $M=0$ and $L=1$ mean that $\sigma(\x_u,\x_v)=0$ and $\sigma^2(\x_u,\x_u)=1$\,. Now it follows from
\eqref{PhiPr_bot^2-tl} that ${\Phi^{\prime \bot}}^2=1$\,.\, If $\M$ is parametrized by asymptotic canonical 
coordinates, then the conditions $M=\pm 1$ and $L=0$ mean that $\sigma^2(\x_u,\x_v)=1$ and $\sigma(\x_u,\x_u)=0$\,,
which gives again ${\Phi^{\prime \bot}}^2=1$\,.

If $\M$ is a minimal time-like surface in $\RR^4_1$, parametrized by canonical coordinates, then the first two conditions
of \eqref{Can_GM1-tl} give that the coordinates are isothermal, while the third and the fourth condition give again 
${\Phi^{\prime \bot}}^2=1$\,, according to \eqref{PhiPr_bot^2-tl}.

\smallskip

 The above observations show that the condition for the coordinates to be canonical is a condition for ${\Phi'^\bot}^2$. 
That is why we find the formulas which ${\Phi'^\bot}^2$ satisfy under a motion of the surface in $\RR^4_2$ and under
a change of the isothermal coordinates.

Suppose that the surface $\hat\M$ is obtained from $\M$ through a motion $A$ (possibly improper) in $\RR^4_2$.
Taking into account \eqref{hat_Phi_pn-Phi_pn-mov_R42-tl}, we take a square of the last formula and get:
\begin{equation}\label{hat_Phi_pn2-Phi_pn2-mov_R42-tl}
\left.\hat\Phi'^\bot\right.^2 \!\! (t)=\left.\Phi'^\bot\right.^2 \!\! (t)\,. 
\end{equation}
This means that ${\Phi'^\bot}^2$ is invariant under a motion in $\RR^4_2$.

 Next we consider a change of the isothermal coordinates. Let $(u,v)$ be arbitrary isothermal coordinates on
the minimal time-like surface $\M$. Denoting $t=u+\jj v$ we make the change $t=t(s)$, where $s\in\DD$ is a
new variable, giving new isothermal coordinates. Denote by $\tilde\Phi(s)$ the function, corresponding to the new
coordinates $s$. The change of the isothermal coordinates generates either a holomorphic or an anti-holomorphic map.

First we consider the holomorphic case. Then, applying formula \eqref{Phi_s-hol-tl} we have $\tilde\Phi=\Phi t'$, 
from where $\tilde\Phi'_s=\Phi'_t t'^{\,2}+\Phi t'\,\!'$. Since the vector $\Phi$ is tangent to $\M$, 
then $\Phi^\bot=0$ and consequently we obtain:
\begin{equation}\label{tildPhiPr2-tl}
\tilde\Phi_s'^\bot = \Phi_t'^\bot \,t^{\prime \, 2}\,; \qquad 
\left.\tilde\Phi_s^{\prime\bot}\right.^2 = {\Phi_t^{\prime \bot}}^2 t'^{\,4}\,.
\end{equation}

The case of an arbitrary anti-holomorphic function is reduced to a holomorphic function and the special case
$t=\bar s$. That is why we consider the last equality. According to
\eqref{Phi_s-t_bs-tl} we have: $\tilde\Phi(s)=\bar\Phi(\bar s)$, from where 
$\tilde\Phi_s'(s)=\overline{\Phi_t'}(\bar s)$.
Therefore: 
\begin{equation}\label{tildPhiPr2-t_bs-tl}
\tilde\Phi_s'^\bot(s) = \overline{\Phi_t'^\bot(\bar s)}\:; \qquad
\left.\tilde\Phi_s^{\prime\bot}\right.^2\!(s) = \overline{{\Phi_t^{\prime \bot}}^2 (\bar s)}\:.
\end{equation}

	Let ${\Phi_t^{\prime\bot}}^2\in\DD_0$, where $\DD_0$ is the set defined by \eqref{D0}. Since the set $\DD_0$
is closed with respect to both: multiplication with a number from $\DD$ or complex conjugation in $\DD$, it 
follows easily from \eqref{tildPhiPr2-tl} and \eqref{tildPhiPr2-t_bs-tl} that 
$\left.\tilde\Phi_s^{\prime\bot}\right.^2\in\DD_0$ as well. Therefore the condition ${\Phi^{\prime\bot}}^2=1$ 
can not be fulfilled under any isothermal coordinates. This means that the points in which ${\Phi^{\prime\bot}}^2\in\DD_0$, 
have to be conidered separately.

We give the following
\begin{dfn}\label{DegP-def-tl}
If $\M$ is a minimal time-like surface in $\RR^4_2$, pametrized by isothermal coordinates $(u,v)$, 
the point $p$ in $\M$ is said to be a \textbf{degenerate point}, if the function $\Phi$, given by \eqref{Phi_def-tl}, 
satisfies the condition ${\Phi^{\prime \bot}}^2(p)\in\DD_0$.
\end{dfn}

 The above definition is analytic, but it follows easily that the property of a point to be degenerate is geometric. 
We have the following statement:
\begin{theorem}\label{DegP_Change_Move-tl}
 Let $\M$ be a minimal time-like surface in $\RR^4_2$. Then the property of a point to be degenerate does not 
depend on the choice of the isothermal coordinates and it is invariant under a motion of $\M$ in $\RR^4_2$.
\end{theorem}
\begin{proof}
 The independance of the property from the choice of the isothermal coordinates, as we noted earlier, is a direct 
consequence of \eqref{tildPhiPr2-tl} and \eqref{tildPhiPr2-t_bs-tl}. The invariance of the property under a motion in 
$\RR^4_2$ follows from \eqref{hat_Phi_pn2-Phi_pn2-mov_R42-tl}. 
\end{proof}

 We have already seen that around a degenerate point one can not introduce canonical coordinates and
that is why we shall consider minimal time-like surfaces free of degenerate points. 
We give the following
\begin{dfn}\label{Min_Surf_Gen_Typ-def-tl}
A minimal time-like surface $\M$ in $\RR^4_2$ is said to be of \textbf{general type}, if it is free of degenerate points.
\end{dfn}

 If $\M$ is of general type, then we have ${\Phi'^\bot}^2\notin\DD_0$ and as we noted considering the exponential 
form \eqref{t-exp} of the double numbers, there exists a unique number $\varepsilon=\pm 1;\,\pm\jj$\,, such that
$\varepsilon{\Phi'^\bot}^2\in\DD_+$. The value of $\varepsilon$ depends on the fact in which quadrant with respect 
to the null basis in $\DD$, given by \eqref{qq-def}, is the value of ${\Phi'^\bot}^2$.
From the above it follows that the minimal time-like surfaces of general type in $\RR^4_2$ are divided into 
different types depending on the value of $\varepsilon$, i.e. depending on the quadrant, where the value of   
${\Phi'^\bot}^2$ lies. 

We give the following 
\begin{dfn}\label{Min_Surf_kind123-def-tl}
 Let $\M$ be a minimal time-like surface of general type in $\RR^4_2$.
The surface $\M$ is said to be:

{
\renewcommand{\baselinestretch}{1.40}
\selectfont
\begin{itemize}
	\item of \textbf{the first type}, if ${\Phi^{\prime\bot}}^2 \in \DD_+$\,, 
	i.e. $\big|{\Phi^{\prime\bot}}^2\big|^2 > 0$ and $\Re \big({\Phi^{\prime\bot}}^2\big) > 0$\,;
	\item of \textbf{the second type}, if $-{\Phi^{\prime\bot}}^2 \in \DD_+$\,, 
	i.e. $\big|{\Phi^{\prime\bot}}^2\big|^2 > 0$ and $\Re \big({\Phi^{\prime\bot}}^2\big) < 0$\,;
	\item of \textbf{the third type}, if $\pm\jj {\Phi^{\prime\bot}}^2 \in \DD_+$\,, 
	i.e. $\big|{\Phi^{\prime\bot}}^2\big|^2 < 0$\,.
\end{itemize}

}
\end{dfn}

 The above definition is analytic, but the next theorem shows that the type of the minimal time-like surface
is a geometric property.
\begin{theorem}\label{kind123_Change_Move-tl}
 Let $\M$ be a minimal time-like surface of general type in $\RR^4_2$. 
The type of $\M$does not depend on the isothermal coordinates and is invariant under a motion of $\M$ in $\RR^4_2$.
\end{theorem}
\begin{proof}
First we consider a holomorphic change of the isothermal coordinates. Then equalities \eqref{tildPhiPr2-tl} are valid.
The square of any number in $\DD\backslash\DD_0$ belongs to $\DD_+$ and therefore $t'^{\,4}\in\DD_+$.
If $\varepsilon{\Phi'^\bot}^2\in\DD_+$, then also $\varepsilon{\Phi'^\bot}^2 t'^{\,4}\in\DD_+$,
since the product of two numbers in $\DD_+$ is also in $\DD_+$. Taking into account \eqref{tildPhiPr2-tl}, it follows 
that the type of the surfaces is preserved under a holomorphic change of the coordinates.

Next we consider an anti-holomorphic change $t=\bar s$ of the isothermal coordinates. Then equalities 
\eqref{tildPhiPr2-t_bs-tl} are valid.
If $\varepsilon{\Phi'^\bot}^2\in\DD_+$, then also $\bar\varepsilon\overline{{\Phi_t^{\prime \bot}}^2 (\bar s)}\in\DD_+$,
since $\DD_+$ is closed with respect to complex conjugation. If $\M$ is of the first or the second type, then
we have $\varepsilon=\pm 1$\,, from where $\bar\varepsilon=\varepsilon$ and therefore the type of $\M$ is preserved.
If $\M$ is of the third type, then we have $\varepsilon=\pm\jj$\,, from where $\bar\varepsilon=\mp\jj$
and consequently the type of $\M$ is also preserved.

 Finally the fact that the type of $\M$ is invariant under a motion in $\RR^4_2$ follows from 
\eqref{hat_Phi_pn2-Phi_pn2-mov_R42-tl}. 
\end{proof}
\begin{remark}
 The surfaces of the third type can not be divided into two separate classes depending on the sign of
$\Im \big({\Phi^{\prime\bot}}^2\big)$, as in the case of surfaces of the first or the second type, which differ 
in the sign of $\Re \big({\Phi^{\prime\bot}}^2\big)$. It is this way, because as it is seen from the proof 
of the above theorem (or from formula \eqref{tildPhiPr2-t_bs-tl}), the sign of $\Im \big({\Phi^{\prime\bot}}^2\big)$ 
converts under a change of the coordinates of the type $t=\bar s$.
\end{remark}

Now we shall consider the relationship between the degenerate points and the type of the surfaces from the one hand, and 
the Gauss curvature $K$ and the normal curvature $\varkappa$ from the other. We have the following statement:
\begin{theorem}\label{DegP_kind123-K_kappa-tl}
 Let $\M$ be a minimal time-like surface in $\RR^4_2$ and $p$ be a point on $\M$. 
Then:
\begin{enumerate}
	\item The point $p$ is degenerate if and only if $K^2(p)-\varkappa^2(p)=0$\,. 
	\item The surface $\M$ is of the first or the second type if and only if $K^2-\varkappa^2>0$\,. 
	\item The surface $\M$ is of the third type if and only if $K^2-\varkappa^2<0$\,. 
\end{enumerate}
\end{theorem}
\begin{proof}
 To prove the theorem, we shall express $\big|{\Phi^{\prime\bot}}^2\big|^2$ through the curvatures $K$ and $\varkappa$.
For that purpose, we use the representation \eqref{Phi_pn_m12_R42-tl} of $\Phi^{\prime \bot}$ in the isotropic basis,
introduced by \eqref{m12_n12_R42-tl}. Taking a square of the equality \eqref{Phi_pn_m12_R42-tl}, we get:
\begin{equation}\label{Phi_pn2_m12_R42-tl}
{\Phi^{\prime\bot}}^2 = 2c_1c_2\m_1\m_2\ = c_1c_2 \,; \qquad \quad
\Big|{\Phi^{\prime\bot}}^2\Big|^2 = |c_1c_2|^2 .
\end{equation}
On the other hand, equalities \eqref{K_kappa_c12_R42-tl} for $K$ and $\varkappa$ give that:
\begin{equation*}
K^2-\varkappa^2 = \frac{\Re^2(c_1\bar c_2)-\Im^2(c_1\bar c_2)}{E^4} = \frac{|c_1 \bar c_2|^2}{E^4}
= \frac{|c_1c_2|^2}{E^4}\:.
\end{equation*}
Consequently
\begin{equation}\label{Phi_pn2_E_K_kappa_R42-tl}
\Big|{\Phi^{\prime\bot}}^2\Big|^2 = E^4(K^2-\varkappa^2)\,.
\end{equation}
Now, the three equivalences of the theorem follow immediately from the above equality.
\end{proof}
\begin{remark}
The surfaces from the first and second type can not be distinguished solely by their invariants $K$ and $\varkappa$.
Further, with the aid of Theorem \ref{Can_antimov_conj_R42-tl} we shall see that to any surface of the first type corresponds a surface of the second type and both surfaces are isometric. In this isometry the curvatures $K$ and $\varkappa$ 
remain invariant.
\end{remark}

 Let us now return to the question of the canonical coordinates. Theorem \ref{kind123_Change_Move-tl} implies that
the equality ${\Phi'^{\bot}}^2=1$ is only possible for surfaces of the first type. For the surfaces of the second 
type it is natural to require the equality ${\Phi'^{\bot}}^2=-1$, while for the surfaces of the third type it is 
natural to require ${\Phi'^{\bot}}^2=\pm\jj$. Moreover, for the surfaces of the third type it is sufficient to 
require only the case $+\jj$, since the case $-\jj$ is reduced to the previous through a change of the coordinates 
of the type $t=\bar s$. It is clear, that the three formulas can be combined into one formula of the type 
${\Phi'^{\bot}}^2=\varepsilon$, where $\varepsilon=\pm1;\,\jj$ depending on the type of the surface. 
So we come to the following definition:
\begin{dfn}\label{Can-def_R42-tl}
Let $\M$ be a minimal time-like surface of general type in $\RR^4_2$, parametrized by isothermal coordinates $(u,v)$, 
such that $E<0$. The coordinates $(u, v)$ on $\M$ are said to be \textbf{canonical}, if the function $\Phi$, 
defined by \eqref{Phi_def-tl}, satisfies the condition: 
\begin{equation}\label{Can-Phi_R42-tl}
{\Phi'^\bot}^2=\varepsilon\,,
\end{equation}
where $\varepsilon=1$ for the surfaces of the first type, $\varepsilon=-1$ for the surfaces of the second type\, and
$\varepsilon=\jj$ for the surfaces of the third type.
\end{dfn}
This definition is analytic, but it  follows from \eqref{hat_Phi_pn2-Phi_pn2-mov_R42-tl} that the canonical coordinates
are geometrically related to the surface under consideration. The following statement is valid:
\begin{theorem}\label{Can_Move-tl}
 Let the minimal time-like surface $\hat\M$ be obtained from the minimal time-like surface $\M$ through 
a motion in $\RR^4_2$. If $(u,v)$ are canonical coordinates for $\M$, then they are also canonical for $\hat\M$.
\end{theorem}

 With the help of the equality \eqref{PhiPr_bot^2-tl} we can characterize the canonical coordinates through 
the second fundamental form $\sigma$:
\begin{prop}\label{Can-sigma_R42-tl}
 Let $\M$ be a time-like surface of general type in $\RR^4_2$, parametrized by isothermal coordinates $(u,v)$, 
such that $E<0$\,. 

If $\M$ is of \textbf{the first or the second type}, then these coordinates are canonical if and only if
the second fundamental form $\sigma$ satisfies the conditions:
\begin{equation}\label{sigma_uv_can12_R42-tl}
\sigma (\x_u,\x_u)\bot \: \sigma(\x_u,\x_v)\,,\qquad \sigma^2(\x_u,\x_u)+\sigma^2(\x_u,\x_v)=\pm 1\,,
\end{equation}
where the sign "$+$" is for the surfaces of the first type, and the sign "$-$" is for the surfaces 
of the second type.

 If $\M$ is of \textbf{the third type}, then these coordinates are canonical if and only if the second 
fundamental form $\sigma$ satisfies the conditions:
\begin{equation}\label{sigma_uv_can3_R42-tl}
2\sigma (\x_u,\x_u)\sigma(\x_u,\x_v)=1\,,\qquad \sigma^2(\x_u,\x_u)+\sigma^2(\x_u,\x_v)=0\,.
\end{equation}
\end{prop}

 As we noted earlier, the function $\Phi^{\prime \bot}$ is not holomorphic in general as a consequence of the fact 
that the orthogonal projection does not preserve the holomorphy. 
However, it turns out that the scalar square ${\Phi^{\prime \bot}}^2$ is always a holomorphic function. In order 
to prove that, we consider equalities \eqref{Phipn-tl}. Squaring the second of them, we find:
\[
{\Phi^{\prime\bot}}^2 =
{\Phi'}^2-2\Phi'\ds\frac{\Phi' \cdot \bar \Phi}{\|\Phi\|^2}\Phi+\left(\ds\frac{\Phi' \cdot \bar \Phi}{\|\Phi\|^2}\right)^2 \Phi^2.
\]
Taking into account that $\Phi^2=0$\; and\; $\Phi\cdot\Phi^\prime=0$\; we get:
\begin{equation}\label{Phi_prim_bot^2-Phi_prim^2-tl}
{\Phi^{\prime \bot}}^2={\Phi^\prime}^2.
\end{equation}
It follows from here that ${\Phi^{\prime \bot}}^2$ is a holomorphic function. 

\medskip

 We will now clarify the question of existence of canonical coordinates. 

\begin{theorem}\label{Can_Coord-exist_R42-tl}
If $\M$ is a minimal time-like surface in $\RR^4_2$ of general type, then it admits locally canonical coordinates.
\end{theorem}
\begin{proof}
  Let the surface $\M$ be parametrized by isothermal coordinates $(u,v)$ and let $t=u+\jj v \in\DD$. 
Next we introduce isothermal coordinates through the formula $t=t(s)$, where $t(s)$ is a holomorphic function,
such that the coordinates determined by the new variable $s\in\DD$, are canonical. According to Definition 
\ref{Can-def_R42-tl} and formula \eqref{tildPhiPr2-tl}, the new coordinates are canonical if and only if: 
\begin{equation}\label{condcan_R42-tl}
{\Phi_t^{\prime \bot}}^2 t'^{\,4} = \left.\tilde\Phi_s^{\prime\bot}\right.^2 = \varepsilon \,,
\end{equation}
where $\varepsilon=\pm1;\,\jj$ depending on the type of the surface. The last condition is equivalent to 
\begin{equation}\label{condcan2_R42-tl}
\varepsilon\, {\Phi_t^{\prime \bot}}^2 t'^{\,4} =  1 \,.
\end{equation}
It follows from Definition \ref{Min_Surf_kind123-def-tl} that: $\varepsilon\,{\Phi^{\prime \bot}}^2 \in \DD_+$, 
which means that $\sqrt[4]{{\varepsilon\,\Phi_t^{\prime \bot}}^2} \in \DD_+$ is a well defined smooth function.
Taking a square root in \eqref{condcan2_R42-tl}, we find the following complex (over $\D$) ODE of the first order
for $t(s)$:
\begin{equation*}
\sqrt[4]{{\varepsilon\,\Phi_t^{\prime \bot}}^2}\:dt = ds\,.
\end{equation*}
In view of \eqref{Phi_prim_bot^2-Phi_prim^2-tl}, we can replace
${\Phi_t^{\prime \bot}}^2$ with ${\Phi_t^\prime}^2$ and then we have:
\begin{equation}\label{eqcan_R42-tl}
ds = \sqrt[4]{{\varepsilon\,\Phi_t^\prime}^2}\:dt\,.
\end{equation}
The last equation \eqref{eqcan_R42-tl} is an ODE with separable variables. Since the square root of the holomorphic 
function ${\Phi_t^\prime}^2$ is a holomorphic function in $\DD_+$, then $\ds\sqrt[4]{{\varepsilon\,\Phi_t^\prime}^2}$
is also a holomorphic function. Consequently, we obtain a solution to \eqref{eqcan_R42-tl} by an integration:
\begin{equation}\label{eqcan-sol_R42-tl}
s = \int\sqrt[4]{{\varepsilon\,\Phi_t^\prime}^2}\:dt\,.
\end{equation}
The last equality implies that $s'(t)=\sqrt[4]{{\varepsilon\,\Phi_t^{\prime\vphantom{\bot}}}^2} \in \DD_+$\,, 
and $|s'(t)|^2>0$.  Consequently \eqref{eqcan-sol_R42-tl} determines $s$ as a holomorphic and locally reversible
function of $t$. This means that $s$ determines locally new isothermal coordinates. 
The inequality $|s'(t)|^2>0$\,, also guarantees the condition $\tilde E(s)<0$ for the new coordinates, according to
\eqref{E_s-hol-tl}. Equation \eqref{eqcan_R42-tl} is equivalent to \eqref{condcan2_R42-tl} and \eqref{condcan_R42-tl}.
Therefore, the new coordinates are canonical. 
\end{proof}

We now move on to the question of the uniqueness of the canonical coordinates.
\begin{theorem}\label{Can_Coord-uniq_R42-tl}
Let $\M$ be a minimal time-like surface in $\RR^4_2$ of general type and let $t\in\DD$, $s\in\tilde\DD$ be two variables,
giving canonical coordinates in a neighborhood of a given point on $\M$.
If $\M$ is of the first or the second type, then $t$ and $s$ are related as follows:
\begin{equation}\label{uniq-kind12_R42-tl}
t=\pm s+c\,; \qquad\quad t=\pm \bar s + c\,.
\end{equation}
If $\M$ is of the third type, then $t$ and $s$ are related by an equality of the type:
\begin{equation}\label{uniq-kind3_R42-tl}
t=\pm s + c\,.
\end{equation}
In the above equalities $c\in\DD$ is a constant. 
\end{theorem}
\begin{proof}
Suppose the surface $\M$ is of the first or the second type. We first consider the case, when $t$ is a
holomorphic function of $s$. Then formula \eqref{tildPhiPr2-tl} is applicable and since by default $t$ and $s$
give canonical coordinates, we have: 
${\Phi_t^{\prime \bot}}^2 = \left.\tilde\Phi_s^{\prime\bot}\right.^2 = \varepsilon$\,, where $\varepsilon = \pm 1$\,.
Therefore \eqref{tildPhiPr2-tl} is reduced to $t'^4 = 1$, from where $t'(s) = \pm 1;\ \pm \jj$\,.
Equalities $t' = \pm \jj$ are dropped, since they give $|t'|^2 = -1 < 0$\,, which is impossible according to \eqref{E_s-hol-tl}.
Thus, it only remains $t'(s) = \pm 1$\,, which is equivalent to the first equality in \eqref{uniq-kind12_R42-tl}.

 Further, let $t$ be an anti-holomorphic function of $s$. Introducing an additional variable $r\in\DD$ through the equality
$r=\bar s$, then formula \eqref{tildPhiPr2-t_bs-tl} is applicable to $r$ and $s$ and then it follows that $r$ also gives
canonical coordinates on $\M$. Since $t$ is a holomorphic function of $r$, then it follows from the above that $t$ and $r$ 
satisfy the first equality in \eqref{uniq-kind12_R42-tl}. It follows from here that the second equality in 
\eqref{uniq-kind12_R42-tl} holds for $t$ and $s$.

 Let us now suppose that  $\M$ is of the third type.
If $t$ is a holomorphic function of $s$, then just as in the case of the surfaces of the first or the second type 
it is proved that the equality \eqref{uniq-kind3_R42-tl} holds. It is impossible for $t$ to be an anti-holomorphic function
of $s$, since it follows from \eqref{tildPhiPr2-tl} and \eqref{tildPhiPr2-t_bs-tl} that under an anti-holomorphic change of 
the coordinates the sign of the imaginary part of ${\Phi^{\prime \bot}}^2$ also changes. This contradicts to equalities
${\Phi_t^{\prime \bot}}^2 = \left.\tilde\Phi_s^{\prime\bot}\right.^2 = \jj$. Thus, it only remains \eqref{uniq-kind3_R42-tl}.
\end{proof}
\begin{remark}
In geometric terms the relations \eqref{uniq-kind12_R42-tl} and \eqref{uniq-kind3_R42-tl} mean that 
the canonical coordinates are unique up to an orientation of the coordinate lines.
\end{remark}

\smallskip

 We will finally determine how the canonical coordinates and also the functions $K$ and $\varkappa$, considered 
as functions of these coordinates, are changed under a change of the coordinates and under the basic geometric 
transformations of the minimal time-like surface $\M$ in $\RR^4_2$. 

First we consider a change of the canonical coordinates. Since by default $K$ and $\varkappa$ are scalar invariants
under a change of the local coordinates, then the necessary formulas in this case are given by Theorem 
\ref{Can_Coord-uniq_R42-tl}\,. Thus, under a holomorphic change of the canonical coordinates the new functions
$\tilde K$ and $\tilde\varkappa$ in the three considered cases for the given surface are as follows: 
\begin{equation}\label{tilde_K_kappa-can_coord_ch_holo_R42-tl}
\tilde K(s)=K(\pm s+c)\,; \qquad \tilde\varkappa(s)=\varkappa(\pm s+c)\,.
\end{equation}
Respectively, under an anti-holomorphic change of the coordinates, which is only admissible for the surfaces
of the first or the second type, we have:
\begin{equation}\label{tilde_K_kappa-can_coord_ch_antiholo_R42-tl}
\tilde K(s)=K(\pm \bar s + c)\,; \qquad \tilde\varkappa(s)=\varkappa(\pm \bar s + c)\,.
\end{equation}

\smallskip

 Now suppose that $\hat\M$ is obtained from $\M$ through a motion of the type $\hat\x(t)=A\x(t)+\vb$;\, $A \in \mathbf{O}(2,2,\RR)$
and the variable $t$ gives canonical coordinates on $\M$. Theorems \ref{DegP_Change_Move-tl}\,, \ref{kind123_Change_Move-tl} 
and \ref{Can_Move-tl}\, guarantee that the type of the surface and the canonical coordinates are invariant when
applying a motion to the surface in $\RR^4_2$. Hence, transformation formulas for $K$ and $\varkappa$ are a special case
of \eqref{hat_K_kappa-K_kappa-mov_R42-tl} and have the form:
\begin{equation}\label{hat_K_kappa-can-mov_R42-tl}
\hat K(t)=K(t)\,; \qquad \hat\varkappa(t) = \varepsilon\varkappa(t) \,;
\qquad\quad  \varepsilon = \det A\,.
\end{equation}

\smallskip

 Now consider the case when $\hat\M$ is obtained from $\M$ through an anti-isometry of the type 
$\hat\x(t)=A\x(t)+\vb$;\, $A \in \mathbf{AO}(2,2,\RR)$. Taking a square in the last of the formulas
\eqref{hat_Phi_pn-Phi_pn-antimov_R42-tl}, we get:
\begin{equation}\label{hat_Phi_pn2-Phi_pn2-antimov_R42-tl}
\left.\hat\Phi'^\bot\right.^2 \!\! (s)=-\left.\Phi'^\bot\right.^2 \!\! (\jj s)\,. 
\end{equation}
It follows from here that the degenerate points are invariant under an anti-isometry in $\RR^4_2$.
Furthermore, if $\M$ is of general type, then $\hat\M$ is also of general type. We have:
\begin{theorem}\label{Can_antimov_R42-tl}
 Let the minimal time-like surface $\hat\M$ be obtained from $\M$ by an anti-isometry in $\RR^4_2$
and let $t\in\D\subset\DD$ give canonical coordinates on $\M$. Then we have:
\begin{enumerate}
	\item If $\M$ is of the first type, then $\hat\M$ is of the second type. The coordinates $s$, obtained by the
	formula $t=\jj s$, are canonical for $\hat\M$.
	\item If $\M$ is of the second type, then $\hat\M$ is of the first type. 
	The coordinates $s$, obtained by the formula $t=\jj s$, are canonical for $\hat\M$.
	\item If $\M$ is of the third type, then $\hat\M$ is also of the third type. 
	The coordinates $s$, obtained by the formula $t=\jj\bar s$, are canonical for $\hat\M$.
\end{enumerate}
\end{theorem}
\begin{proof}
 As we noted just before formula \eqref{hat_M-M-antimov-tl}, the variable $s$, determined by the
equality $t=\jj s$, gives isothermal coordinates on $\hat\M$, satisfying the condition $\hat E(s) < 0$\,. 
If $\M$ is of the first type, then $\left.\Phi'^\bot\right.^2 \!\! (\jj s)=1$\,. Equality 
\eqref{hat_Phi_pn2-Phi_pn2-antimov_R42-tl} implies that $\left.\hat\Phi'^\bot\right.^2 \!\! (s)=-1$\,.
This means that $\hat\M$ is of the second type with canonical coordinates $s$, which proves the first statement.
The second statement follows in a similar way. If $\M$ is of the third type, then 
$\left.\Phi'^\bot\right.^2 \!\! (\jj s)=\jj$\,. It follows from \eqref{hat_Phi_pn2-Phi_pn2-antimov_R42-tl} 
that $\left.\hat\Phi'^\bot\right.^2 \!\! (s)=-\jj$\,. This means that $\hat\M$ is also of the third type, but
$s$ does not determine canonical coordinates on $\hat\M$. Considering \eqref{tildPhiPr2-t_bs-tl} we see
that in order to obtain canonical coordinates on $\hat\M$, we need to change further $s$ with $\bar s$, which 
proves the third statement.
\end{proof}

 Applying the last theorem and formulas \eqref{hat_K_kappa-K_kappa-antimov_R42-tl}, we obtain the transformation 
formulas for $K$ and $\varkappa$ as functions of the canonical coordinates under an anti-isometry in $\RR^4_2$.
If $\M$ is of the first or the second type, then the necessary formulas are a special case of 
\eqref{hat_K_kappa-K_kappa-antimov_R42-tl} and have the form:
\begin{equation}\label{hat_K_kappa-can-antimov_kind12_R42-tl}
\hat K(s)=-K(\jj s)\,; \qquad \hat\varkappa(s) = -\varepsilon\varkappa(\jj s) \,;
\qquad\quad  \varepsilon = \det A\,.
\end{equation}
If $\M$ is of the third type, then according to Theorem \ref{Can_antimov_R42-tl}\,,
we need to change further $s$ by $\bar s$, and then the formulas get the form:
\begin{equation}\label{hat_K_kappa-can-antimov_kind3_R42-tl}
\hat K(s)=-K(\jj\bar s)\,; \qquad \hat\varkappa(s) = -\varepsilon\varkappa(\jj\bar s) \,;
\qquad\quad  \varepsilon = \det A\,.
\end{equation}

\begin{remark}
In Theorem \ref{Can_Coord-uniq_R42-tl} we proved that changing $s$ by $\bar s$, we obtain again canonical coordinates
on the surfaces of the first or the second type. This means that the change $t=\jj\bar s$ can be applied to all three 
types of surfaces under an anti-isometry in $\RR^4_2$. Then formula \eqref{hat_K_kappa-can-antimov_kind3_R42-tl} has 
to be also applied in all three cases.
\end{remark}

Now let us consider how the canonical coordinates are transformed in the case of a homothety in $\RR^4_2$.
\begin{theorem}\label{Can_Sim-tl}
Let $\M=(\D,\x)$ be a minimal time-like surface of general type in $\RR^4_2$ and let $t\in\D\subset\DD$ be a
variable giving canonical coordinates on $\M$. Let the surface $\hat\M=(\D,\hat\x)$ be obtained from $\M$ 
through a homothety with coefficient $k>0$ in $\RR^4_2$. Then $\hat\M$ is also of general type and is of the same type
as the surface $\M$. If we introduce the variable $s$ by the formula $t=\frac{1}{\sqrt{k}}s$, then $s$ gives
canonical coordinates on $\hat\M$.
\end{theorem}
\begin{proof}
 Let the homothety be given by the equality $\hat\x=k\x$. Then we have:  $\hat\Phi=k\Phi$, 
$\hat\Phi'_t=k\Phi'_t$ and respectively $\left.\hat\Phi_t^{\prime\bot}\right.^2 = k^2 {\Phi_t^{\prime \bot}}^2$.
Since by default $t$ gives canonical coordinates on $\M$, then the last equality reduces to
$\left.\hat\Phi_t^{\prime\bot}\right.^2 = k^2\varepsilon$, where $\varepsilon=\pm1;\,\jj$ in dependance on 
the type of the surface. It follows from here that $\hat\M$ is of general type and is of the same type
as $\M$. On the other hand, the condition $t=\frac{1}{\sqrt{k}}s$ gives that $t'^{\,4}=\frac{1}{k^2}$.
Now, applying \eqref{tildPhiPr2-tl} to $\hat\M$, we get: 
\[
\left.\tilde{\hat\Phi}_s^{\prime\bot}\right.^2 = \left.\hat\Phi_t^{\prime\bot}\right.^2 t'^{\,4} =
k^2 \varepsilon \ds\frac{1}{k^2} = \varepsilon\,.
\]
The last equality means that, $s$ gives canonical coordinates on $\hat\M$. 
\end{proof}
\begin{remark}
The formula $\left.\hat\Phi_t^{\prime\bot}\right.^2 = k^2 {\Phi_t^{\prime \bot}}^2$,
obtained in the proof of the last theorem, shows that the set of degenerate points is invariant under a homothety
(more generally: similarity) in $\RR^4_2$.
\end{remark}

 If $\hat\M$ is obtained from $\M$ through a homothety of the type $\hat\x=k\x$, then
\eqref{hat_K_kappa-K_kappa-homotet_R42} implies that 
$\hat K=\frac{1}{k^2}K$ and $\hat\varkappa = \frac{1}{k^2}\varkappa$. 
According to Theorem \ref{Can_Sim-tl}, the canonical coordinates $s$ on $\hat\M$ are related to the canonical 
coordinates $t$ on $\M$ by the formula $t=\frac{1}{\sqrt{k}}s$. Combining the last formulas and using that
$K$ and $\varkappa$ are scalar invariants under any change of the local coordinates, we obtain:
\begin{equation}\label{hat_K_kappa-homotet_R42-tl}
\hat K(s) = \frac{1}{k^2}\; K\!\left(\frac{1}{\sqrt{k}}s\right); \qquad
\hat\varkappa(s) = \frac{1}{k^2}\;\varkappa\!\left(\frac{1}{\sqrt{k}}s\right).  
\end{equation}

\smallskip

 Next we consider canonical coordinates on a surface which is conjugate to a given minimal time-like surface. 
\begin{theorem}\label{Can_Conj_Min_Surf-tl}
Let $\M=(\D,\x)$ be a minimal time-like surface of general type in $\RR^4_2$ and let $t\in\D\subset\DD$ be
a variable giving canonical coordinates on $\M$. Further, let $\bar\M=(\D,\y)$ be the minimal time-like surface
conjugate to the surface $\M$. Then the surface $\bar\M$ is also of general type and is of the same type as $\M$.
The variable $s$ introduced by the formula $t=\jj s$, determines canonical coordinates on $\bar\M$.
\end{theorem}
\begin{proof}
 From the notes just before Definition \ref{Conj_Min_Surf-tl} of a conjugate minimal time-like surface and from 
equalities \eqref{Phi_conj-tl} we see that the variable $s$ gives isothermal coordinates on $\bar\M$, 
satisfying the condition $\hat E(s) < 0$\,. 
Squaring the last of the equalities \eqref{hat_Phi_pn-Phi_pn-conj_R42-tl}, we get:
\begin{equation}\label{hat_Phi_pn2-Phi_pn2-conj_R42-tl}
\left.\hat\Phi_s^{\prime\bot}\right.^2\!\! (s) = {\Phi_t^{\prime \bot}}^2\! (\jj s)\,.
\end{equation}
Since $t$ gives canonical coordinates on $\M$, then it follows from the last equality that 
$\left.\hat\Phi_s^{\prime\bot}\right.^2\!\! (s) = \varepsilon$\,, where $\varepsilon=\pm1;\,\jj$ 
depending on the type of the surface. This means that $\bar\M$ is also of general type and is of 
the same type as the surface $\M$. Furthermore, $s$ gives canonical coordinates on $\bar\M$. 
\end{proof}
\begin{remark}
The equality $\left.\hat\Phi_s^{\prime\bot}\right.^2\!\! (s) = {\Phi_t^{\prime \bot}}^2\! (\jj s)$,
shows that the set of degenerate points is invariant by taking the conjugate surface in the following sense:
If $p$ is a degenerate point on $\M$, then the point $\mathcal{F}(p)$, corresponding to $p$ in the standard
anti-isometry $\mathcal{F}$ between $\M$ and $\bar\M$, defined in Proposition \ref{AntiIsom_conj_M-M-tl}\,, 
is also a degenerate point on $\bar\M$.
\end{remark}
 The last theorem and formulas \eqref{hat_K_kappa-K_kappa-conj_R42-tl} give transformation formulas
for the curvatures $\hat K$ and $\hat\varkappa$ of the conjugate surface as functions of the canonical 
coordinates . These formulas are a special case of \eqref{hat_K_kappa-K_kappa-conj_R42-tl} and have the form:
\begin{equation}\label{hat_K_kappa-can-conj_R42-tl}
\hat K(s)=-K(\jj s)\,; \qquad \hat\varkappa(s) = -\varkappa(\jj s) \,.
\end{equation}

\smallskip

 Finally we show how to obtain canonical coordinates on the surfaces of the one-parameter family 
of minimal time-like surfaces, associated to a given one.
\begin{theorem}\label{Can_1-param_family-tl}
Let $\M=(\D,\x)$ be a minimal time-like surface of general type in $\RR^4_2$ and let $t\in\D\subset\DD$ be a
variable, giving canonical coordinates on $\M$. Further, denote by $\M_\theta=(\D,\x_\theta)$ the one-parameter 
family of minimal surfaces, associated to $\M$. Then, for every $\theta\in\RR$\,, $\M_\theta$ is of general type and 
is of the same type as $\M$. The variable $s$, determined by the equality $t=\e^{-\jj\frac{\theta}{2}} s$, 
gives canonical coordinates on $\M_\theta$.
\end{theorem}
\begin{proof}
 We again apply formula \eqref{Phi_1-param_family-tl}, which gives 
$\Phi_\theta = \e^{\jj\theta}\Phi$. From here $\Phi'_{\theta|t} = \e^{\jj\theta}\Phi'_t$ and consequently 
$\left.\Phi_{\theta|t}^{\prime\bot}\right.^2 = \e^{2\jj\theta} {\Phi_t^{\prime \bot}}^2$.
Under the condition that $t$ gives canonical coordinates on $\M$, the last equality is reduced to
$\left.\Phi_{\theta|t}^{\prime\bot}\right.^2 = \e^{2\jj\theta}\varepsilon$, 
where $\varepsilon=\pm1;\,\jj$ depending on the type of the surface. 
It follows from here that $\M_\theta$ is also of general type and is of the same type as $\M$, since 
$\e^{2\jj\theta}\in\DD_+$. On the other hand, the condition $t=\e^{-\jj\frac{\theta}{2}} s$ gives 
$t'^{\,4}=\e^{-2\jj\theta}$. Now, applying \eqref{tildPhiPr2-tl} to $\M_\theta$, we get: 
\[
\left.\tilde\Phi_{\theta|s}^{\prime\bot}\right.^2 = \left.\Phi_{\theta|t}^{\prime\bot}\right.^2 t'^{\,4} =
\e^{2\jj\theta} \varepsilon \e^{-2\jj\theta} = \varepsilon\,.
\]
The last means that $s$ gives canonical coordinates on $\M_\theta$. 
\end{proof}
\begin{remark}
The formula $\left.\Phi_{\theta|t}^{\prime\bot}\right.^2 = \e^{2\jj\theta} {\Phi_t^{\prime \bot}}^2$,
obtained above shows that the set of degenerate points is invariant for the one-parameter family of 
associated minimal time-like surfaces in the following sense: If $p$ is a degenerate point on $\M$, 
then for any $\theta$, the point $\mathcal{F}_\theta(p)$, corresponding to $p$ under the standard 
isometry $\mathcal{F}_\theta$ between $\M$ and $\M_\theta$, is also a degenerate point on $\M_\theta$.
\end{remark}
 If $K_\theta$ and $\varkappa_\theta$ are the curvatures of the minimal time-like surface $\M_\theta$ 
associated to the given surface $\M$, then $K_\theta(t)=K(t)$ and $\varkappa_\theta(t) = \varkappa(t)$
according to \eqref{K_kappa-1-param_family_R42-tl}. Theorem \ref{Can_1-param_family-tl}\, gives that
the canonical coordinates $s$ on $\M_\theta$ and $t$ on $\M$ are related by the formula 
$t=\e^{-\jj\frac{\theta}{2}} s$. Combining the last formulas, we obtain:
\begin{equation}\label{K_kappa_1-param_family_can_R42-tl}
K_\theta(s) = K(\e^{-\jj\frac{\theta}{2}} s)\,; \qquad \varkappa_\theta(s) = \varkappa(\e^{-\jj\frac{\theta}{2}} s)\,.
\end{equation}

\medskip

 Let us return to the obtained formulas and properties in the cases of an anti-isometry and a conjugate surface.
Let $\M=(\D,\x(t))$ be a minimal time-like surface of general type in $\RR^4_2$ and let $\bar\M=(\D,\y(t))$ be 
its conjugate surface. Proposition \ref{AntiIsom_conj_M-M-tl} gives that the map $\x(t) \to y(t)$ is an
anti-isometry. Consider a third surface $\hat\M=(\D,\hat\x(t))$, obtained from $\bar\M$ through an anti-isometry
in $\RR^4_2$ of the type $\hat\x(t)=A\y(t)+b$, such that $\det A = 1$\,. Then the map $\x(t) \to \hat\x(t)$, as a
composition of two anti-isometries, is an isometry from $\M$ to $\hat\M$. But any isometry preserves the Gauss
curvature $K$. Further, formulas \eqref{hat_K_kappa-K_kappa-antimov_R42-tl}  and 
\eqref{hat_K_kappa-K_kappa-conj_R42-tl} imply that the curvature of the normal connection $\varkappa$ is also invariant.
On the other hand, formulas \eqref{hat_Phi_pn2-Phi_pn2-antimov_R42-tl} and \eqref{hat_Phi_pn2-Phi_pn2-conj_R42-tl} 
give that the function ${\Phi^{\prime \bot}}^2\! (t)$ changes the sign under this isometry.
The above observations give the following:
\begin{theorem}\label{Can_antimov_conj_R42-tl}
Let $\M=(\D,\x)$ be a minimal time-like surface of general type in $\RR^4_2$ parametrized by isothermal coordinates
$t\in\D\subset\DD$ and let the surface $\hat\M=(\D,\hat\x)$ be obtained from the conjugate surface $\bar\M=(\D,\y)$
of $\M$ through an anti-isometry in $\RR^4_2$ of the type:
\begin{equation}\label{hatM_M-antimov_conj_R42-tl}
\hat\x(t)=A\y(t)+b\,; \qquad\quad  A \in \mathbf{AO}(2,2,\RR)\,,\ \  \det A = 1\,, \ \  b \in \RR^4_2\,.
\end{equation}
Then the map $\x(t) \to \hat\x(t)$ is an isometry from $\M$ to $\hat\M$ and the following equalities hold:
\begin{equation}\label{E_K_kappa_Phi_pn2-antimov_conj_R42-tl}
\hat E(t)=E(t)\,; \qquad \hat K(t)=K(t)\,; \qquad \hat\varkappa(t) = \varkappa(t) \,; \qquad
\left.\hat\Phi^{\prime\bot}\right.^2\!\! (t) = - {\Phi^{\prime \bot}}^2\! (t)\,. 
\end{equation}

 If $\M$ is of the first or of the second type and $t$ determines canonical coordinates on $\M$, 
then $\hat\M$ is of the second or the first type, respectively and the coordinates determined by $t$ are also
canonical on $\hat\M$.

 If $\M$ is of the third type and $t$ determines canonical coordinates on $\M$, then $\hat\M$ is also of the third type
with canonical coordinates $s$ on $\hat\M$ given by the equality $t=\bar s$.
\end{theorem}

 The last theorem shows that the surfaces of the first and the second type can not be distinguished only by the invariants
$K$ and $\varkappa$, as we noted immediately after Theorem~\ref{DegP_kind123-K_kappa-tl}\,.


\section{Frenet type formulas and natural equations for a minimal time\,-like surface in $\RR^4_2$.}
\label{sect_nat_eq_R42-tl}

 Let $\M=(\D,\x)$ be a minimal time-like surface of general type in $\RR^4_2$ parametrized by canonical coordinates
$t=u+\jj v$. Consider the orthonormal pair of tangent vector fields $(\vX_1,\vX_2)$, oriented as the coordinate vector fields 
$(\x_u,\x_v)$. It is clear that the canonical coordinates determine a canonical tangent orthonormal frame field.
According to Theorem \ref{Can_Coord-uniq_R42-tl}\,, the canonical parametric lines are determined uniquely up to 
orientation, which implies that the tangent orthonormal frame field is determined uniquely up to orientation 
of the basic vector fields.

Our next goal is to obtain a canonical normal frame field, generated by the canonical coordinates.
In the cases of minimal space-like surfaces in $\RR^4$, $\RR^4_1$ and $\RR^4_2$, as well as in the case of minimal 
time-like surfaces in $\RR^4_1$, such a normal frame field is obtained taking a pair of normal vector fields
$(\n_1,\n_2)$, collinear with $\sigma (\x_u,\x_u)$ and $\sigma(\x_u,\x_v)$. This is possible because
$\sigma (\x_u,\x_u)\bot \: \sigma(\x_u,\x_v)$. The same property $\sigma (\x_u,\x_u)\bot \: \sigma(\x_u,\x_v)$ is 
present in the case of minimal time-like surfaces of the first or the second type in $\RR^4_2$, according to
Proposition \ref{Can-sigma_R42-tl}. Therefore one can apply the same method in the mentioned case. Such a method 
is applied in \cite{M-A-1}, where Frenet formulas for the basis $(\vX_1,\vX_2,\n_1,\n_2)$ are obtained.
This method is not applicable for the minimal time-like surfaces of the third type, since according to 
Proposition \ref{Can-sigma_R42-tl}\, $\sigma (\x_u,\x_u)\not\!\!\bot \: \sigma(\x_u,\x_v)$. In the present paper
we shall apply an alternative method to obtain a canonical isotropic basis of $N_p(\M)$ 
with the properties~\eqref{m12_psd_ort_R42-tl}.

 Let $(\m_1,\m_2)$ be a pair of isotropic normal vector fields, such that for every $p\in\M$,
$(\vX_1,\vX_2,\m_1,\m_2)$ is a right oriented basis of $\RR^4_2$\, and\, $\m_1\m_2 = \frac{1}{2}$.
Consider again the representation \eqref{Phi_pn_m12_R42-tl} of $\Phi'^\bot$. In the case of canonical coordinates,
the first formula in \eqref{Phi_pn2_m12_R42-tl} gives that: 
$c_1c_2={\Phi'^\bot}^2=\varepsilon$, where $\varepsilon=\pm1;\,\jj$ depending on the type of the surface.
It follows from here that $c_2=\varepsilon c_1^{-1}$. Using the exponential representation \eqref{t-exp} of the 
double numbers, we write $c_1$ in the form: $c_1 = \delta\rho\e^{\jj\frac{\varphi}{2}}$.
Then we have $c_2 = \varepsilon\delta\rho^{-1}\e^{-\jj\frac{\varphi}{2}}$.
Consequently formula \eqref{Phi_pn_m12_R42-tl} gets the form:
\begin{equation*}
\Phi^{\prime \bot} = 
\delta\rho\e^{\jj\frac{\varphi}{2}}\m_1 + \varepsilon\delta\rho^{-1}\e^{-\jj\frac{\varphi}{2}}\m_2\,; \qquad\quad 
\varepsilon=\pm 1;\jj\,, \quad \delta=\pm 1; \pm\jj\,, \quad  \rho>0\,, \quad  \varphi\in\RR\,.
\end{equation*}
Changing the basis $(\m_1,\m_2)$ with $(\rho\m_1,\rho^{-1}\m_2)$, then in the last formula we obtain $\rho=1$\,.
Furthermore we can choose the directions of the vectors $(\m_1,\m_2)$, in such a way that: $\delta= 1$ or $\delta= \jj$. 
So, with that choice of the pair $(\m_1,\m_2)$, the representation \eqref{Phi_pn_m12_R42-tl} of $\Phi'^\bot$ 
gets the form:
\begin{equation}\label{Phi_pn_m12_can_R42-tl}
\Phi'^\bot=
\delta\e^{\jj\frac{\varphi}{2}}\m_1 + \varepsilon\delta\e^{-\jj\frac{\varphi}{2}}\m_2\,; \qquad\quad 
\varepsilon=\pm 1;\jj\,, \quad \delta= 1; \jj\,, \quad  \varphi\in\RR\,.
\end{equation}

 Next we shall establish that $\varepsilon$, $\delta$ and $\varphi$, as well as the isotropic basis $(\m_1,\m_2)$,
are determined uniquely by \eqref{Phi_pn_m12_can_R42-tl}. For that purpose, suppose that $\hat\varepsilon$, 
$\hat\delta$, $\hat\varphi$ and $(\hat\m_1,\hat\m_2)$ is another set for which \eqref{Phi_pn_m12_can_R42-tl}
is valid. We know from \eqref{hat_m12_R42-tl} that $\hat\m_1 = k\m_1$ and $\hat\m_2 = k^{-1}\m_2$. Then it follows that: 
$\Phi'^\bot= \hat\delta k\e^{\jj\frac{\hat\varphi}{2}}\m_1 + \hat\varepsilon\hat\delta k^{-1}\e^{-\jj\frac{\hat\varphi}{2}}\m_2$.
Comparing the coefficients before $\m_1$ in the last formula and in \eqref{Phi_pn_m12_can_R42-tl}, we find
$\delta\e^{\jj\frac{\varphi}{2}}=\hat\delta k\e^{\jj\frac{\hat\varphi}{2}}$. As a corollary of the uniqueness 
of the exponential representation \eqref{t-exp}, we get $|k|=1$ and $\varphi=\hat\varphi$. 
Thus we have $\delta=k\hat\delta=\pm\hat\delta$. 
Taking into account the form of $\delta$ and $\hat\delta$ we conclude that the last equality is possible
only if $k=1$\,. Consequently $\delta=\hat\delta$, $\m_1=\hat\m_1$\, and\, $\m_2=\hat\m_2$. Comparing the 
coefficients before $\m_2$, we get $\varepsilon=\hat\varepsilon$. We combine the properties of the representation
\eqref{Phi_pn_m12_can_R42-tl} in the following statement:
\begin{prop}\label{m12_e_d_phi-unic_R42-tl}
 Let $\M=(\D,\x)$ be a minimal time-like surface of general type in $\RR^4_2$, parametrized by canonical coordinates
and let $(\vX_1,\vX_2)$ be the orthonormal vector fields, oriented as the coordinate vectors. Then there exists a
unique pair of isotropic normal vector fields $(\m_1,\m_2)$ and $\varepsilon$, $\delta$, $\varphi$,
such that $(\vX_1,\vX_2,\m_1,\m_2)$ is a right oriented quadruple in $\RR^4_2$, $\m_1\m_2 = \frac{1}{2}$ and 
$\Phi'^\bot$ satisfies \eqref{Phi_pn_m12_can_R42-tl}. The number $\varepsilon$ in \eqref{Phi_pn_m12_can_R42-tl} 
coincides with $\varepsilon$ from Definition \ref{Can-def_R42-tl} for canonical coordinates.\end{prop}

 The above proposition shows that the canonical coordinates of any minimal time-like surface of general type in $\RR^4_2$
generate also a canonical normal isotropic basis of the corresponding normal space. On the other hand, any isotropic 
basis generates an orthonormal basis according to formulas \eqref{n12_m12_R42-tl}. Therefore the canonical 
coordinates generate a canonical orthonormal frame field $(\vX_1,\vX_2,\n_1,\n_2)$ in $\RR^4_2$. It is easy to see 
that in the case of surfaces of the first or the second type,
this frame field coincides with the frame field, obtained in the way described at the beginning of this section. 
More precisely, in the case of surfaces of the first or the second type, thus obtained vectors $\n_1$ and $\n_2$, are
collinear with $\sigma(\x_u,\x_u)$ and $\sigma(\x_u,\x_v)$. For example, let $\varepsilon=1$ and $\delta=1$\,.
Then we have from \eqref{Phi_pn_m12_can_R42-tl} and \eqref{m12_n12_R42-tl}:
\begin{equation*}
\Phi'^\bot=
\e^{\jj\frac{\varphi}{2}}\m_1 + \e^{-\jj\frac{\varphi}{2}}\m_2=
\e^{\jj\frac{\varphi}{2}}\frac{\n_1 + \n_2}{2} + \e^{-\jj\frac{\varphi}{2}}\frac{-\n_1 + \n_2}{2}
=\jj\sinh\!\frac{\varphi}{2}\,\n_1 + \cosh\!\frac{\varphi}{2}\,\n_2 \,.
\end{equation*}
From here and \eqref{PhiPr-tl} it follows that:
\begin{equation*}
\sigma(\x_u,\x_v) = \sinh\!\frac{\varphi}{2}\,\n_1\,; \qquad \sigma(\x_u,\x_u) = \cosh\!\frac{\varphi}{2}\,\n_2 \,.
\end{equation*}
In analogous way, it can be seen that for the other values of $\varepsilon$ and $\delta$ similar equalities are valid.
As we noted at the beginning of the section, in the case of surfaces of the third type, the formulas are different.
For example, if $\varepsilon=\jj$ and $\delta=1$, after some calculations it follows that:
\begin{equation*}
\sigma(\x_u,\x_u) =  \frac{1}{2}\e^{ \frac{\varphi}{2}}\n_1 + \frac{1}{2}\e^{-\frac{\varphi}{2}}\n_2 \,; \qquad 
\sigma(\x_u,\x_v) = -\frac{1}{2}\e^{-\frac{\varphi}{2}}\n_1 + \frac{1}{2}\e^{ \frac{\varphi}{2}}\n_2 \,.
\end{equation*}

 Further in this work, we shall use the isotropic basis $(\m_1,\m_2)$ and formula \eqref{Phi_pn_m12_can_R42-tl}, 
in order to uniform the formulas in all three types of surfaces. Furthermore, as in the case of space-like surfaces, 
we shall use in the tangent space the complex $\DD^4_2$-valued vectors $\Phi$ and $\bar\Phi$ instead of the vectors
$\vX_1$ and $\vX_2$. The equality $\Phi^2=0$ means that $\Phi$ and $\bar\Phi$ are isotropic with respect to
the bilinear product in $\DD^4_2$. Consequently, the quadruple $(\Phi,\bar\Phi,\m_1,\m_2)$ is an isotropic frame field
in $\DD^4_2$. This frame field will be used as the Frenet frame field and the next our goal is to 
obtain Frenet type formulas for this frame field.

\smallskip

 Since $\Phi$ is a holomorphic (over $\DD$) function, the first formula is simply $\frac{\partial\Phi}{\partial \bar t}=0$\,. 
The tangent projection of $\frac{\partial\Phi}{\partial t}$ is given by \eqref{Phi_p-Tang-tl}, while \eqref{Phi_pn_m12_can_R42-tl}
gives its normal projection. Thus we have:
\begin{equation*}
\frac{\partial\Phi}{\partial t} =
\frac{\partial\ln |E|}{\partial t}\Phi + \delta\e^{\jj\frac{\varphi}{2}}\m_1 + \varepsilon\delta\e^{-\jj\frac{\varphi}{2}}\m_2\,.
\end{equation*}
The derivative of the vector field $\m_1$ is represented as follows:
\begin{equation*}
\frac{\partial\m_1}{\partial t} =  
\frac{\frac{\partial\m_1}{\partial t}\bar\Phi}{\|\Phi\|^2}\Phi + 
\frac{\frac{\partial\m_1}{\partial t}\Phi}{\|\Phi\|^2}\bar\Phi + 
2\big( \textstyle\frac{\partial\m_1}{\partial t}\m_2 \big)\m_1 + 
2\big( \textstyle\frac{\partial\m_1}{\partial t}\m_1 \big)\m_2 \,. 
\end{equation*}
We find the coefficients before $\Phi$ and $\bar\Phi$: 
\[
\frac{\partial\m_1}{\partial t}\bar\Phi=-\m_1\frac{\partial\bar\Phi}{\partial t}=0\,, \quad
\frac{\partial\m_1}{\partial t}\Phi=-\m_1\frac{\partial\Phi}{\partial t}=-\frac{1}{2}\varepsilon\delta\e^{-\jj\frac{\varphi}{2}}
\text{\ \ \ и\ \ \ } \|\Phi\|^2=2E\,. 
\]

It follows from the equality $\m_1^2=0$ that $\frac{\partial\m_1}{\partial t}\m_1=0$, and denoting by $\beta$ the coefficient
before $\m_1$, we have:
\begin{equation*}
\frac{\partial\m_1}{\partial t} = -\frac{\varepsilon\delta\e^{-\jj\frac{\varphi}{2}}}{4E}\bar\Phi + \beta\m_1 \,. 
\end{equation*}

In a similar way we obtain the corresponding formula for $\frac{\partial\m_2}{\partial t}$.
In view of the equality $\m_1\m_2=\frac{1}{2}$ it follows that 
$\frac{\partial\m_1}{\partial t}\m_2+\m_1\frac{\partial\m_2}{\partial t}=0$\,,
which gives that the coefficient before $\m_2$ is equal to $-\beta$. Thus we have:
\begin{equation*}
\frac{\partial\m_2}{\partial t} = -\frac{\delta\e^{\jj\frac{\varphi}{2}}}{4E}\bar\Phi - \beta\m_2 \,. 
\end{equation*}

 So far we have found the formulas for four of  the derivatives: 
$\frac{\partial\Phi}{\partial \bar t}$, $\frac{\partial\Phi}{\partial t}$, $\frac{\partial\m_1}{\partial t}$ and 
$\frac{\partial\m_2}{\partial t}$. The formulas for
$\frac{\partial\bar\Phi}{\partial t}$, $\frac{\partial\bar\Phi}{\partial\bar t}$, $\frac{\partial\m_1}{\partial\bar t}$ 
and 
$\frac{\partial\m_2}{\partial\bar t}$
are obtained from the first ones applying complex conjugation. Therefore they do not give new information.
Summarizing, we have the following Frenet type formulas for the frame field $(\Phi,\bar\Phi,\m_1,\m_2)$\,:

\begin{equation}\label{Frene_Phi_bar_Phi_m1_m2_R42-tl}
\begin{array}{llcrrr}
 \ds\frac{\partial\Phi}{\partial \bar t} &=&  0\,;                                     & &            & \\[2ex]
 \ds\frac{\partial\Phi}{\partial t}      &=&  \ds\frac{\partial\ln |E|}{\partial t}\Phi  & &
+\quad \delta\e^{\jj\frac{\varphi}{2}}\m_1\phantom{\,;} &\ +\ \ \ \ \,\varepsilon\delta\e^{-\jj\frac{\varphi}{2}}\m_2\,;\\[2ex]
 \ds\frac{\partial\m_1}{\partial t}      &=&            &\ -\ \ \ \ds\frac{\varepsilon\delta\e^{-\jj\frac{\varphi}{2}}}{4E}\bar\Phi&
 \quad + \qquad\, \beta\m_1\,; &\\[2ex]
 \ds\frac{\partial\m_2}{\partial t}      &=&            & -\quad\ \ \ \ds\frac{\delta\e^{\jj\frac{\varphi}{2}}}{4E}\bar\Phi  &
 & -\qquad \quad \ \ \beta\m_2\,.
\end{array}
\end{equation}

\bigskip
 Considering the last equations as a system of PDE's, solved with respect to the derivatives, we shall find 
the integrability conditions for this system:
\[
\frac{\partial^2\Phi}{\partial \bar t \partial t}=\frac{\partial^2\Phi}{\partial t \partial\bar t}\:; \qquad
\frac{\partial^2\m_1}{\partial \bar t \partial t}=\frac{\partial^2\m_1}{\partial t \partial\bar t}\:; \qquad
\frac{\partial^2\m_2}{\partial \bar t \partial t}=\frac{\partial^2\m_2}{\partial t \partial\bar t}\:.
\]

 The first equality gives 
$\frac{\partial^2\Phi}{\partial t \partial\bar t}=0$.
Then the equality of the mixed partial derivatives of $\Phi$ implies:

\[
\begin{array}{lll}
0=\ds\frac{\partial^2\Phi}{\partial \bar t \partial t} &=&
      \ds\frac{\partial^2\ln |E|}{\partial\bar t \partial t}\Phi + 
      \ds\frac{\partial\ln |E|}{\partial t}\ds\frac{\partial\Phi}{\partial\bar t} + 
      \ds\frac{\partial(\delta\e^{\jj\frac{\varphi}{2}}\m_1)}{\partial\bar t} + 
			\ds\frac{\partial(\varepsilon\delta\e^{-\jj\frac{\varphi}{2}}\m_2)}{\partial\bar t}\\[2.0ex]
 				 &=&\ds\frac{\Delta^h\ln |E|}{4}\Phi+
				\ds\frac{\partial(\delta\e^{\jj\frac{\varphi}{2}})}{\partial\bar t}\m_1+
				 \delta\e^{\jj\frac{\varphi}{2}}\ds\frac{\partial\m_1}{\partial\bar t} + 
				 \ds\frac{\partial(\varepsilon\delta\e^{-\jj\frac{\varphi}{2}})}{\partial\bar t}\m_2+
				  \varepsilon\delta\e^{-\jj\frac{\varphi}{2}}\ds\frac{\partial\m_2}{\partial\bar t}\\[2.0ex]
				 &=&\left(\ds\frac{\Delta^h\ln |E|}{4}-\ds\frac{\bar\varepsilon|\delta|^2\e^{\jj\varphi}}{4E}-
				\ds\frac{\varepsilon|\delta|^2\e^{-\jj\varphi}}{4E}\right)\Phi\\[2.0ex]
&+&\left(\ds\frac{\partial(\delta\e^{\jj\frac{\varphi}{2}})}{\partial\bar t}+\delta\e^{\jj\frac{\varphi}{2}}\bar\beta\right)\m_1+
	 \left(\ds\frac{\partial(\varepsilon\delta\e^{-\jj\frac{\varphi}{2}})}{\partial\bar t} -
	       \varepsilon\delta\e^{-\jj\frac{\varphi}{2}}\bar\beta\right)\m_2\,.
\end{array}
\]

 Setting equal to zero the coefficient before $\Phi$, and applying the equality: 
\[
\bar\varepsilon|\delta|^2\e^{\jj\varphi} + \varepsilon|\delta|^2\e^{-\jj\varphi}
= 2\Re\big(\bar\varepsilon|\delta|^2\e^{\jj\varphi}\big)\,,
\]
we get:
\begin{equation}\label{Nat_Eq_Gauss_E_phi_R42-tl}
\Delta^h\ln |E| - \frac{2\Re\big(\bar\varepsilon|\delta|^2\e^{\jj\varphi}\big)}{E}=0\,.
\end{equation} 
Taking into account 
\eqref{K_kappa_c12_R42-tl}, we see that the equation \eqref{Nat_Eq_Gauss_E_phi_R42-tl} is in essence
 the classical Gauss equation \eqref{Delta_lnE_2K-tl} for a minimal time-like surface in $\RR^4_2$, 
parametrized by canonical coordinates.

 Setting equal to zero the coefficient before $\m_1$ in the equality $\frac{\partial^2\Phi}{\partial\bar t \partial t}=0$\,, 
we find:
\[
0 = \frac{\partial(\delta\e^{\jj\frac{\varphi}{2}})}{\partial\bar t}+\delta\e^{\jj\frac{\varphi}{2}}\bar\beta = 
\delta\e^{\jj\frac{\varphi}{2}} \frac{\jj}{2} \frac{\partial\varphi}{\partial\bar t}+\delta\e^{\jj\frac{\varphi}{2}}\bar\beta\,.
\]
The last equality is equivalent to:
\begin{equation}\label{Nat_Eq_Codazzi_beta_E_phi_R42-tl}
\beta = \frac{\jj}{2} \frac{\partial\varphi}{\partial t}\:.
\end{equation}
Setting equal to zero the coefficient before $\m_2$ in the equality $\frac{\partial^2\Phi}{\partial\bar t \partial t}=0$\,, 
we again obtain the equality \eqref{Nat_Eq_Codazzi_beta_E_phi_R42-tl}. Therefore, the classical Codazzi equations 
for a minimal time-like surface in $\RR^4_2$, parametrized by canonical coordinates, are reduced to 
\eqref{Nat_Eq_Codazzi_beta_E_phi_R42-tl}.

\smallskip

 Further we turn to the equality of the mixed partial derivatives of the vector fields $\m_1$ and $\m_2$.
First we consider their projections into the tangent space  of the given surface:
\[
\begin{array}{lll}
 \left(\ds\frac{\partial^2\m_1}{\partial \bar t \partial t}\right)^\top &=&
    -\ds\frac{\partial}{\partial \bar t}\left(\frac{\varepsilon\delta\e^{-\jj\frac{\varphi}{2}}}{4E}\right) \bar\Phi -
     \ds\frac{\varepsilon\delta\e^{-\jj\frac{\varphi}{2}}}{4E} \left(\frac{\partial \bar\Phi}{\partial \bar t}\right)^\top+
		 \beta \left(\frac{\partial \m_1}{\partial \bar t}\right)^\top\\[2.0ex]
&=& -\ds\frac{\partial(\varepsilon\delta\e^{-\jj\frac{\varphi}{2}})}{\partial\bar t} \frac{1}{4E} \bar\Phi +
     \varepsilon\delta\e^{-\jj\frac{\varphi}{2}} \frac{1}{4E^2} \frac{\partial E}{\partial\bar t} \bar\Phi -
		 \ds\frac{\varepsilon\delta\e^{-\jj\frac{\varphi}{2}}}{4E}  \ds\frac{\partial\ln |E|}{\partial\bar t}\bar\Phi -
		 \beta  \ds\frac{\bar\varepsilon\bar\delta\e^{\jj\frac{\varphi}{2}}}{4E} \Phi \\[2.0ex]
&=&  -\ds\frac{\partial(\varepsilon\delta\e^{-\jj\frac{\varphi}{2}})}{\partial\bar t} \frac{1}{4E} \bar\Phi -
   	 \beta  \ds\frac{\bar\varepsilon\bar\delta\e^{\jj\frac{\varphi}{2}}}{4E} \Phi \,.
\end{array}
\]

\[
\begin{array}{lll}
 \left(\ds\frac{\partial^2\m_1}{\partial t \partial\bar t}\right)^\top &=&
    -\ds\frac{\partial}{\partial t}\left(\frac{\bar\varepsilon\bar\delta\e^{\jj\frac{\varphi}{2}}}{4E}\right) \Phi -
     \ds\frac{\bar\varepsilon\bar\delta\e^{\jj\frac{\varphi}{2}}}{4E} \left(\frac{\partial \Phi}{\partial t}\right)^\top+
		 \bar\beta \left(\frac{\partial \m_1}{\partial t}\right)^\top\\[2.0ex]
&=& -\ds\frac{\partial(\bar\varepsilon\bar\delta\e^{\jj\frac{\varphi}{2}})}{\partial t} \frac{1}{4E} \Phi +
     \bar\varepsilon\bar\delta\e^{\jj\frac{\varphi}{2}} \frac{1}{4E^2} \frac{\partial E}{\partial t} \Phi -
		 \ds\frac{\bar\varepsilon\bar\delta\e^{\jj\frac{\varphi}{2}}}{4E}  \ds\frac{\partial\ln |E|}{\partial t} \Phi -
		 \bar\beta  \ds\frac{\varepsilon\delta\e^{-\jj\frac{\varphi}{2}}}{4E} \bar\Phi \\[2.0ex]
&=&  -\ds\frac{\partial(\bar\varepsilon\bar\delta\e^{\jj\frac{\varphi}{2}})}{\partial t} \frac{1}{4E} \Phi -
   	 \bar\beta  \ds\frac{\varepsilon\delta\e^{-\jj\frac{\varphi}{2}}}{4E} \bar\Phi \,.
\end{array}
\]

 Equating the coefficients before $\Phi$ and $\bar\Phi$ in the last equalities, we obtain again 
\eqref{Nat_Eq_Codazzi_beta_E_phi_R42-tl}. In a similar way we find the tangential projections of the mixed 
derivatives of the vector field $\m_2$. Equating these projections, we obtain again equality
\eqref{Nat_Eq_Codazzi_beta_E_phi_R42-tl}.

\smallskip

 Finally we consider the normal projections of the mixed partial derivatives of $\m_1$ and $\m_2$. 
For the vector field $\m_1$ we find:
\[
\begin{array}{lll}
 \left(\ds\frac{\partial^2\m_1}{\partial \bar t \partial t}\right)^\bot &=&
    -\ds\frac{\varepsilon\delta\e^{-\jj\frac{\varphi}{2}}}{4E} \left(\ds\frac{\partial \bar\Phi}{\partial \bar t}\right)^\bot +
     \ds\frac{\partial \beta}{\partial \bar t}\m_1 + \beta \left(\ds\frac{\partial \m_1}{\partial \bar t}\right)^\bot \\[2.0ex]
&=&  -\ds\frac{\varepsilon\delta\e^{-\jj\frac{\varphi}{2}}}{4E}
      \left(\bar\delta\e^{-\jj\frac{\varphi}{2}}\m_1 + \bar\varepsilon\bar\delta\e^{\jj\frac{\varphi}{2}}\m_2 \right) + 
      \ds\frac{\partial \beta}{\partial \bar t}\m_1 + \beta \bar\beta \m_1 \\[2.0ex]
&=&  \phantom{-}\left(-\ds\frac{\varepsilon|\delta|^2\e^{-\jj\varphi}}{4E} + \ds\frac{\partial \beta}{\partial \bar t} + 
     |\beta|^2\right)\m_1 -\ds\frac{|\varepsilon|^2|\delta|^2}{4E} \m_2 \,.
\end{array}
\]
In a similar way we have:
\[
\begin{array}{lll}
 \left(\ds\frac{\partial^2\m_1}{\partial t \partial \bar t}\right)^\bot &=&
    -\ds\frac{\bar\varepsilon\bar\delta\e^{\jj\frac{\varphi}{2}}}{4E} \left(\ds\frac{\partial \Phi}{\partial t}\right)^\bot +
     \ds\frac{\partial \bar\beta}{\partial t}\m_1 + \bar\beta \left(\ds\frac{\partial \m_1}{\partial t}\right)^\bot \\[2.0ex]
&=&  -\ds\frac{\bar\varepsilon\bar\delta\e^{\jj\frac{\varphi}{2}}}{4E}
      \left(\delta\e^{\jj\frac{\varphi}{2}}\m_1 + \varepsilon\delta\e^{-\jj\frac{\varphi}{2}}\m_2 \right) + 
      \ds\frac{\partial \bar\beta}{\partial t}\m_1 + \bar\beta \beta \m_1 \\[2.0ex]
&=&  \phantom{-}\left(-\ds\frac{\bar\varepsilon|\delta|^2\e^{\jj\varphi}}{4E} + \ds\frac{\partial \bar\beta}{\partial t} + 
     |\beta|^2\right)\m_1 -\ds\frac{|\varepsilon|^2|\delta|^2}{4E} \m_2 \,.
\end{array}
\]
The last formulas show that the coefficients before $\m_2$ coincide. Equating the coefficients before $\m_1$, 
we obtain an equation, which is the Ricci equation for a minimal time-like surface in $\RR^4_2$, parametrized 
by canonical coordinates:
\[
 -\ds\frac{\varepsilon|\delta|^2\e^{-\jj\varphi}}{4E}    + \ds\frac{\partial \beta}{\partial \bar t} + |\beta|^2 = 
 -\ds\frac{\bar\varepsilon|\delta|^2\e^{\jj\varphi}}{4E} + \ds\frac{\partial \bar\beta}{\partial t}  + |\beta|^2 .
\]
The last equality is equivalent to:
\[
\ds\frac{\partial \beta}{\partial \bar t} - \ds\frac{\partial \bar \beta}{\partial t} = 
-\ds\frac{\bar\varepsilon|\delta|^2\e^{\jj\varphi}}{4E} + \ds\frac{\varepsilon|\delta|^2\e^{-\jj\varphi}}{4E}
\quad \Leftrightarrow \quad
2 \jj \Im \ds\frac{\partial \beta}{\partial \bar t} = -\frac{2\jj\Im\big(\bar\varepsilon|\delta|^2\e^{\jj\varphi}\big)}{4E} \,.
\]
Finally we have:
\begin{equation}\label{Nat_Eq_Ricci_beta_E_phi_R42-tl}
\Im \ds\frac{\partial \beta}{\partial \bar t} = -\frac{\Im\big(\bar\varepsilon|\delta|^2\e^{\jj\varphi}\big)}{4E} \:.
\end{equation}
Equating the normal projections of the mixed partial derivatives of the second vector field $\m_2$, we obtain 
no new equations. 

Summarizing we have: 
\begin{prop}\label{Frene_integr_cond_R42-tl}
The integrability conditions of the system \eqref{Frene_Phi_bar_Phi_m1_m2_R42-tl} are given by the equalities 
\eqref{Nat_Eq_Gauss_E_phi_R42-tl}, \eqref{Nat_Eq_Codazzi_beta_E_phi_R42-tl} and \eqref{Nat_Eq_Ricci_beta_E_phi_R42-tl}. 
\end{prop}

 We can easily eliminate the function~$\beta$ in the integrability conditions obtained above. Thus we will 
find two PDE's for the functions $E$ and $\varphi$. Replacing $\beta$ from equality 
\eqref{Nat_Eq_Codazzi_beta_E_phi_R42-tl} in formula \eqref{Nat_Eq_Ricci_beta_E_phi_R42-tl} we get:
\[
\Im \left( \frac{\jj}{2} \ds\frac{\partial^{\,2}\varphi}{\partial \bar t \partial t} \right) = 
-\frac{\Im\big(\bar\varepsilon|\delta|^2\e^{\jj\varphi}\big)}{4E} \:.
\]
Using that
$\frac{\partial^2}{\partial \bar t \partial t} = \frac{1}{4}\Delta^h$, we obtain
\[
\Im \left( \frac{\jj}{8} \Delta^h\varphi \right) =
\frac{1}{8} \Delta^h\varphi = -\frac{\Im\big(\bar\varepsilon|\delta|^2\e^{\jj\varphi}\big)}{4E} \:,
\]
which is equivalent to: 
\[
\Delta^h\varphi = -\frac{2\Im\big(\bar\varepsilon|\delta|^2\e^{\jj\varphi}\big)}{E} \:.
\]

Adding the equation \eqref{Nat_Eq_Gauss_E_phi_R42-tl} to the last equation, we find that $E$ and $\varphi$ 
satisfy the following system of PDE's:
\begin{equation}\label{Nat_Eq_E_phi_R42-tl}
\begin{array}{l}
\Delta^h\ln |E| =  \ds\frac{2\Re\big(\bar\varepsilon|\delta|^2\e^{\jj\varphi}\big)}{E} \:; \\[2.0ex]
\Delta^h\varphi = -\ds\frac{2\Im\big(\bar\varepsilon|\delta|^2\e^{\jj\varphi}\big)}{E} \:;
\end{array} \qquad\quad E<0\;, \quad \varepsilon=\pm 1;\jj\;, \quad \delta= 1;\jj \;.
\end{equation} 
We call this system of PDE's the \emph{system of natural equations} of the minimal time-like surfaces
in $\RR^4_2$.

\smallskip

 Further, consider two minimal time-like surfaces of general type $\M=(\D,\x)$ and $\hat\M=(\D,\hat\x)$ in $\RR^4_2$,
which are obtained one from the other through a proper motion as follows:
\begin{equation}\label{hat_M-M-prop_mov_R42-tl}
\hat\x = A\x + \vb\,; \qquad A \in \mathbf{SO}(2,2,\RR), \ \ \vb \in \RR^4_2 \,.
\end{equation}
If the variable $t=u+\jj v$ determines canonical coordinates on $\M$, then according to Theorem \ref{Can_Move-tl}\,
these coordinates are also canonical on $\hat\M$. Then we have  formulas \eqref{hat_Phi_pn-Phi_pn-mov_R42-tl}.
From the first of these formulas and \eqref{EG-tl} it follows that $\hat E = E$.
Next we prove that the whole representation \eqref{Phi_pn_m12_can_R42-tl} of $\Phi'^\bot$ is invariant 
under a proper motion. From the third formula in \eqref{hat_Phi_pn-Phi_pn-mov_R42-tl} and \eqref{Phi_pn_m12_can_R42-tl} 
we have:
\begin{equation}\label{hat_Phi_pn_m12_can-mov_R42-tl}
\hat\Phi'^\bot=
\delta\e^{\jj\frac{\varphi}{2}}A\m_1 + \varepsilon\delta\e^{-\jj\frac{\varphi}{2}}A\m_2\,; \qquad\quad 
\varepsilon=\pm 1;\jj\,, \quad \delta= 1; \jj\,, \quad  \varphi\in\RR\,.
\end{equation}
We shall apply Proposition \ref{m12_e_d_phi-unic_R42-tl}\, to the above formula.
The quadruple $(A\vX_1,A\vX_2,A\m_1,A\m_2)$ is right oriented under a proper motion. Furthermore, the vectors
$A\m_1$ and $A\m_2$ are isotropic normal vectors, such that $A\m_1\,A\m_2 = \m_1\m_2 = \frac{1}{2}$\,.
This means that the pair $(A\m_1,A\m_2)$ is the canonical basis in the normal space of $\hat\M$, described in
Proposition \ref{m12_e_d_phi-unic_R42-tl}\,. According to the same proposition the representation
\eqref{Phi_pn_m12_can_R42-tl} is unique. Then \eqref{hat_Phi_pn_m12_can-mov_R42-tl} implies:
$\hat\m_1=A\m_1$\,,\  $\hat\m_2=A\m_2$\,,\  $\hat\varepsilon=\varepsilon$\,,\  $\hat\delta=\delta$ and $\hat\varphi=\varphi$\,.
Now we summarize the above results in the following:
\begin{theorem}\label{Thm-Nat_Eq_E_phi_R42-tl}
Let $\M$ be a minimal time-like surface of general type in $\RR^4_2$, parametrized by canonical coordinates.
Then the coefficient $E$ of the first fundamental form, the function $\varphi$ and the constants $\varepsilon$, $\delta$,
determined by the representation \eqref{Phi_pn_m12_can_R42-tl} of $\Phi'^\bot$, satisfy the system of natural equations
\eqref{Nat_Eq_E_phi_R42-tl} of the minimal time-like surfaces in $\RR^4_2$. 

If $\hat\M$ is obtained from $\M$ through a proper motion of the type \eqref{hat_M-M-prop_mov_R42-tl}, 
then it generates the same solution to \eqref{Nat_Eq_E_phi_R42-tl}. 
\end{theorem}

 Next we obtain different equivalent forms of system \eqref{Nat_Eq_E_phi_R42-tl}.
In order to obtain more symmetric form of the system, we introduce the real function $\eta$ by the equality:
\begin{equation}\label{E_eta_R42-tl}
E = -\,\e^{-\eta} \,.
\end{equation}
Replacing the last formula in \eqref{Nat_Eq_E_phi_R42-tl}, we get:
\begin{equation}\label{Nat_Eq_eta_phi_R42-tl}
\begin{array}{lll}
\Delta^h\eta    \!\! &=& \!\!  2\Re\big(\bar\varepsilon|\delta|^2\e^{\eta+\jj\varphi}\big) \,; \\[1.0ex]
\Delta^h\varphi \!\! &=& \!\! 2\Im\big(\bar\varepsilon|\delta|^2\e^{\eta+\jj\varphi}\big) \,;
\end{array} 
\qquad\quad \eta\in\RR\;, \quad \varphi\in\RR\;, \quad \varepsilon=\pm 1;\jj\;, \quad \delta= 1;\jj \;.
\end{equation} 
The above system is similar to the natural systems, obtained in \cite{G-M-1}, for the surfaces with zero mean curvature
in $\RR^4$ and $\RR^4_1$.

 Next we introduce in the last system \eqref{Nat_Eq_eta_phi_R42-tl} the $\DD$-valued function $\lambda=\eta+\jj\varphi$.
Multiplying by $\jj$ the second equation and adding it to the first one, we get:
\begin{equation}\label{Nat_Eq_lambda_R42-tl}
\Delta^h\lambda    =  2\,\bar\varepsilon|\delta|^2\e^\lambda \,;  
\qquad\quad \lambda\in\DD\;, \quad \varepsilon=\pm 1;\jj\;, \quad \delta= 1;\jj \;.
\end{equation} 
This is a hyperbolic equation of Liouville type for the $\DD$-valued function $\lambda$. 

\medskip

 Further we obtain a system of natural equations for the Gauss curvature $K$ and the curvature of the normal 
connection $\varkappa$, as obtained in \cite{M-A-1} and \cite{S-2}. To this end, we find some relations between
$K$ and $\varkappa$ from one hand, and $E$\,, $\varphi$\,, $\varepsilon$ and $\delta$ from the other.
Let us turn to formulas \eqref{K_kappa_c12_R42-tl} for $K$ and $\varkappa$, related to the representation
\eqref{Phi_pn_m12_R42-tl} of $\Phi'^\bot$. Formula \eqref{Phi_pn_m12_can_R42-tl} is a special case of 
\eqref{Phi_pn_m12_R42-tl} in canonical coordinates.
In this case we have $c_1 = \delta\e^{\jj\frac{\varphi}{2}}$ and $c_2 = \varepsilon\delta\e^{-\jj\frac{\varphi}{2}}$.
Hence $c_1\bar c_2 = \bar\varepsilon|\delta|^2\e^{\jj\varphi}$. Therefore, formulas \eqref{K_kappa_c12_R42-tl} get 
the form:
\begin{equation}\label{K_kappa_c12-can_R42-tl}
K= -\frac{\Re(\bar\varepsilon|\delta|^2\e^{\jj\varphi})}{E^2} \:; \qquad\quad
\varkappa = -\frac{\Im(\bar\varepsilon|\delta|^2\e^{\jj\varphi})}{E^2}\:.
\end{equation}
The last formulas can be combined in the following way:
\begin{equation}\label{K+j_kappa_E_phi_R42-tl}
K+\jj\varkappa = -\frac{\bar\varepsilon|\delta|^2\e^{\jj\varphi}}{E^2} \:.
\end{equation}

 Comparing formulas \eqref{K_kappa_c12-can_R42-tl} with the right parts of the equations \eqref{Nat_Eq_E_phi_R42-tl}, 
we see that the system of natural equations can be written in the form:
\begin{equation}\label{Nat_Eq_E_phi_K_kappa_R42-tl}
\begin{array}{l}
\Delta^h\ln |E| =  -2EK \,; \\
\Delta^h\varphi = 2E\varkappa \,.
\end{array} 
\end{equation} 

 In order to express the left sides of the equations of the system by the functions $K$ and $\varkappa$, we have to consider
two cases in accordance with the type of the surface.

First, let the given surface be of the first or the second type according to Definition~\ref{Min_Surf_kind123-def-tl}\,.
Then we have $\bar\varepsilon=\varepsilon\in\RR$ and formulas \eqref{K_kappa_c12-can_R42-tl} can be written in the form:
\begin{equation}\label{K_kappa_E_phi-kind12_R42-tl}
K= -\frac{\varepsilon|\delta|^2\cosh\varphi}{E^2} \:; \qquad\quad
\varkappa = -\frac{\varepsilon|\delta|^2\sinh\varphi}{E^2}\:.
\end{equation}
From here it follows that:
\begin{equation*}
K^2-\varkappa^2 = \frac{1}{E^{\,4}} \:; \qquad\quad
\tanh\varphi = \frac{\varkappa}{K} \:.
\end{equation*}
\smallskip
Taking into account the inequality $E<0$ and the identity\ $\arctanh x = \ds\frac{1}{2}\ln\ds\frac{1+x}{1-x}$\:, 
these equalities can be written as follows:
\begin{equation}\label{E_phi_K2_kappa2-kind12_R42-tl}
E = -\ds\frac{1}{\sqrt[4]{ \vphantom{\mu^2} K^2-\varkappa^2}} \:; \qquad\quad
\varphi = \arctanh \frac{\varkappa}{K} =  \ds\frac{1}{2}\ln\ds\frac{K+\varkappa}{K-\varkappa}\:.
\end{equation}

 Now, let the given surface be of the third type, according to Definition \ref{Min_Surf_kind123-def-tl}\,.
Then $\bar\varepsilon=-\jj$ and formulas \eqref{K_kappa_c12-can_R42-tl} can be written in the form:
\begin{equation}\label{K_kappa_E_phi-kind3_R42-tl}
K= \frac{|\delta|^2\sinh\varphi}{E^2} \:; \qquad\quad
\varkappa = \frac{|\delta|^2\cosh\varphi}{E^2}\:.
\end{equation}
From here it follows that:
\begin{equation*}
\varkappa^2-K^2 = \frac{1}{E^{\,4}} \:; \qquad\quad
\tanh\varphi = \frac{K}{\varkappa} \:.
\end{equation*}
Further, analogously to \eqref{E_phi_K2_kappa2-kind12_R42-tl}, we get:
\begin{equation}\label{E_phi_K2_kappa2-kind3_R42-tl}
E = -\ds\frac{1}{\sqrt[4]{ \vphantom{\mu^2} \varkappa^2 - K^2}} \:; \qquad\quad
\varphi = \arctanh \frac{K}{\varkappa} =  \ds\frac{1}{2}\ln\ds\frac{\varkappa+K}{\varkappa-K}\:.
\end{equation}

 Finally we note that formulas \eqref{E_phi_K2_kappa2-kind12_R42-tl} and \eqref{E_phi_K2_kappa2-kind3_R42-tl} 
can be written in such a form which is valid for all three types of surfaces:
\begin{equation}\label{E_phi_K2_kappa2_R42-tl}
E = -\ds\frac{1}{\sqrt[4]{ \vphantom{\mu^2} |K^2-\varkappa^2|}} \:; \qquad\quad
\varphi = \ds\frac{1}{2}\ln \left| \ds\frac{K+\varkappa}{K-\varkappa} \right| .
\end{equation}

\smallskip

 Replacing the last equalities in \eqref{Nat_Eq_E_phi_K_kappa_R42-tl}, we obtain that the curvatures $K$ and 
$\varkappa$ satisfy the following system, which we also call a \emph{system of natural equations} of the minimal 
time-like surfaces in $\RR^4_2$: 
\begin{equation}\label{Nat_Eq_K_kappa_R42-tl}
\begin{array}{lll}
\sqrt[4]{\big|K^2-\varkappa^2\big|}\; \Delta^h\ln \big|K^2-\varkappa^2\big| &=& -8K\,; \\[1.5ex]
\sqrt[4]{\big|K^2-\varkappa^2\big|}\; \Delta^h\ln \left|\ds\frac{\vphantom{\mu^2}K+\varkappa}{K-\varkappa}\right|&=&-4\varkappa\;;
\end{array}  \qquad\quad K^2-\varkappa^2\neq 0 \,.
\end{equation}

 The natural equations of this type are also obtained in \cite{M-A-1}, under the condition \linebreak $K^2-\varkappa^2 > 0$\,,
which corresponds to surfaces of the first or the second type, according to Definition \ref{Min_Surf_kind123-def-tl}\,.
We also note that the signs of the right sides of the equations \eqref{Nat_Eq_K_kappa_R42-tl} in \cite{M-A-1}
are the opposite to ours. This difference comes from the fact that in this paper the authors use isothermal coordinates
with the condition $E>0$\,. Such coordinates can be obtained changing the places of $u$ and $v$ of our coordinates $(u\,,v)$.
This means that the operator $\Delta^h$ in \cite{M-A-1} is different from ours in sign and the right sides 
of the equations have the opposite sign of ours. 

\smallskip

 As we know from \eqref{hat_K_kappa-K_kappa-mov_R42-tl}, $K$ and $\varkappa$ are invariant under a proper 
motion in $\RR^4_2$. Then, similarly to Theorem \ref{Thm-Nat_Eq_E_phi_R42-tl}\,, we have:
\begin{theorem}\label{Thm-Nat_Eq_K_kappa_R42-tl}
Let $\M$ be a minimal time-like surface of general type in $\RR^4_2$ parametrized by canonical coordinates.
Then the Gauss curvature $K$ and the curvature of the normal connection $\varkappa$ of $\M$, 
are solutions to the system of natural equations \eqref{Nat_Eq_K_kappa_R42-tl} of the minimal time-like 
surfaces in $\RR^4_2$.

If $\hat\M$ is obtained from $\M$ through a proper motion of the type \eqref{hat_M-M-prop_mov_R42-tl},
then it generates the same solution to the system \eqref{Nat_Eq_K_kappa_R42-tl}. 
\end{theorem}

 The functions $K$ and $\varkappa$ in \eqref{Nat_Eq_K_kappa_R42-tl} are scalar invariants,
but the system as a whole is not invariant since the hyperbolic Laplace operator $\Delta^h$ is not invariant,
as we noted previously. In order to obtain a completely invariant form of the natural equations, we use the
Laplace-Beltrami operator $\Delta^h_b$ for the surface $\M$, which is given by the equality:
\begin{equation}\label{def_Lpl-Beltr}
d*df = (\Delta^h_b f)\,dV \,,
\end{equation}
where "$*$" denotes the Hodge star operator, $d$ is the exterior differential and $dV\!=\! -E\, du\,\wedge\!~dv$ is 
the volume form. This operator is invariant by definition and in isothermal coordinates it is expressed through the
ordinary hyperbolic Laplace operator in $\RR^2$ as follows:
\begin{equation}\label{Lpl-Beltr_Lpl-tl}
\Delta^h_b = \frac{1}{E}\, \Delta^h .
\end{equation}
In canonical coordinates the equality \eqref{E_phi_K2_kappa2_R42-tl} is valid and therefore we have the following 
representation for this operator:
\begin{equation}\label{Lpl-Beltr_Lpl_R42-tl}
\Delta^h_b = \frac{1}{E}\, \Delta^h = -\sqrt[4]{ \vphantom{\mu^2} |K^2-\varkappa^2| }\; \Delta^h .
\end{equation}

Applying the last formula to the system \eqref{Nat_Eq_K_kappa_R42-tl}, it gets the form:
\begin{equation}\label{Nat_Eq_Beltr_K_kappa_R42-tl}
\begin{array}{lll}
\Delta^h_b \ln \big|K^2-\varkappa^2\big| &=& 8K\,; \\[1.5ex]
\Delta^h_b \ln \left|\ds\frac{\vphantom{\mu^2}K+\varkappa}{K-\varkappa}\right| &=& 4\varkappa\;;
\end{array}  \qquad\quad K^2-\varkappa^2\neq 0 \,.
\end{equation}
This form of the natural system of equations is completely invariant and it is also valid in arbitrary
(possibly non isothermal) local coordinates.  

 If we add and subtract consecutively the two equations in \eqref{Nat_Eq_Beltr_K_kappa_R42-tl}, then we 
obtain the system of natural equations in the form, obtained in \cite{S-2}\,: 
\begin{equation}\label{Nat_Eq_Sakaki_K_kappa_R42-tl}
\begin{array}{l}
\Delta^h_b \ln |K + \varkappa | = 2(2K+\varkappa) \,; \\[1.0ex]
\Delta^h_b \ln |K - \varkappa | = 2(2K-\varkappa) \,;
\end{array} \qquad\quad K^2-\varkappa^2\neq 0 \,.
\end{equation}

\smallskip

 Now, let us return to the Frenet type formulas \eqref{Frene_Phi_bar_Phi_m1_m2_R42-tl}. We shall prove that 
the equations \eqref{Nat_Eq_E_phi_R42-tl} are not only necessary but also sufficient conditions for a local existence  
of solutions to the system \eqref{Frene_Phi_bar_Phi_m1_m2_R42-tl}. This means that any quadruple 
$E$, $\varphi$, $\varepsilon$ and $\delta$, satisfying \eqref{Nat_Eq_E_phi_R42-tl}, at least locally coincides
with the corresponding quadruple obtained from a minimal time-like surface  of general type in $\RR^4_2$. 
Furthermore, the theorem for the uniqueness of the solution to \eqref{Frene_Phi_bar_Phi_m1_m2_R42-tl} gives that 
this surface is unique up to a proper motion in $\RR^4_2$. More concretely, we shall formulate and prove the Bonnet type
theorem for the system \eqref{Frene_Phi_bar_Phi_m1_m2_R42-tl}.
\begin{theorem}\label{Bone_Phi_bar_Phi_m1_m2_E_phi_R42-tl}
Let $E<0$ and $\varphi$ be real functions, defined in a domain $\D\subset\RR^2$ and let
$\varepsilon = \pm 1;\jj$ and $\delta = 1;\jj$ be such constants that the four quantities give a solution to 
the system of natural equations \eqref{Nat_Eq_E_phi_R42-tl} of the minimal time-like surfaces in $\RR^4_2$.
For any point $p_0\in\D$, there exists a neighborhood $\D_0\subset\D$ of $p_0$ and a map $\x: \D_0 \to \RR^4_2$,
such that $(\D_0,\x)$ is a minimal time-like surface of general type in $\RR^4_2$, parametrized by canonical coordinates,
for which the given function $E$ is the coefficient of the first fundamental form and having a representation of 
$\Phi'^\bot$ of the type \eqref{Phi_pn_m12_can_R42-tl} with the given $\varphi$, $\varepsilon$ and $\delta$.
If $(\hat\D_0,\hat\x)$ is another surface with the same properties, then there exists a subdomain $\tilde\D_0$
of $\D_0$ and $\hat\D_0$, containing $p_0$, so that the surface $(\tilde\D_0,\hat\x)$ is obtained from
$(\tilde\D_0,\x)$ through a proper motion in $\RR^4_2$ of the type \eqref{hat_M-M-prop_mov_R42-tl}.
\end{theorem}
\begin{proof}
 Let $(u,v)$ be the coordinates in $\D$ with $t=u+\jj v$ and $p_0\in\D$ be a fixed point with coordinates 
$t_0=u_0+\jj v_0$. Define the function $\beta$ in $\D$ by the equality 
$\beta = \frac{\jj}{2} \frac{\partial\varphi}{\partial t}$\,,
which coincides with \eqref{Nat_Eq_Codazzi_beta_E_phi_R42-tl}. Next we prove that the five quantities: 
$E$, $\varphi$, $\varepsilon$, $\delta$ and $\beta$ satisfy the integrability conditions for the system
\eqref{Frene_Phi_bar_Phi_m1_m2_R42-tl}.

 The integrability condition \eqref{Nat_Eq_Gauss_E_phi_R42-tl} is fulfilled, since it coincides with 
the first equation of the system of natural equations \eqref{Nat_Eq_E_phi_R42-tl}. The condition 
\eqref{Nat_Eq_Codazzi_beta_E_phi_R42-tl} is also fulfilled, because it coincides with the definition of
$\beta$. Using the equality $\frac{\partial^2}{\partial \bar t \partial t}=\frac{1}{4}\Delta^h$ and
the definition of $\beta$, we get:
\[
\Im \ds\frac{\partial \beta}{\partial \bar t} = 
\Im \left( \frac{\jj}{2} \frac{\partial^2 \varphi}{\partial \bar t \partial t} \right) = 
\Im \left( \frac{\jj}{8} \Delta^h\varphi \right) = \frac{1}{8} \Delta^h\varphi \,.
\]
Expressing in the last equality $\Delta^h\varphi$ by the second equation of the system of natural equations
\eqref{Nat_Eq_E_phi_R42-tl}, we obtain the integrability condition \eqref{Nat_Eq_Ricci_beta_E_phi_R42-tl}.

Now, consider the equalities \eqref{Frene_Phi_bar_Phi_m1_m2_R42-tl}, as a system of PDE's with an unknown
complex (over $\DD$) vector function $\Phi$ and unknown real vector functions $\m_1$ and $\m_2$. 
The integrability conditions of the system are exactly \eqref{Nat_Eq_Gauss_E_phi_R42-tl}, 
\eqref{Nat_Eq_Codazzi_beta_E_phi_R42-tl} and \eqref{Nat_Eq_Ricci_beta_E_phi_R42-tl}, according to
Proposition \ref{Frene_integr_cond_R42-tl}, which are fulfilled. Consequently, the system has at least locally 
a unique solution by the given initial conditions.

 Let $(\x_{u;0}\,,\x_{v;0}\,,\n_{1;0}\,,\n_{2;0})$ be a right oriented quadruple of mutually orthogonal vectors
in $\RR^4_2$, such that $\x_{u;0}^2=-\x_{v;0}^2=E(t_0)$ and $\n_{1;0}^2=-\n_{2;0}^2=-1$. Using this basis, we 
shall obtain an isotropic basis $(\Phi_0\,, \bar\Phi_0\,,\m_{1;0}\,,\m_{2;0})$ through the formulas
\eqref{Phi_def-tl} and \eqref{m12_n12_R42-tl}. Thus we define 
$\Phi_0 = \x_{u;0}+\jj\x_{v;0}$\,, $\m_{1;0}=\frac{\n_{1;0} + \n_{2;0}}{2}$\, and\,
$\m_{1;0}=\frac{-\n_{1;0} + \n_{2;0}}{2}$\,.
Consider the system \eqref{Frene_Phi_bar_Phi_m1_m2_R42-tl} with initial conditions for $t_0$ the obtained vectors
$\Phi_0$\,, $\m_{1;0}$\, and\, $\m_{2;0}$\,. This system has a solution consisting of three functions: 
$\Phi$\,, $\m_1$\, and\, $\m_2$, defined in a neighborhood $\D_0$ of $t_0$ satisfying the initial conditions 
$\Phi(t_0)=\Phi_0$\,, $\m_1(t_0)=\m_{1;0}$\, and\, $\m_2(t_0)=\m_{2;0}$\,. Next we prove that in a neighborhood
$\D_0$ the obtained solution satisfies the identities:
$\Phi^2=\m_1^2=\m_2^2=0$\,, $\Phi\m_1=\Phi\m_2=0$\,, $\|\Phi\|^2=2E$ and $\m_1\m_2=\frac{1}{2}$\,.
Note that by the given initial conditions these equalities are fulfilled for $t=t_0$\,.  
Let us consider the following functions:
$f_1=\Phi^2$\,, $f_2=\m_1^2$\,, $f_3=\m_2^2$\,,  $f_4=\Phi\m_1$\,, $f_5=\Phi\m_2$\,, 
$f_6=\|\Phi\|^2-2E$\, and\, $f_7=\m_1\m_2-\frac{1}{2}$\, and add to them their conjugate functions: 
$\bar f_1$\,, $\bar f_4$\, and\, $\bar f_5$. The other four functions are real-valued and that is why we do not 
write their conjugate functions. These ten functions are zero for $t=t_0$\,. Further, we see that these functions 
give a solution to a homogeneous system of PDE's of the first order solved with respect to the derivatives.
For $\frac{\partial f_1}{\partial t}$ we have from the second equation in \eqref{Frene_Phi_bar_Phi_m1_m2_R42-tl}:
\[
\begin{array}{lll}
\ds\frac{\partial f_1}{\partial t} 
&=& 2\ds\frac{\partial\Phi}{\partial t}\Phi = 
2\left( \frac{\partial\ln |E|}{\partial t}\Phi + 
\delta\e^{\jj\frac{\varphi}{2}}\m_1 + \varepsilon\delta\e^{-\jj\frac{\varphi}{2}}\m_2 \right)\Phi \\[2ex]
&=& 
2\ds\frac{\partial\ln |E|}{\partial t}f_1 + 2\delta\e^{\jj\frac{\varphi}{2}}f_4 + 2\varepsilon\delta\e^{-\jj\frac{\varphi}{2}}f_5\,.
\end{array}
\]
The first equation in \eqref{Frene_Phi_bar_Phi_m1_m2_R42-tl} gives that:
\[
\ds\frac{\partial f_1}{\partial \bar t} = 2\frac{\partial\Phi}{\partial \bar t}\Phi = 0\,.
\]
For $\frac{\partial f_2}{\partial t}$ we apply the third equation in \eqref{Frene_Phi_bar_Phi_m1_m2_R42-tl}:
\[
\frac{\partial f_2}{\partial t} = 2\frac{\partial m_1}{\partial t}\m_1 = 
2\left( -\frac{\varepsilon\delta\e^{-\jj\frac{\varphi}{2}}}{4E}\bar\Phi + \beta\m_1 \right)\m_1 = 
-2\frac{\varepsilon\delta\e^{-\jj\frac{\varphi}{2}}}{4E}\bar f_4 + 2\beta f_2 \,.
\]
Since $f_2$ is real-valued, then the formula for $\frac{\partial f_2}{\partial \bar t}$ is obtained from the above 
equality by a conjugation and we omit it. Using the fourth equation in \eqref{Frene_Phi_bar_Phi_m1_m2_R42-tl},
we obtain for $\frac{\partial f_3}{\partial t}$:
\[
\frac{\partial f_3}{\partial t} = 2\frac{\partial m_2}{\partial t}\m_2 = 
2\left( -\frac{\delta\e^{\jj\frac{\varphi}{2}}}{4E}\bar\Phi - \beta\m_2 \right)\m_2 = 
-2\frac{\delta\e^{\jj\frac{\varphi}{2}}}{4E}\bar f_5 - 2\beta f_3 \,.
\]
The formula for $\frac{\partial f_3}{\partial \bar t}$ is obtained from the above equality by a conjugation.
For $\frac{\partial f_4}{\partial t}$ we apply the second and the third equation in 
\eqref{Frene_Phi_bar_Phi_m1_m2_R42-tl} and find:
\[
\begin{array}{lll}
\ds\frac{\partial f_4}{\partial t} 
&=& \ds \frac{\partial\Phi}{\partial t}\m_1 + \Phi\frac{\partial\m_1}{\partial t} \\[2ex]
&=& 
\ds \left( \frac{\partial\ln |E|}{\partial t}\Phi +  
\delta\e^{\jj\frac{\varphi}{2}}\m_1 + \varepsilon\delta\e^{-\jj\frac{\varphi}{2}}\m_2 \right)\m_1 + 
\Phi \left( -\frac{\varepsilon\delta\e^{-\jj\frac{\varphi}{2}}}{4E}\bar\Phi + \beta\m_1 \right)  \\[2ex]
&=& 
\ds\frac{\partial\ln |E|}{\partial t}f_4 + \delta\e^{\jj\frac{\varphi}{2}}f_2 + 
\varepsilon\delta\e^{-\jj\frac{\varphi}{2}}\left(f_7+\frac{1}{2}\right) - 
\frac{\varepsilon\delta\e^{-\jj\frac{\varphi}{2}}}{4E}\left(f_6+2E\right) + \beta f_4  \\[2ex]
&=&
\ds \left( \frac{\partial\ln |E|}{\partial t}+\beta \right)f_4 + \delta\e^{\jj\frac{\varphi}{2}}f_2 + 
\varepsilon\delta\e^{-\jj\frac{\varphi}{2}}f_7 - \frac{\varepsilon\delta\e^{-\jj\frac{\varphi}{2}}}{4E}f_6 \,.
\end{array}
\]
Then for $\frac{\partial f_4}{\partial \bar t}$ we get:
\[
\ds\frac{\partial f_4}{\partial \bar t} 
= \ds \frac{\partial\Phi}{\partial \bar t}\m_1 + \Phi\frac{\partial\m_1}{\partial \bar t} = 
\Phi\frac{\partial\m_1}{\partial \bar t} =  
\Phi \left( -\frac{\bar\varepsilon\bar\delta\e^{\jj\frac{\varphi}{2}}}{4E}\Phi + \bar\beta\m_1 \right) =
-\ds\frac{\bar\varepsilon\bar\delta\e^{\jj\frac{\varphi}{2}}}{4E}f_1 + \bar\beta f_4  \,.
\]
For $\frac{\partial f_5}{\partial t}$ we apply the second and the fourth equation in \eqref{Frene_Phi_bar_Phi_m1_m2_R42-tl} 
and obtain:
\[
\begin{array}{lll}
\ds\frac{\partial f_5}{\partial t} 
&=& \ds \frac{\partial\Phi}{\partial t}\m_2 + \Phi\frac{\partial\m_2}{\partial t} \\[2ex]
&=& 
\ds \left( \frac{\partial\ln |E|}{\partial t}\Phi +  
\delta\e^{\jj\frac{\varphi}{2}}\m_1 + \varepsilon\delta\e^{-\jj\frac{\varphi}{2}}\m_2 \right)\m_2 + 
\Phi \left( -\frac{\delta\e^{\jj\frac{\varphi}{2}}}{4E}\bar\Phi - \beta\m_2 \right)  \\[2ex]
&=& 
\ds\frac{\partial\ln |E|}{\partial t}f_5 + \delta\e^{\jj\frac{\varphi}{2}}\left(f_7+\frac{1}{2}\right) + 
\varepsilon\delta\e^{-\jj\frac{\varphi}{2}}f_3 - 
\frac{\delta\e^{\jj\frac{\varphi}{2}}}{4E}\left(f_6+2E\right) - \beta f_5  \\[2ex]
&=&
\ds \left( \frac{\partial\ln |E|}{\partial t}-\beta \right)f_5 + \delta\e^{\jj\frac{\varphi}{2}}f_7 + 
\varepsilon\delta\e^{-\jj\frac{\varphi}{2}}f_3 - \frac{\delta\e^{\jj\frac{\varphi}{2}}}{4E}f_6 \,.
\end{array}
\]
Then we calculate $\frac{\partial f_5}{\partial \bar t}$:
\[
\ds\frac{\partial f_5}{\partial \bar t} 
= \ds \frac{\partial\Phi}{\partial \bar t}\m_2 + \Phi\frac{\partial\m_2}{\partial \bar t} = 
\Phi\frac{\partial\m_2}{\partial \bar t} =  
\Phi \left( -\frac{\bar\delta\e^{-\jj\frac{\varphi}{2}}}{4E}\Phi - \bar\beta\m_2 \right) =
-\ds\frac{\bar\delta\e^{-\jj\frac{\varphi}{2}}}{4E}f_1 - \bar\beta f_5  \,.
\]
Using the first and the second equation in \eqref{Frene_Phi_bar_Phi_m1_m2_R42-tl}, we get for 
$\frac{\partial f_6}{\partial t}$\,
\[
\begin{array}{lll}
\ds\frac{\partial f_6}{\partial t} 
&=& \ds \frac{\partial\Phi}{\partial t}\bar\Phi + \Phi\frac{\partial\bar\Phi}{\partial t} - 2\frac{\partial E}{\partial t} =
\frac{\partial\Phi}{\partial t}\bar\Phi - 2\frac{\partial E}{\partial t}  \\[2ex]
&=& 
\ds \left( \frac{\partial\ln |E|}{\partial t}\Phi +  
\delta\e^{\jj\frac{\varphi}{2}}\m_1 + \varepsilon\delta\e^{-\jj\frac{\varphi}{2}}\m_2 \right)\bar\Phi - 
2\frac{\partial E}{\partial t}  \\[2ex]
&=& 
\ds \frac{\partial\ln |E|}{\partial t}\left(f_6+2E\right) + \delta\e^{\jj\frac{\varphi}{2}}\bar f_4 + 
\varepsilon\delta\e^{-\jj\frac{\varphi}{2}}\bar f_5 - 2\frac{\partial E}{\partial t}  \\[2ex]
&=&
\ds \frac{\partial\ln |E|}{\partial t}f_6 + \delta\e^{\jj\frac{\varphi}{2}}\bar f_4 + 
\varepsilon\delta\e^{-\jj\frac{\varphi}{2}}\bar f_5 \,.
\end{array}
\]
The formula for $\frac{\partial f_6}{\partial \bar t}$ is obtained from the above equality by a conjugation.
Finally, using the third and the fourth equation in \eqref{Frene_Phi_bar_Phi_m1_m2_R42-tl}, we obtain for 
$\frac{\partial f_7}{\partial t}$:
\[
\begin{array}{lll}
\ds\frac{\partial f_7}{\partial t} 
&=& \ds \frac{\partial\m_1}{\partial t}\m_2 + \m_1\frac{\partial\m_2}{\partial t} =
\ds  \left( -\frac{\varepsilon\delta\e^{-\jj\frac{\varphi}{2}}}{4E}\bar\Phi + \beta\m_1 \right)\m_2 + 
\m_1 \left( -\frac{\delta\e^{\jj\frac{\varphi}{2}}}{4E}\bar\Phi - \beta\m_2 \right)  \\[2ex]
&=& 
\ds -\frac{\varepsilon\delta\e^{-\jj\frac{\varphi}{2}}}{4E}\bar f_5 - \frac{\delta\e^{\jj\frac{\varphi}{2}}}{4E}\bar f_4 \,.  
\end{array}
\]
The formula for $\frac{\partial f_7}{\partial \bar t}$ is obtained from the above equality by a conjugation. 
The formulas for $\bar f_1$\,, $\bar f_4$\, and\, $\bar f_5$ are also obtained in a similar way from the corresponding 
formulas for $f_1$\,, $f_4$\, and \, $f_5$\,.

 It follows from the above calculations that the functions $f_1$\,, $f_2$\,, $f_3$\,, $f_4$\,, $f_5$\,, $f_6$\,, 
$f_7$\,, $\bar f_1$\,, $\bar f_4$\, and\, $\bar f_5$ are solutions to a homogeneous system PDE's of the first order
solved with respect to the derivatives with zero initial conditions. Consequently, the ten functions under 
consideration are identically zero. Using this fact, we shall prove that there exists a minimal time-like 
surface of general type in $\RR^4_2$, for which the obtained function $\Phi$ is given by the formula \eqref{Phi_def-tl}
and the canonical isotropic normal basis coincides with the obtained pair $(\m_1\,,\m_2)$.

 Now we have to check that the function $\Phi$ satisfies the conditions of Theorem \ref{x_Phi-thm-tl}\,.
The equality $f_1=0$ gives the first one of the conditions \eqref{Phi_cond-tl}. The equality $f_6=0$ 
gives $\|\Phi\|^2=2E<0$\,, which is the second condition in \eqref{Phi_cond-tl}. The first equation in
\eqref{Frene_Phi_bar_Phi_m1_m2_R42-tl} implies that $\Phi$ is holomorphic (over $\DD$), which gives the third condition.
Thus the function $\Phi$ satisfies all three conditions \eqref{Phi_cond-tl} and therefore Theorem \ref{x_Phi-thm-tl}\,
is applicable. According to this theorem, there exists a neighborhood $\D_0$ of $t_0$ and a vector function $\x$
defined in $\D_0$, such that \eqref{Phi_def-tl} is valid and $(\D_0,\x)$ is a time-like surface with isothermal 
coordinates $(u,v)$. It remains to see that this surface has the required properties.

 First we note that $\Phi$ is a holomorphic function and in view of Theorem \ref{Min_x_Phi-thm-tl}\,,
the surface $(\D_0,\x)$ is minimal. The equation $\|\Phi\|^2=2E$ shows that the coefficient $E$ of the first 
fundamental form of $(\D_0,\x)$ coincides with the given function $E$. Using that the functions $f_i$ are zero, 
we see that the pair $(\m_1\,,\m_2)$ is a normalized isotropic basis in the normal space of the obtained surface.
Squaring the second equality in \eqref{Frene_Phi_bar_Phi_m1_m2_R42-tl}, we get ${\Phi'^\bot}^2=\varepsilon$.
This means that the given coordinates are canonical for the obtained surface, according to Definition \ref{Can-def_R42-tl}\,.
Under the given initial conditions, the quadruple $(\x_u\,,\x_u\,,\m_1\,,\m_2)$ is right oriented for $t=t_0$\,,
and consequently, in a neighborhood of $t_0$\,. The second equality in \eqref{Frene_Phi_bar_Phi_m1_m2_R42-tl} 
gives that $\Phi'^\bot$ has a representation of the type \eqref{Phi_pn_m12_can_R42-tl} via the obtained vectors
$\m_1$ and $\m_2$ and via the given functions $\varphi$, $\varepsilon$ and $\delta$. Consequently, the surface 
$(\D_0,\x)$ satisfies the conditions of the theorem. 

\smallskip

 Next we prove the uniqueness of the surface. To this end, suppose that $(\hat\D_0,\hat\x)$ is another surface, 
having the properties of $(\D_0,\x)$. Then $\big(\hat\x_u(t_0),\hat\x_v(t_0),\hat\m_1(t_0),\hat\m_2(t_0)\big)$ 
is a right oriented quadruple of vectors in $\RR^4_2$, with the properties:\, 
$\hat\x_u^2(t_0)=-\hat\x_v^2(t_0)=E(t_0)$\,,\ \; $\hat\m_1^2(t_0)=\hat\m_2^2(t_0)=0$\, and\, 
$\hat\m_1(t_0)\hat\m_2(t_0)=\frac{1}{2}$\,. Let $A$ be the unique linear operator transforming the quadruple 
$\big(\x_u(t_0),\x_v(t_0),\m_1(t_0),\m_2(t_0)\big)$\, in\,  
$\big(\hat\x_u(t_0),\hat\x_v(t_0),\hat\m_1(t_0),\hat\m_2(t_0)\big)$. 
Then $A$ is a proper motion from $\mathbf{SO}(2,2,\RR)$. Define the map $\hat{\hat\x}$ from $\D_0$ to $\RR^4_2$ 
by $\hat{\hat\x}=A\x$. Then $(\D_0,\hat{\hat\x})$ is a minimal time-like surface, obtained from $(\D_0,\x)$ 
through a proper motion. Hence $(u,v)$ are canonical coordinates for $(\D_0,\hat{\hat\x})$ and the quantities
$E$, $\varphi$, $\varepsilon$ and $\delta$ of $(\D_0,\hat{\hat\x})$ coincide with those for $(\D_0,\x)$.
It follows from \eqref{Nat_Eq_Codazzi_beta_E_phi_R42-tl} that the coefficient $\beta$ for $(\D_0,\hat{\hat\x})$ 
also coincides with that for $(\D_0,\x)$. Thus we obtain that the corresponding functions
$\hat{\hat\Phi}$, ${\hat{\hat\m}}_1$ and ${\hat{\hat\m}}_2$ are also a solution to the system
\eqref{Frene_Phi_bar_Phi_m1_m2_R42-tl}. It follows from the definition of the motion $A$ that
$\hat{\hat\Phi}$, ${\hat{\hat\m}}_1$ and ${\hat{\hat\m}}_2$ satisfy the same initial conditions for $t=t_0$,
as the functions $\hat\Phi$, $\hat\m_1$ and $\hat\m_2$, of $(\hat\D_0,\hat\x)$. Since the system 
\eqref{Frene_Phi_bar_Phi_m1_m2_R42-tl} has locally a unique solution under the given initial conditions, then 
it follows that there exists a connected neighborhood $\tilde\D_0$ of $t_0$, in which $\hat\Phi=\hat{\hat\Phi}$. 
But $\hat{\hat\x}=A\x$ implies that $\hat{\hat\Phi}=A\Phi$ and consequently $\hat\Phi=A\Phi$ in $\tilde\D_0$. 
For a connected neighborhood, the last equality is equivalent to $\hat\x = A\x + \vb$ for a given 
$\vb\in\RR^4_2$ and for all $t\in\tilde\D_0$. This proves the assertion.
\end{proof}

 Our next goal is to prove a theorem of Bonnet type for the system of natural equations \eqref{Nat_Eq_K_kappa_R42-tl}
as well. Some results of this kind are obtained in \cite{M-A-1} and \cite{S-2}, under the assumption $K^2-\varkappa^2>0$\,.
Now, using Theorem \ref{Bone_Phi_bar_Phi_m1_m2_E_phi_R42-tl}, we can obtain theorems of Bonnet type for
the system \eqref{Nat_Eq_K_kappa_R42-tl} in all three types of surfaces.

 Let us begin with the surfaces of the first type.

\begin{theorem}\label{Bone_Phi_bar_Phi_m1_m2_K_kappa_kind1_R42-tl}
Let $K$ and $\varkappa$ be real functions, defined in a domain $\D\subset\RR^2$, satisfying the inequality 
$K^2-\varkappa^2>0$ and let the pair $(K,\varkappa)$ be a solution to the system of natural equations
\eqref{Nat_Eq_K_kappa_R42-tl} for the minimal time-like surfaces in $\RR^4_2$.

 Then in a neighborhood of any point $t_0\in\D$, there exists a unique up to a proper motion minimal time-like surface  
$\M$ of the first tye, according to Definition \ref{Min_Surf_kind123-def-tl}\,, given in canonical coordinates, for which
the given functions $K$ and $\varkappa$ are the Gauss curvature and the curvature of the normal connection, respectively.
\end{theorem}
\begin{proof}
 In order to prove the existence of the desired surface, we will reduce the statement to the corresponding statement
of Theorem~\ref{Bone_Phi_bar_Phi_m1_m2_E_phi_R42-tl}\,. To this end we will show how the functions $K$ and 
$\varkappa$ generate the functions $E$, $\varphi$, $\varepsilon$ and $\delta$, satisfying the conditions of
Theorem~\ref{Bone_Phi_bar_Phi_m1_m2_E_phi_R42-tl}\,. First we define the function $E$ by the first formula
in \eqref{E_phi_K2_kappa2-kind12_R42-tl} and put $\varepsilon= 1$\,. Then we determine uniquely $\delta = 1;\jj$\,, 
depending on the sign of $K$, so that $-\varepsilon|\delta|^2K>0$\,. Finally we define uniquely $\varphi$, 
so that the following equality holds $-\varepsilon|\delta|^2E^2\varkappa=\sinh\varphi$\,, which is equivalent to 
the second equality in \eqref{K_kappa_E_phi-kind12_R42-tl}. It follows from the definition of $E$ that 
$(E^2K)^2-(E^2\varkappa)^2=1$\,. The equalities and the inequalities obtained so far imply that
$-\varepsilon|\delta|^2E^2K=\cosh\varphi$\,, which is equivalent to  the first equality in \eqref{K_kappa_E_phi-kind12_R42-tl}.
Under the condition $\varepsilon= 1$\,, we already know that the two formulas in \eqref{K_kappa_E_phi-kind12_R42-tl}
are equivalent to the two formulas in \eqref{K_kappa_c12-can_R42-tl}. Further, equalities \eqref{K_kappa_E_phi-kind12_R42-tl}
imply the equalities \eqref{E_phi_K2_kappa2_R42-tl}.

Consequently, the so defined quantities $E$, $\varphi$, $\varepsilon$ and $\delta$ satisfy the equalities 
\eqref{K_kappa_c12-can_R42-tl} and \eqref{E_phi_K2_kappa2_R42-tl}. Since $K$ and $\varkappa$ satisfy the system 
\eqref{Nat_Eq_K_kappa_R42-tl}, then applying equalities \eqref{K_kappa_c12-can_R42-tl} and \eqref{E_phi_K2_kappa2_R42-tl} 
to this system, we obtain that $E$, $\varphi$, $\varepsilon=1$ and $\delta$ satisfy the system \eqref{Nat_Eq_E_phi_R42-tl}.
Now, apllying Theorem~\ref{Bone_Phi_bar_Phi_m1_m2_E_phi_R42-tl} to these quantities, we obtain a minimal time-like surface
$\M$ with invariants $E$, $\varphi$, $\varepsilon=1$ and $\delta$. Under the condition $\varepsilon=1$ it follows that 
$\M$ is of the first type according to Definition \ref{Min_Surf_kind123-def-tl}\,. Furthermore, the surface $\M$ satisfies
formulas \eqref{K_kappa_c12-can_R42-tl}. Hence, the Gauss curvature and the curvature of the normal connection of $\M$
coincide with the given functions $K$ and $\varkappa$.

Next we prove that the surface $\M$ is unique. Assume that $\hat\M$ is another minimal surface of the first type, 
given in canonical coordinates and having the same $K$ and $\varkappa$. Under the fixed $\varepsilon=1$, the 
quantities $E$, $\varphi$ and $\delta$ are determined uniquely by $K$ and $\varkappa$. Therefore, $\hat\M$ has
the same invariants $E$, $\varphi$, $\varepsilon$ and $\delta$ as the surface $\M$. Then, the uniqueness part of
Theorem~\ref{Bone_Phi_bar_Phi_m1_m2_E_phi_R42-tl} implies that $\hat\M$ is obtained from $\M$ through a proper 
motion.
\end{proof}
Next we consider the theorem of Bonnet type for the surfaces of the second type.
\begin{theorem}\label{Bone_Phi_bar_Phi_m1_m2_K_kappa_kind2_R42-tl}
Let $K$ and $\varkappa$ be real functions, defined in a domain $\D\subset\RR^2$, satisfying the inequality 
$K^2-\varkappa^2>0$ and let the pair $(K,\varkappa)$ is a solution to the system of natural equations
\eqref{Nat_Eq_K_kappa_R42-tl} for the minimal time-like surfaces in $\RR^4_2$.

 Then in a neighborhood of any point $t_0\in\D$, there exists a unique up to a proper motion minimal time-like surface
$\M$ of the second type, according to Definition \ref{Min_Surf_kind123-def-tl}\,, given in canonical coordinates,
whose Gauss curvature and curvature of the normal connection are the given functions $K$ and $\varkappa$, respectively.
\end{theorem}
\begin{proof}
 As in the case of surfaces of the first type, we will reduce the assertion to the corresponding assertion of 
Theorem~\ref{Bone_Phi_bar_Phi_m1_m2_E_phi_R42-tl}\,. Next we show how $K$ and $\varkappa$ generate the quantities 
$E$, $\varphi$, $\varepsilon$ and $\delta$, satisfying the conditions of Theorem~\ref{Bone_Phi_bar_Phi_m1_m2_E_phi_R42-tl}\,.
First we define the function $E$ by the first formula in \eqref{E_phi_K2_kappa2-kind12_R42-tl} and put $\varepsilon= -1$\,.
Further we determine uniquely $\delta = 1;\jj$\,, depending on the sign of $K$, so that $-\varepsilon|\delta|^2K>0$\,. 
Finally we determine uniquely $\varphi$, so that $-\varepsilon|\delta|^2E^2\varkappa=\sinh\varphi$\,, which is equivalent
to the second equation in \eqref{K_kappa_E_phi-kind12_R42-tl}. t follows from the definition of $E$ that 
$(E^2K)^2-(E^2\varkappa)^2=1$\,. Using these equalities and inequalities we conclude that 
$-\varepsilon|\delta|^2E^2K=\cosh\varphi$\,, which is equivalent to the first equality in \eqref{K_kappa_E_phi-kind12_R42-tl}. 
Under the condition $\varepsilon= -1$\,, the two formulas in \eqref{K_kappa_E_phi-kind12_R42-tl}
are equivalent to the two formulas in \eqref{K_kappa_c12-can_R42-tl}. Formulas \eqref{K_kappa_E_phi-kind12_R42-tl} imply
formulas \eqref{E_phi_K2_kappa2_R42-tl}. Consequently, the quantities $E$, $\varphi$, $\varepsilon$ and $\delta$
satisfy equalities \eqref{K_kappa_c12-can_R42-tl} and \eqref{E_phi_K2_kappa2_R42-tl}. Since $K$ and $\varkappa$ satisfy
the system \eqref{Nat_Eq_K_kappa_R42-tl}, applying equalities \eqref{K_kappa_c12-can_R42-tl} and 
\eqref{E_phi_K2_kappa2_R42-tl} to this system, we obtain that $E$, $\varphi$, $\varepsilon=-1$ and $\delta$ satisfy
the system \eqref{Nat_Eq_E_phi_R42-tl}. Now, we apply Theorem~\ref{Bone_Phi_bar_Phi_m1_m2_E_phi_R42-tl} to these
quantities and obtain a minimal time-like surface $\M$, with invariants $E$, $\varphi$, $\varepsilon=-1$ and $\delta$.
Since $\varepsilon=-1$, then it follows that $\M$ is of the second type according Definition \ref{Min_Surf_kind123-def-tl}\,.
The surface $\M$ satisfies formulas \eqref{K_kappa_c12-can_R42-tl}. 
Hence the Gauss curvature and the curvature of the normal connection of $\M$ coincide with the given functions $K$ and $\varkappa$, respectively.

 To prove the uniqueness of the surface $\M$, let $\hat\M$ be another surface of the second type, given in canonical 
coordinates, whose Gauss curvature and the curvature of the normal connection are $K$ and $\varkappa$, respectively.
Under the condition $\varepsilon=-1$, the quantities $E$, $\varphi$ and $\delta$ are determined uniquely by $K$ 
and $\varkappa$. Therefore, the surface $\hat\M$ has the same invariants $E$, $\varphi$, $\varepsilon$ and $\delta$ as 
the surface $\M$. Then the uniqueness part of Theorem~\ref{Bone_Phi_bar_Phi_m1_m2_E_phi_R42-tl} implies that $\hat\M$ 
is obtained from $\M$ through a proper motion.
\end{proof}

 If $\M$ is a minimal time-like surface of the first type in $\RR^4_2$, given in canonical coordinates, then
it follows from Theorem~\ref{Thm-Nat_Eq_K_kappa_R42-tl} that the pair of invariants $(K,\varkappa)$ is a solution 
to the system \eqref{Nat_Eq_K_kappa_R42-tl}. Then Theorem \ref{Bone_Phi_bar_Phi_m1_m2_K_kappa_kind2_R42-tl} 
implies that $\M$ corresponds to a minimal time-like surface of the second type with the same invariants
$(K,\varkappa)$. We have already obtained a correspondence of the same type in Theorem~\ref{Can_antimov_conj_R42-tl}\,.
Next we show that up to a proper motion there exists only one such correspondence. Therefore, the correspondence 
just described, coincides with that obtained in Theorem~\ref{Can_antimov_conj_R42-tl}\,.

 For the sake of brevity, further we introduce a short denotation for the surface, obtained from another surface 
through an anti-isometry, preserving the orientation in $\RR^4_2$.
\begin{dfn}\label{dfh-M^a_R42-tl}
Let $\M=(\D,\x)$ be a minimal time-like surface in $\RR^4_2$, given in isothermal coordinates, and 
$\hat\x(s)=A\x(\jj s)+\vb$\,, where $A$ is a linear anti-isometry in $\RR^4_2$, such that $\det A = 1$\; 
and\; $\vb \in \RR^4_2$. Then $\M^a$ denotes any surface of the type $(\jj\D,\hat\x)$.
\end{dfn}

Note that, if $A_1$ and $A_2$ are two different anti-isometries of this type, we obtain two different surfaces 
$\M^a_1$ and $\M^a_2$. The surface $\M^a_2$ is obtained from the surface $\M^a_1$ through the transformation 
$\ds A_2^{\vphantom{-1}} A_1^{-1}$, which is a proper motion in $\RR^4_2$. Consequently, $\M^a$ is determined 
uniquely from $\M$ up to a proper motion. Using this denotation, for the two surfaces $\hat\M$ and $\M$
in Theorem~\ref{Can_antimov_conj_R42-tl} we can write the relation $\hat\M \equiv \bar\M^a$.

For the surfaces of the first and the second type we will prove a theorem, which is inverse to 
Theorem~\ref{Can_antimov_conj_R42-tl}\,.
\begin{theorem}\label{M1_M2_same_K_kappa_R42-tl}
Let $\M_1=(\D,\x_1(t))$ be a minimal time-like surface of the first type and  $\M_2=(\D,\x_2(t))$ be a minimal 
time-like surface of the second type in $\RR^4_2$, defined in $\D$, $t\in\D$ being canonical coordinates
for both surfaces $\M_1$ and $\M_2$. Let both surfaces have the same Gauss curvature $K$ and the same curvature 
of the normal connection $\varkappa$ considered as functions  of the coordinates $t$.

 Then the surfaces $\M_1$ and $\M_2$ cannot be obtained one from the other through a motion 
(including an improper motion) in $\RR^4_2$. In a sufficiently small neighborhood of any point $t_0\in\D$ the 
following relation $\M_2 \equiv \bar\M_1^a$ holds.
\end{theorem}
\begin{proof}
 The surfaces $\M_1$ and $\M_2$ cannot be obtained one from the other through a motion in $\RR^4_2$, 
since they are of a different type. According to Theorem~\ref{kind123_Change_Move-tl}\, the type of a surface 
is invariant under any motion in $\RR^4_2$.

On the other hand, let $\hat\M$ denote a surface obtained from the conjugate surface of $\M_1$, through an 
anti-isometry preserving the orientation of $\RR^4_2$. Then, according to Theorem~\ref{Can_antimov_conj_R42-tl}\,,\, 
$\hat\M$ is of the second type with canonical coordinates $t$ and has the same invariants $K$ and $\varkappa$, 
as functions of these coordinates. It follows from the uniqueness part of 
Theorem~\ref{Bone_Phi_bar_Phi_m1_m2_K_kappa_kind2_R42-tl} that $\hat\M$ and $\M_2$ are locally related by a proper 
motion in $\RR^4_2$. The composition of a proper motion with an anti-isometry, preserving the orientation of $\RR^4_2$,
is again an anti-isometry, preserving the orientation of $\RR^4_2$. This gives the relation $\M_2 \equiv \bar\M_1^a$.
\end{proof}

Now we will prove the corresponding theorem of Bonnet type for the minimal time-like surfaces of the third type.
\begin{theorem}\label{Bone_Phi_bar_Phi_m1_m2_K_kappa_kind3_R42-tl}
Let $K$ and $\varkappa$ be real functions, defined in a domain $\D\subset\RR^2$, satisfying the inequality
$K^2-\varkappa^2<0$ and let the pair $(K,\varkappa)$ be a solution to the system of natural equations 
\eqref{Nat_Eq_K_kappa_R42-tl} of the minimal time-like surfaces in $\RR^4_2$.

 Then in a neighborhood of any point $t_0\in\D$, there exists a unique up to a proper motion minimal time-like 
surface $\M$ of the third type, according to Definition \ref{Min_Surf_kind123-def-tl}\,, given in canonical 
coordinates, whose Gauss curvature and  curvature of the normal connection are the given functions $K$ and 
$\varkappa$, respectively.
\end{theorem}
\begin{proof} The scheme of the proof is similar to the previous cases.
First we define the function $E$ by the first formula in \eqref{E_phi_K2_kappa2-kind3_R42-tl} and put $\varepsilon=\jj$\,.
Further we determine uniquely $\delta = 1;\jj$\,, depending on the sign of $\varkappa$, so that 
$|\delta|^2\varkappa>0$\,. Finally we determine uniquely $\varphi$, so that $|\delta|^2E^2K=\sinh\varphi$\,,
which is equivalent to the first equality in \eqref{K_kappa_E_phi-kind3_R42-tl}. It follows from the definition of $E$
that $(E^2\varkappa)^2-(E^2K)^2=1$\,. In view of the so obtained equalities and inequalities we get 
$|\delta|^2E^2\varkappa=\cosh\varphi$\,, which is equivalent to the second equality in \eqref{K_kappa_E_phi-kind3_R42-tl}. 
Under the condition $\varepsilon=\jj$\,, the two formulas in \eqref{K_kappa_E_phi-kind3_R42-tl} 
are equivalent to the two formulas in \eqref{K_kappa_c12-can_R42-tl}. Further, formulas \eqref{K_kappa_E_phi-kind3_R42-tl} 
imply formulas \eqref{E_phi_K2_kappa2_R42-tl}. From here we conclude that the quantities $E$, $\varphi$, $\varepsilon$ 
and $\delta$ satisfy equalities \eqref{K_kappa_c12-can_R42-tl} and \eqref{E_phi_K2_kappa2_R42-tl}. Since $K$ and $\varkappa$ 
satisfy the system \eqref{Nat_Eq_K_kappa_R42-tl}, then applying equalities \eqref{K_kappa_c12-can_R42-tl} and 
\eqref{E_phi_K2_kappa2_R42-tl} to this system, we obtain that $E$, $\varphi$, $\varepsilon$ and $\delta$ satisfy
the system \eqref{Nat_Eq_E_phi_R42-tl}. Now we apply Theorem~\ref{Bone_Phi_bar_Phi_m1_m2_E_phi_R42-tl} under the condition 
$\varepsilon=\jj$\, and obtain a minimal time-like surface $\M$ of the third type with invariants $E$, $\varphi$, 
$\varepsilon$ and $\delta$. This surface $\M$ satisfies formulas \eqref{K_kappa_c12-can_R42-tl}, which gives that 
the Gauss curvature and the curvature of the normal connection are the given functions $K$ and $\varkappa$.

Finally we prove that the surface  $\M$ is unique. Assume that $\hat\M$ is another minimal time-like surface of  the third
type, given in canonical coordinates, whose Gauss curvature and curvature of the normal connection are $K$ and $\varkappa$,
respectively. Under the condition $\varepsilon=\jj$, the quantities $E$, $\varphi$ and $\delta$ are determined uniquely 
through $K$ and $\varkappa$. Therefore, $\hat\M$ has the same invariants $E$, $\varphi$, $\varepsilon$ and $\delta$ as 
the surface $\M$. Now the uniqueness part of Theorem~\ref{Bone_Phi_bar_Phi_m1_m2_E_phi_R42-tl} implies that $\hat\M$ 
is obtained from $\M$ through a proper motion.
\end{proof}
 
\smallskip

As an application of the proved theorems we consider the question: \emph{When two minimal time-like surfaces of general
type in $\RR^4_2$ are locally isometric to each other, so that the curvature of the normal connection $\varkappa$
is also preserved?} This question is considered in \cite{S-2}, in the case $K^2-\varkappa^2>0$, which means in the case 
of surfaces of the first and the second type, according to Definition \ref{Min_Surf_kind123-def-tl}\,. Now, using 
the theorems of Bonnet type for surfaces of all three types, we can prove the corresponding theorems for all minimal 
time-like surfaces of general type. 

 In order to simplify the formulations we introduce the following notion:
\begin{dfn}\label{dfh-Str_Isom_R42-tl}
An isometry between two minimal time-like surfaces in $\RR^4_2$ is said to be a \textbf{strong isometry}, 
if it preserves the curvature of the normal connection $\varkappa$.
\end{dfn}

\medskip

 If $\M$ is minimal time-like surface in $\RR^4_2$, consider the natural isometry between $\M$ and $\M_\theta$, 
defined in Proposition \ref{Isom_M_theta-M-tl}. It follows from formula \eqref{K_kappa-1-param_family_R42-tl} 
that this isometry is a strong isometry. Combining formulas \eqref{hat_K_kappa-K_kappa-antimov_R42-tl} and
\eqref{hat_K_kappa-K_kappa-conj_R42-tl} it follows that the natural isometry between $\M$ and $\bar\M^a$ is 
also a strong isometry. Consequently, the whole one-parameter family of surfaces $\bar\M^a_\theta$ are strongly 
isometric to the given surface $\M$. It turns out for the minimal time-like surfaces of general type in $\RR^4_2$, 
that these examples are all possible. In order to prove that, we need the following
\begin{lem}\label{f_prime-const_DD}
Let $f$ be a holomorphic $\DD$-valued function, defined in a connected domain of $\DD$, such that $|f'(t)|^2=1$\,.
Then $f$ has the form $f(t)=\pm\,e^{\jj\theta}t+c$\,, where $\theta\in\RR$ and $c\in\DD$ are constants. 
\end{lem}
\begin{proof}
 Differentiating the equality $f'\bar f' = |f'|^2=1$\,, we find 
$\frac{\partial f'}{\partial t} \bar f' + f' \frac{\partial \bar f'}{\partial t} = 0$\,.
Since $f$ is holomorphic, then $\frac{\partial \bar f'}{\partial t} = 0$\,, which gives
$\frac{\partial f'}{\partial t} \bar f' = 0$\,. It follows from $|\bar f'|^2=|f'|^2=1$ that $\bar f'(t)$ is 
an invertible element of $\DD$ for every $t$. Thus, it follows from $\frac{\partial f'}{\partial t} \bar f' = 0$ 
that $f'' = \frac{\partial f'}{\partial t} = 0$\,. From here and from $|f'|^2=1$ we get $f'(t)=\pm\,e^{\jj\theta}$, 
where $\theta\in\RR$. The last is equivalent to $f(t)=\pm\,e^{\jj\theta}t+c$\,,\; $c= \mathrm{ const}\in\DD$.
\end{proof}

The next theorem answers the question above.
\begin{theorem}\label{Thm-Str_Isom_Surf_R42-tl}
Let $\M$ be a minimal time-like surface of general type in $\RR^4_2$. Then all surfaces of the one-parameter 
families of minimal time-like surfaces $\M_\theta$ and $\bar\M^a_\theta$ associated to $\M$ and $\bar\M^a$,
respectively, are strongly isometric to $\M$. 

 Conversely, let $\M$ and $\hat\M$ be minimal time-like surfaces of general type in $\RR^4_2$ and let $p_0\in\M$, 
$\hat p_0\in\hat\M$ be fixed points in them. Suppose that there exist such neighborhoods of $p_0$ and $\hat p_0$
in $\M$ and $\hat\M$, respectively, which are strongly isometric to each other. Further, suppose that $p_0$ 
and $\hat p_0$ are corresponding points under this isometry. Then there exists a neighborhood of $\hat p_0$ 
in $\hat\M$, which is obtained through a proper motion from a neighborhood of $p_\theta$
on the surface $\M_\theta$ or $\bar\M^a_\theta$, where $\M_\theta$ and $\bar\M^a_\theta$ are the corresponding
one-parameter families of minimal time-like surfaces associated to $\M$ and $\bar\M^a$; \:$p_\theta$ being 
the corresponding to $p_0$ point under the natural isometry between $\M$ and $\M_\theta$ or $\bar\M^a_\theta$. 
\end{theorem}
\begin{proof}
We have already mentioned, that the assertion in the right direction is true. It remains to prove the inverse assertion.

 Denote by $f:\hat\M \to \M$ the given strong isometry. Let $t\in\DD$ give isothermal coordinates in a neighborhood of $p_0$
on $\M$, and $s\in\DD$ determine isothermal coordinates in a neighborhood of $\hat p_0$ on $\hat\M$. Furthermore, we
suppose that the points $p_0$ and $\hat p_0$ have zero coordinates. Since $f$ is an isometry, then it follows that 
the function $t=f(s)$ preserves the orthogonality in $\RR^2_1$. Hence, the function $f(s)$ is either holomorphic (over $\DD$), 
or anti-holomorphic. Suppose $f(s)$ is holomorphic.  
 
 The condition that $f$ is an isometry means that the first fundamental forms of the surfaces are related as follows:
$\mathbf{\hat I}=f^*\mathbf{I}$. Using the representation \eqref{Idt-tl}, the last relation gives: 
\[
\mathbf{\hat I}=\hat E(s) |ds|^2=f^*E(t)|dt|^2=E(f(s))|df(s)|^2=E(f(s))|f'(s)|^2 |ds|^2.
\]
Therefore the coefficients of the first fundamental forms satisfy the condition: 
\begin{equation}\label{hat_E-E-Isom_R42-tl}
\hat E(s) = E(f(s))|f'(s)|^2.
\end{equation}

 On the other hand, under the conditions of the theorem we have:
\begin{equation}\label{hat_K-kappa-Str_Isom_R42-tl}
\hat K(s) = K(f(s))\,; \qquad \hat\varkappa(s) = \varkappa(f(s))\,.
\end{equation}

 First we suppose that $\M$ is of the first or the second type according to Definition \ref{Min_Surf_kind123-def-tl}\,.
Then Theorem \ref{DegP_kind123-K_kappa-tl} gives that $\hat\M$ is also of the first or the second type. So far 
we used arbitrary isothermal coordinates. Since $\M$ and $\hat\M$ are minimal surfaces of general type, further, we suppose 
that $t$ and $s$ are canonical coordinates on $\M$ and $\hat\M$, respectively. If necessary, according to Theorem
\ref{Can_Coord-uniq_R42-tl}, we can replace the coordinates $s$ on $\hat\M$ with $\bar s$, which will ensure us 
the function $f(s)$ is holomorphic. Taking into account \eqref{E_phi_K2_kappa2-kind12_R42-tl} and 
\eqref{hat_K-kappa-Str_Isom_R42-tl}, we get the equality:
\begin{equation}\label{hat_E-E_Can-Str_Isom_R42-tl}
\hat E(s) = E(f(s))\,.
\end{equation}

 Comparing equalities \eqref{hat_E-E-Isom_R42-tl} and \eqref{hat_E-E_Can-Str_Isom_R42-tl}, we find $|f'(s)|^2=1$\,. 
Applying Lemma \ref{f_prime-const_DD}\, to $f$, we obtain $f(s)=\pm\,e^{-\jj\frac{\theta}{2}}s+c$, where 
$\theta\in\RR$ and $c \in \DD$ are constants. Under the given conditions we have $f(0)=0$\,. Consequently the function
$f$ has the form $f(s)=\pm\,e^{-\jj\frac{\theta}{2}}s$.
We can always choose the canonical coordinates $s$, so that we have $f(s)=\e^{-\jj\frac{\theta}{2}}s$.  
Replacing the last formula in the equalities \eqref{hat_K-kappa-Str_Isom_R42-tl}, we obtain the following: 
\begin{equation}\label{hat_K-kappa_theta-Str_Isom_R42-tl}
\hat K(s) = K(\e^{-\jj\frac{\theta}{2}}s)\,; \qquad  \hat\varkappa(s) = \varkappa(\e^{-\jj\frac{\theta}{2}}s)\,.
\end{equation}

 The last expressions for $\hat K(s)$ and $\hat\varkappa(s)$ coincide with the expressions in formula 
\eqref{K_kappa_1-param_family_can_R42-tl} for the curvatures $K$ and $\varkappa$ of the corresponding
surfaces $\M_\theta$ in the one-parameter family of minimal surfaces associated to $\M$. Thus we have:
\begin{equation}\label{hat_K-kappa_1-param-Str_Isom_R42-tl} 
\hat K(s)=K_\theta(s)\,; \qquad \hat\varkappa(s)=\varkappa_\theta(s)\,.
\end{equation}
where $s$ gives canonical coordinates on both surfaces $\hat\M$ and $\M_\theta$.

 First we suppose that $\M$ and $\hat\M$ are of the same type. Then, according to Theorem \ref{Can_1-param_family-tl}\,, 
$\M_\theta$ is of the same type as $\M$ and respectively $\hat\M$. Equalities \eqref{hat_K-kappa_1-param-Str_Isom_R42-tl} 
mean that we can apply to $\hat\M$ and $\M_\theta$ the uniqueness part of 
Theorem \ref{Bone_Phi_bar_Phi_m1_m2_K_kappa_kind1_R42-tl} or Theorem \ref{Bone_Phi_bar_Phi_m1_m2_K_kappa_kind2_R42-tl}\,.
Therefore, there exists a neighborhood of $\hat p_0$ in $\hat\M$, 
which is obtained through a proper motion from a neighborhood of the corresponding point $p_\theta\in\M_\theta$. 
This proves the assertion in the case under consideration.

Now, suppose that $\M$ and $\hat\M$ are of different type. Then, it follows from Theorem \ref{Can_antimov_conj_R42-tl} 
that $\bar\M^a$ is of the same type as $\hat\M$. The same theorem implies that $\bar\M^a$ is locally strongly 
isometric to $\M$, and consequently, also to $\hat\M$. Therefore, $\bar\M^a$ and $\hat\M$ satisfy the conditions
of the case, considered previously. Hence, there exists a neighborhood of $\hat p_0$ in $\hat\M$, which is obtained
through a proper motion from a neighborhood of the corresponding point $p_\theta\in\bar\M^a_\theta$. This proves
the assertion in this case as well.

\smallskip

 Further, we consider the case, when $\M$ is of the third type according to Definition~\ref{Min_Surf_kind123-def-tl}\,.
Since $K$ and  $\varkappa$ are invariant under a strong isometry, then it follows from Theorem 
\ref{DegP_kind123-K_kappa-tl} that $\hat\M$ is also of the third type. According to Theorem \ref{Can_Coord-uniq_R42-tl},
we cannot replace the coordinates $s$ on $\hat\M$ with $\bar s$ and therefore we have to consider separately the 
cases of a holomorphic and an anti-holomorphic function $f(s)$.

 If $f(s)$ is holomorphic, then we have again \eqref{hat_E-E_Can-Str_Isom_R42-tl}, \eqref{hat_K-kappa_theta-Str_Isom_R42-tl} 
and \eqref{hat_K-kappa_1-param-Str_Isom_R42-tl}. Analogously to the case when the surfaces are of the first or 
the second type, we conclude that there exists a neighborhood of $\hat p_0$ in $\hat\M$, which is obtained through
a proper motion from a neighborhood of the corresponding point $p_\theta\in\M_\theta$.

 Finally, it remained the case of an anti-holomorphic function $f(s)$. Denote by $g$ the natural strong isometry
from $\M$ to $\bar\M^a$. Since the surfaces are of the third type, parametrized by canonical coordinates, this 
isometry is given by a formula of the type $r=\bar t$, according to Theorem \ref{Can_antimov_conj_R42-tl}\,, 
where $r$ are canonical coordinates on $\bar\M^a$. Then the composition of $f$ and $g$ gives a local isometry
from $\hat\M$ to $\bar\M^a$, which in canonical coordinates is given by the function $r=\bar f(s)$. 
The last function is holomorphic. Therefore, $\bar\M^a$ and $\hat\M$ satisfy the conditions of the case, 
considered previously. Consequently, there exists a neighborhood of $\hat p_0$ in $\hat\M$, which
is obtained through a proper motion from a neighborhood of the corresponding point $p_\theta\in\bar\M^a_\theta$.
This proves the theorem.
\end{proof} 

\textbf{Acknowledgements}

The first author is partially supported by the National Science Fund, Ministry of Education and Science of Bulgaria 
under contract DN 12/2.

\end{document}